\documentclass[onefignum,onetabnum,final]{siamart251216}

\usepackage{lipsum}
\usepackage{amsmath,amssymb,amsfonts,latexsym,stmaryrd}
\usepackage{graphicx,float}
\usepackage{mathrsfs} 
\usepackage{color}
\usepackage{epstopdf}
\usepackage{algorithmic}
\usepackage{multirow}
\ifpdf
  \DeclareGraphicsExtensions{.eps,.pdf,.png,.jpg}
\else
  \DeclareGraphicsExtensions{.eps}
\fi
\usepackage{diagbox}

\def\ds{\displaystyle}

\def\O{\Omega}

\renewcommand\sp{\mathop{\mathrm{Sp}}\nolimits}

\usepackage{booktabs}
\usepackage{float}

\newcommand\bu{\boldsymbol{u}}
\newcommand\bv{\boldsymbol{v}}

\newcommand\bn{\boldsymbol{n}}

\newcommand\curl{\textbf{\text{curl}\,}}

\def\hdel{\widehat{\delta}}
\def\CM{\mathcal{X}}
\def\CN{\mathcal{Y}}

\newcommand\bF{\boldsymbol{f}}

\newcommand\bT{\boldsymbol{T}}
\newcommand\bP{\boldsymbol{P}}

\def\CT{{\mathcal T}}

\newcommand{\dd}{\texttt{d}}

\newcommand\bsig{\boldsymbol{\sigma}}
\newcommand\brho{\boldsymbol{\rho}}
\newcommand\btau{\boldsymbol{\tau}}
\newcommand\bPi{\boldsymbol{\Pi}}

\renewcommand\H{\mathrm{H}}
\renewcommand\L{\mathrm{L}}

\renewcommand\O{\Omega}

\newcommand\bdiv{\mathop{\mathbf{div}}\nolimits}

\renewcommand\div{\mathop{\mathrm{div}}\nolimits}
\newcommand\rot{\mathop{\mathrm{rot}}\nolimits}

\newcommand\tr{\mathop{\mathrm{tr}}\nolimits}

\renewcommand\sp{\mathop{\mathrm{sp}}\nolimits}

\newcommand{\vertiii}[1]{{\left\vert\kern-0.25ex\left\vert\kern-0.25ex\left\vert #1 
    \right\vert\kern-0.25ex\right\vert\kern-0.25ex\right\vert}}
\newcommand{\vertiH}[1]{{\left\vert\kern-0.25ex\left\vert #1 
		\right\vert\kern-0.25ex\right\vert_{\H(h)}}}

\newsiamremark{remark}{Remark}
\newsiamremark{assumption}{Assumption}
\newsiamremark{hypothesis}{Hypothesis}
\crefname{hypothesis}{Hypothesis}{Hypotheses}
\newsiamthm{claim}{Claim}

\headers{elasticity eigenvalue problem}{Arbaz Khan, Felipe Lepe and Jesus Vellojin}

\title{A locking free mixed FEM based on a  pure pseudostress based formulation for the elasticity eigenproblem \thanks{Submitted to the editors DATE.
\funding{Arbaz Khan was  partially
	supported by ANRF ARG MATRICS grant	ANRF/ARGM/2025/001949/MTR. Felipe Lepe was partially supported by  DICREA through Proyecto Regular   RE2514703, Universidad del B\'io-B\'io. Jesus Vellojin was partially supported by the National Agency for Research and Development, ANID-Chile through FONDECYT Postdoctorado project 3230302.
}}}

\author{Arbaz Khan\thanks{Department of Mathematics, Indian Institute of Technology Roorkee, Roorkee 247667, India. \email{arbaz@ma.iitr.ac.in}}
\and Felipe Lepe\thanks{GIMNAP-Departamento de Matem\'atica, Universidad del B\'io - B\'io, Casilla 5-C, Concepci\'on, Chile. \email{flepe@ubiobio.cl}.}
\and Jesus Vellojin\thanks{Departamento de Ciencias, Universidad Técnica Federico Santa María, Av. Federico Sta. María 6090, Viña del Mar, Chile. \email{jesus.vellojinm@usm.cl}.}}

\usepackage{amsopn}

\ifpdf
\hypersetup{
  pdftitle={DG for the nearly incompressible elasticity eigenvalue problem},
  pdfauthor={Arbaz Khan, Felipe Lepe and Jesus Vellojin}
}
\fi

\newcommand\vdiv{\mathop{\mathrm{div}}\nolimits}
\newcommand\Q{\mathrm{Q}}
\renewcommand\H{\mathrm{H}}

\newcommand\bz{\boldsymbol{z}}
\def\CE{{\mathcal E}}

\newcommand\err{\texttt{err}}
\newcommand\eff{\texttt{eff}}

\newcommand{\cred}[1]{\textcolor{black}{#1}}
\newcommand{\cblue}[1]{\textcolor{black}{#1}}

\newcommand{\jumpp}[1]{\llbracket #1 \rrbracket}

\begin{document}

\maketitle

\begin{abstract}
\cblue{We analyze a novel locking-free mixed formulation for the elasticity eigenvalue problem in both two and three dimensions, expressed exclusively in terms of the pseudostress tensor. An important feature of this formulation is that it does not require the enforcement of symmetry, either in a weak or strong sense.}
The displacement of the structure is recovered via a postprocess of the computed pseudostress. We introduce a mixed finite element method based in the tensorial version of the standard families of finite elements to discretize the space $\boldsymbol{\mathcal{H}}(\bdiv)$. We prove convergence and a priori error estimates under the theory of non-compact operators. Additionally, we perform an a posteriori error analysis for the problem, proving reliability and efficiency of the proposed indicator. We validate our theoretical results with numerical tests on different geometrical and physical configurations.
\end{abstract}



\section{Introduction}\label{sec:intro}

\cblue{The accurate description of the deformation of elastic structures is crucial for the design of different devices, vehicles, buildings, etc. since their stability  depends on the material properties, environments, interaction with other structures or fluids, just for mention some of the most relevant. These variables (among others of course) demand the use of reliable and robust numerical schemes for the analysis and prediction of  the behavior of structures under certain conditions. On this context, we focus our attention on the linear elasticity equations, which are the most common and simple set of partial differential equations that describe the displacement of an elastic structure. On this context, we focus our attention on the description of the displacement for the eigenvalue problem associated to the elasticity equations, since this problem is related to vibration of structures which are important on the applications that we mention at the beginning of this article.}

\cblue{The primary goal of the elasticity equations is to describe the displacement of a structure. However, other quantities of relevance may be important to also consider, such as the stress, rotations, bending moments, etc. This motivates the design of mixed formulation where additional unknowns on the PDE systems allows to describe in a more complete manner the response of a structure under certain conditions.  This demands to design and analyze convenient numerical schemes capable to approximate all the unknowns involved on the mixed formulations, which are commonly called mixed methods. On this subject, the literature is abundant and particularly for the elasticity eigenvalue  system we mention 
\cite{MR4279087, MR4570534, MR3962898, MR4396855, MR3376135, MR4542511,MR3036997}. These mixed formulations have the capability of avoiding the numerical locking phenomenon that appears when the Poisson ratio is close (or even equal) to $1/2$. Moreover, mixed formulations allow to prove that the spectrum of the eigenvalue problem on its nearly incompressible regime converge to the spectrum of the Stokes eigenvalue system as is proved in, for instance \cite{MR4570534,MR3962898}, which coincided with the perfectly incompressible elasticity eigenvalue problem.}

\cblue{In the present work we propose a mixed formulation of the linear elasticity eigenvalue problem inspired by the recent work \cite{MR4542511}, where a DG method for a formulation depending only on the
stress is proposed. This formulation is similar to the one studied in \cite{MMR3} for the Stokes eigenvalue problem, where only the pseudostress tensor is the unknown of the problem and the rest of the relevant unknowns as the velocity and pressure are recovered with a simple postprocessing of the pseudostress. For elasticity this is also possible and our intention is to introduce a pure pseudostress formulation of the elasticity eigensystem. The continuous problem is a formulation as the ones in  \cite{MR4542511,MMR3} where the regularity results for the source and eigenvalue problems are well established. The mixed formulation of our paper also allows to analyze the behavior of the elasticity spectrum when the Poisson ratio converges to $1/2$ where precisely the Lam\'e coefficient $\lambda$ (cf. \eqref{eq:lame}) blows up. We prove precisely that the spectrum of the nearly incompressible elasticity problem converge to the perfectly incompressible case which coincides with the Stokes spectrum (see \cite{MR4570534,MR3962898}).}  
\cblue{An important  difference with a formulation like \cite{MR4570534} yields in the fact that the solution operator is non-compact and hence, the theory of \cite{MR483400,MR483401} must be applied for the analysis of convergence and error estimates. For the discrete analysis, we focus on the classic Raviart--Thomas $\boldsymbol{\mathcal{H}}(\bdiv)$-conforming finite elements. This is a natural choice for the pure pseudostress formulation that we present, capable to approximate  accurately the spectrum of the elasticity eigensystem on the nearly and perfectly incompressible regimes, with no presence of spurious eigenvalues. }

\cblue{Additionally to the convergence and a priori error estimates, we propose a residual type of a posteriori error estimator. This estimator, which results to be efficient and reliable, is capable to recover the optimal order of convergence when the nearly incompressible elasticity eigenvalue problem is considered and, with minor modifications on the definition of the indicator, works also con the perfectly incompressible case. This is a completely novel type of a posteriori  error estimator  for eigenvalue problems  which emerges due to the pure stress formulation that we propose. }
\cblue{\subsection{Outline} The paper is organized as follows: In Section \ref{sec:model-problem} we introduce the eigenvalue  model problem. Here, we present the pure pseudostress formulation and the analysis of the source and eigenvalue problems. We introduce the solution operator, spectral characterization, and the relation of the spectrum of the nearly incompressible eigenvalue problem and its limit. In Section \ref{sec:mixed_method} we introduce the mixed finite element method to approximate the spectrum. We present the FEM spaces, the discrete eigenvalue problem, and the discrete solution operator. With this discrete solution operator, we analyze its convergence to the continuous operator as the mesh size goes to zero, according to the theory of non compact operators. Section \ref{sec:spec_convergence} contains the analysis of the spectral convergence and a priori error estimates for the eigenvalues and eigenfunctions, whereas Section \ref{sec:apost} introduces and analyzes the a posteriori error estimator. Finally in Section \ref{sec:numerical-experiments} we report several numerical tests to assess the performance of the proposed mixed method for both, the a priori and a posteriori error analyses.}

\section{Model problem}\label{sec:model-problem}
Let $\Omega\subset\mathbb{R}^d$ with $d\in\{2,3\}$ be an open and bounded
polygonal/polyhedral Lipschitz domain with boundary $\Gamma$.
We are interested in the elasticity eigenvalue problem: Find $\kappa\in\mathbb{R}^+$ and the pair $(\bsig,\bu)$ such that
\begin{equation}\label{def:elast_system}
\left\{
\begin{array}{rcll}
\boldsymbol{\sigma} & = & 2\mu\boldsymbol{\varepsilon}(\bu)+\lambda\tr(\boldsymbol{\varepsilon}(\bu))\mathbf{I}&  \text{ in } \quad \Omega, \\
\bdiv\bsig & = & -\kappa\varrho\bu & \text{ in } \quad \Omega, \\
\bu & = & \mathbf{0} & \text{ on } \quad \Gamma,\\
\end{array}
\right.
\end{equation}
 Here, $\bu$ represents the displacement of the elastic structure, $\bsig$ is the Cauchy  tensor, $\varrho$ is the mass density of the material which we assume as positive, $\lambda$ and $\mu$ are  the positive Lam\'e constants defined by
 \begin{equation}
 \label{eq:lame}
\lambda:=\frac{E\nu}{(1+\nu)(1-2\nu)}\quad\text{and}\quad\mu:=\frac{E}{2(1+\nu)},
\end{equation}
where $E$ is the Young's modulus and $\nu$ is the Poisson ratio. Let us remark that on our paper, we are considering ordinary materials, implying that $\nu\in [0,1/2)$ where the case $\nu=1/2$ corresponds to perfectly incompressible materials. Also, $\mathbf{I}\in\mathbb{R}^{d\times d}$ is the identity matrix, and $\boldsymbol{\varepsilon}(\bu)$ represents the tensor of small deformations,  given by $\boldsymbol{\varepsilon}(\bu):=\frac{1}{2}(\nabla\bu+(\nabla\bu)^{\texttt{t}})$, where $\texttt{t}$ is the transpose operator. 
 
 From the first equation of
\eqref{def:elast_system} we have that $\bdiv\bsig=\mu\Delta\bu+(\lambda+\mu)\nabla\div\bu$.
This allows us  to rewrite \eqref{def:elast_system} as follows
\begin{equation*}\label{def:elast_system_reduced}
\left\{
\begin{array}{rcll}
\mu\Delta\bu+(\lambda+\mu)\nabla\div\bu & = & -\kappa\varrho\bu &  \text{ in } \quad \Omega, \\
\bu & = & \mathbf{0} & \text{ on } \quad \Gamma.
\end{array}
\right.
\end{equation*}
Now we introduce the so called pseudostress tensor, defined by (see \cite{MR3453481} for instance)
\begin{equation*}
\label{eq:pseudo}
\brho:=\mu\nabla\bu+(\lambda+\mu)\div\bu\mathbf{I}=\mu\nabla\bu+(\lambda+\mu)\tr(\nabla\bu)\mathbf{I}.
\end{equation*}
Observe that $\bdiv\bsig=\bdiv\brho$. Hence, we have the following formulation where the pseudostress and the displacement are the main unknowns: Find $\kappa\in\mathbb{R}^+$ and $(\brho, \bu)$ such that
\begin{equation}\label{def:elast_system_rho_1}
\left\{
\begin{array}{rcll}
\brho & = &\mu\nabla\bu+(\lambda+\mu)\tr(\nabla\bu)\mathbf{I}&  \text{ in } \quad \Omega, \\
\bdiv\brho & = & -\kappa\varrho\bu & \text{ in } \quad \Omega, \\
\bu & = & \mathbf{0} & \text{ on } \quad \Gamma.
\end{array}
\right.
\end{equation}
Moreover, the following identity holds (see \cite[Section 2]{MR3453481} for details)
\begin{equation*}
\displaystyle\frac{1}{\mu}\left\{\brho-\frac{\lambda+\mu}{d\lambda+(d+1)\mu}\tr(\brho)\mathbf{I} \right\}=\nabla\bu,
\end{equation*}
which, replacing in \eqref{def:elast_system_rho_1}, leads to the following  eigenvalue problem
\begin{equation}\label{def:elast_system_rho_mix}
\left\{
\begin{array}{rcll}
\displaystyle\frac{1}{\mu}\left\{\brho-\frac{\lambda+\mu}{d\lambda+(d+1)\mu}\tr(\brho)\mathbf{I} \right\}& = &\nabla\bu&  \text{ in } \quad \Omega, \\
\bdiv\brho & = & -\kappa\varrho\bu & \text{ in } \quad \Omega, \\
\bu & = & \mathbf{0} & \text{ on } \quad \Gamma.
\end{array}
\right.
\end{equation}

From  the second equation on \eqref{def:elast_system_rho_mix},  we observe that when is restricted to the boundary $\Gamma$, we have $\varrho^{-1}(\bdiv\brho|_{\Gamma})=-\kappa(\bu|_{\Gamma})=0$. Hence, replacing the displacement on the first equation and using the Dirichlet boundary condition, we obtain the following pure pseudostress system of PDEs
\begin{equation}\label{def:elast_system_rho_pure}
\left\{
\begin{array}{rcll}
\kappa\displaystyle \left\{\brho-\frac{\lambda+\mu}{d\lambda+(d+1)\mu}\tr(\brho)\mathbf{I} \right\}& = &-\mu\nabla(\varrho^{-1}\bdiv\brho)&  \text{ in } \quad \Omega, \\
\varrho^{-1}\bdiv\brho  & = & \mathbf{0} & \text{ on } \quad \Gamma.
\end{array}
\right.
\end{equation}
\subsection{Variational formulation and preliminary results}
Now we derive a variational formulation for the analysis of \eqref{def:elast_system_rho_pure}. Let us define $\mathbf{X}:=\boldsymbol{\mathcal{H}}(\bdiv,\O)$.  If $\boldsymbol{n}$ denotes the unit vector on $\Gamma$, we  define the trace operators $\gamma_{\boldsymbol{n}}:\mathbf{X}\rightarrow \H^{-1/2}(\Gamma)$ and $\gamma_0:\H^1(\O)\rightarrow \L^2(\Gamma)$. Now, for $\btau\in\mathbf{X}$, we multiply the first equation of \eqref{def:elast_system_rho_pure} and integrate by parts obtaining 
\begin{multline}
\label{eq:IBP1}
\kappa\displaystyle \left\{\int_{\O}\brho:\btau-\frac{\lambda+\mu}{d\lambda+(d+1)\mu}\int_{\O}\tr(\brho) \tr(\btau)\right\} = -\mu\int_{\O}\nabla(\varrho^{-1}\bdiv\boldsymbol{\rho)}:\btau\\
=\mu\int_{\O}\varrho^{-1}\bdiv\brho\cdot\bdiv\btau+\mu\langle\gamma_{\boldsymbol{n}}(\btau),\gamma_0(\varrho^{-1}\bdiv\brho)\rangle=\mu\int_{\O}\varrho^{-1}\bdiv\brho\cdot\bdiv\btau,
\end{multline}
where $\langle\cdot,\cdot\rangle$ denotes the corresponding duality pairing. If $\btau^{\texttt{d}}$ represents the deviatoric of $\btau$ defined by $\btau^{\texttt{d}}:=\btau-d^{-1}\tr(\btau)\mathbf{I}$, \eqref{eq:IBP1} can be written alternatively as follows
\begin{equation}
\label{eq:no_shift}
\kappa\displaystyle \left\{\int_{\O}\brho^{\texttt{d}}:\btau^{\texttt{d}}+\frac{\mu}{d(d\lambda+(d+1)\mu)}\int_{\O}\tr(\brho) \tr(\btau)\right\}=\mu\int_{\O}\varrho^{-1}\bdiv\brho\cdot\bdiv\btau\quad\forall\btau\in\mathbf{X}.
\end{equation}
As a first remark, we observe from \eqref{eq:no_shift} that setting $\btau=\brho$ and taking into account that $\brho\neq\boldsymbol{0}$ and the coefficients are positive,  we have
\begin{equation*}
\kappa=\frac{\mu\varrho^{-1}d(d\lambda +(d+1)\mu)\|\bdiv\brho\|_{0,\O}^2}{d(d\lambda +(d+1)\mu)\|\brho^{\texttt{d}}\|_{0,\O}^2+\mu\|\tr(\brho)\|_{0,\O}^2}>0,
\end{equation*}
which guarantees that the eigenvalues of the problem are positive.

On the other hand, since the right hand of \eqref{eq:no_shift} is not coercive, it is necessary to use a shift argument. Hence, adding in both sides of  \eqref{eq:no_shift} the terms on the left hand side, we obtain 
\begin{multline}
\label{eq:with_shift}
(\kappa+1)\displaystyle \left\{\int_{\O}\brho^{\texttt{d}}:\btau^{\texttt{d}}+\frac{\mu}{d(d\lambda+(d+1)\mu)}\int_{\O}\tr(\brho) \tr(\btau)\right\}\\=\mu\int_{\O}\varrho^{-1}\bdiv\brho\cdot\bdiv\btau+\int_{\O}\brho^{\texttt{d}}:\btau^{\texttt{d}}+\frac{\mu}{d(d\lambda+(d+1)\mu)}\int_{\O}\tr(\brho) \tr(\btau)\quad\forall\btau\in\mathbf{X}.
\end{multline}

Let us introduce the bilinear forms $a:\mathbf{X}\times\mathbf{X}\rightarrow\mathbb{R}$ and $b:\mathbf{X}\times\mathbf{X}\rightarrow\mathbb{R}$, which are  defined for all $\boldsymbol{\chi},\btau\in\mathbf{X}$ by 
\begin{equation*}
a(\boldsymbol{\chi},\btau):=\mu\int_{\O}\varrho^{-1}\bdiv\boldsymbol{\chi}\cdot\bdiv\btau+\int_{\O}\boldsymbol{\chi}^{\texttt{d}}:\btau^{\texttt{d}}+\frac{\mu}{d(d\lambda+(d+1)\mu)}\int_{\O}\tr(\boldsymbol{\chi}) \tr(\btau),
\end{equation*}
and 
\begin{equation*}
b(\boldsymbol{\chi},\btau):=\int_{\O}\boldsymbol{\chi}^{\texttt{d}}:\btau^{\texttt{d}}+\frac{\mu}{d(d\lambda+(d+1)\mu)}\int_{\O}\tr(\boldsymbol{\chi}) \tr(\btau).
\end{equation*}

With these bilinear forms at hand, \eqref{eq:with_shift} is stated as follows: find $\kappa\in\mathbb{R}^+$ and $\boldsymbol{0}\neq\brho\in\mathbf{X}$ such that
\begin{equation}
\label{eq:spectral_continuous}
a(\brho,\btau)=(\kappa+1) b(\brho,\btau)\quad\forall\btau\in\mathbf{X}.
\end{equation}

It is easy to check that $a(\cdot,\cdot)$ and $b(\cdot,\cdot)$ are continuous. In fact, for any $\brho,\btau\in\mathbf{X}$, we have 
\begin{equation*}
|a(\brho,\btau)|\leq \max\left\{\mu\varrho^{-1},\left(\frac{d+\sqrt{d}}{d}\right)^2+\frac{d}{(d+1)\mu}\right\}\|\brho\|_{\bdiv,\O}\|\btau\|_{\bdiv,\O},
\end{equation*}
and 
\begin{equation*}
|b(\brho,\btau)|\leq \left\{\left(\frac{d+\sqrt{d}}{d}\right)^2 +\frac{d}{(d+1)\mu} \right\}\|\brho\|_{\bdiv,\O}\|\btau\|_{\bdiv,\O}.
\end{equation*}
It is important to remark that the continuity constant for each bilinear form is independent of $\lambda$. 

Following the approach of \cite{MR3453481,MR4570534}, for the analysis we need to decompose properly the space $\mathbf{X}$. This decomposition consists, first,  into consider the following subspace
\begin{equation*}
\mathbf{X}_0=\left\{\btau\in\mathbf{X}\,: \,\int_{\O}\tr(\btau)=0\,\,\,\text{in}\,\,\O \right\},
\end{equation*}
and then, write $\mathbf{X}:=\mathbf{X}_0\oplus\mathbb{R}\mathbf{I}$. This decomposition states that there exists $\btau_0\in\mathbf{X}_0$ and a constant $m$ such that for $\btau\in\mathbf{X}$ there holds $\btau=\btau_0+w\mathbf{I}$, where the constant $w$ is defined by $w:=(d|\O|)^{-1}(\tr(\btau),1)_{0,\O}$. With this space $\mathbf{X}_0$ at hand, eigenvalue problem \eqref{eq:spectral_continuous} can be written as follows:  find $\kappa\in\mathbb{R}^+$ and $\boldsymbol{0}\neq\brho\in\mathbf{X}_0$ such that
\begin{equation}
\label{eq:spectral_continuous_H0}
a(\brho,\btau)=(\kappa+1) b(\brho,\btau)\quad\forall\btau\in\mathbf{X}_0.
\end{equation}

Let us invoke the following technical result available in  \cite[Chapter 9, Proposition 9.1.1]{MR3097958}).
\begin{lemma}\label{lmm:cota}
There exists a constant $c_1>0$, depending on $\O$, such that
\begin{equation*}
c_1\|\btau\|_{0,\O}^2\leq \|\btau^{\dd}\|_{0,\O}^2+\|\bdiv\btau\|_{0,\O}^2\qquad\forall \btau\in \mathbf{X}_0.
\end{equation*}
\end{lemma}
 
 Now, for $\btau\in\mathbf{X}_0$, the fact that the Lam\'e constants are positive, and Lemma \ref{lmm:cota} at hand, we have
 \begin{multline}
 \label{eq:coercive_a}
 a(\btau,\btau)=\mu\varrho^{-1}\|\bdiv\btau\|_{0,\O}^2+\|\btau^{\texttt{d}}\|_{0,\O}^2+\frac{\mu}{d(d\lambda+(d+1)\mu)}\|\tr(\btau)\|_{0,\O}^2\\
 \geq \min\{\mu\varrho^{-1},1\}(\|\bdiv\btau\|_{0,\O}^2+\|\btau^{\texttt{d}}\|_{0,\O}^2)\\
 \geq \frac{1}{2}\min\{\mu\varrho^{-1},1\}\min\{c_1,1\}\|\btau\|_{\bdiv,\O}^2,
 \end{multline}
 proving the coercivity of $a(\cdot,\cdot)$ in $\mathbf{X}_0$.
 
 As usual in the analysis of eigenvalue problems, we introduce the linear solution operator $\bT_{\lambda}$, which we define by 
 \begin{equation*}
 \bT_{\lambda}:\mathbf{X}_0\rightarrow \mathbf{X}_0,\qquad \boldsymbol{f}\mapsto \bT_{\lambda}\boldsymbol{f}:=\widehat{\brho},
 \end{equation*}
 where $\widehat{\brho}\in\mathbf{X}_0$ is the solution of the following source problem: given $\boldsymbol{f}\in\mathbf{X}_0$, find $\widehat{\brho}\in\mathbf{X}_0$ such that 
 \begin{equation}
 \label{eq:source_cont}
 a(\widehat{\brho},\btau)=b(\boldsymbol{f},\btau)\quad\forall\btau\in\mathbf{X}_0.
 \end{equation}
 From the coercivity of $a(\cdot,\cdot)$ in $\mathbf{X}_0$ given by \eqref{eq:coercive_a} and Lax-Milgram's lemma, we conclude that problem \eqref{eq:source_cont} admits a unique solution that satisfies
 \begin{equation}
 \label{eq:dependency}
\|\bT_{\lambda}\boldsymbol{f}\|_{\bdiv,\O }= \|\widehat{\brho}\|_{\bdiv,\O}\leq C\|\boldsymbol{f}\|_{0,\O},
 \end{equation}
 where $C>0$ is independent of $\lambda$. This implies that $\bT_{\lambda}$ is well defined. Moreover it is easy to check that $(\kappa,\boldsymbol{\rho})\in \mathbb{R}^+\times\mathbf{X}_0$ solves problem \eqref{eq:spectral_continuous_H0} if and only if $(1/(1+\kappa),\boldsymbol{\rho})$ is an eigenpair of $\bT_{\lambda}$ with a nonvanishing eigenvalue. This means that $\bT_{\lambda}\boldsymbol{\rho}=\xi\boldsymbol{\rho}$ with $\xi:=\frac{1}{\kappa+1}\neq 0$ and $\boldsymbol{\rho}\neq\boldsymbol{0}$.

 Since bilinear forms $a(\cdot,\cdot)$ and $b(\cdot,\cdot)$ are symmetric, it is easy to prove that $\bT_{\lambda}$ is selfadjoint with respect to $a(\cdot,\cdot)$. Indeed, if $\boldsymbol{f},\boldsymbol{g}\in\mathbf{X}_0$, then
 \begin{equation*}
 a(\bT_{\lambda}\boldsymbol{f},\boldsymbol{g})=b(\boldsymbol{f},\boldsymbol{g})=b(\boldsymbol{g},\boldsymbol{f})=a(\bT_{\lambda}\boldsymbol{g},\boldsymbol{f})=a(\boldsymbol{f},\bT_{\lambda}\boldsymbol{g}).
 \end{equation*}

 An important fact of $\bT_{\lambda}$ is that, according to its definition, is non-compact, since $\mathbf{X}_0$ is not compactly embedded onto $\mathbf{L}^2(\O)$. This is important since it gives the highlights that our analysis of spectral convergence and a priori error estimates now must be performed under the framework of the theory of \cite{MR483400,MR483401}. However, additional remarks on $\bT_{\lambda}$ are possible to observe.
 
 Let us notice that  $\bT_{\lambda}$ satisfies the following property: let us define the space 
 \begin{equation}
 \label{eq:spaceK}
 \mathbf{K}:=\{\btau\in\mathbf{X}_0\,\,:\,\,\bdiv\btau=\boldsymbol{0}\,\,\text{in}\,\O\}.
 \end{equation}
 We observe that if $\bT_{\lambda}$ is restricted to this space in the following sense $\bT_{\lambda}|_{\mathbf{K}}:\mathbf{K}\rightarrow\mathbf{K}$, it  reduces to the identity and hence, $\xi=1$ is an eigenvalue of $\bT_{\lambda}$ where $\boldsymbol{\rho}$ is an  associated eigenfunction if an only if
 \begin{equation*}
 \int_{\O}\bdiv\boldsymbol{\rho}\cdot\bdiv\boldsymbol{\tau}=0\quad\forall\btau\in\mathbf{X}_0,
 \end{equation*}
 implying that $\xi=1$ is an eigenvalue of $\bT_{\lambda}$ with associated eigenspace $\mathbf{K}$. 
 
\cblue{Now the aim is to decompose $\mathbf{X}_0$ as a direct sum $\mathbf{X}_0=\mathbf{K}\otimes\mathbf{K}^{\bot}$ and prove that $\mathbf{K}^{\bot}$ is a more regular space. To do this task, we follow the arguments on \cite[Section 3]{MMR3}.
Let us introduce the following operator $\bP:\mathbf{X}_0\rightarrow\mathbf{X}_0$, where if $\boldsymbol{\rho}\in\mathbf{X}_0$, then $\bP\boldsymbol{\rho}:=\widetilde{\boldsymbol{\rho}}$. Here, $(\widetilde{\boldsymbol{\rho}},\widetilde{\bu})\in\mathbf{X}_0\times\L^2(\O)^d$ is the solution of the following mixed problem with source $\bdiv\brho$:
{ \footnotesize
 \begin{equation}\label{def:aux_mixed_problem}
\left\{
\begin{array}{rcll}
\displaystyle\int_{\O}\widetilde{\boldsymbol{\rho}}^{\texttt{d}}:\btau^{\texttt{d}}+\frac{\mu}{d(d\lambda+(d+1)\mu)}\int_{\O}\tr(\widetilde{\boldsymbol{\rho}})\tr(\btau)+\int_{\O}\widetilde{\bu}\cdot\bdiv\btau& = &0\quad\forall\btau\in\mathbf{X}_0, \\
\displaystyle\int_{\O}\bv\cdot\bdiv\widetilde{\boldsymbol{\rho}} & = & \displaystyle\int_{\O}\bv\cdot\bdiv\boldsymbol{\rho}\quad\forall\bv\in\L^2(\O)^d. 
\end{array}
\right.
\end{equation}}
If we define the bilinear forms $\widetilde{a}:\mathbf{X}_0\times\mathbf{X}_0\rightarrow\mathbb{R}$ and $b:\mathbf{X}_0\times\L^2(\O)^d\rightarrow\mathbb{R}$ by 
$$\widetilde{a}(\boldsymbol{\rho},\btau):=\displaystyle\int_{\O}\boldsymbol{\rho}^{\texttt{d}}:\btau^{\texttt{d}}+\frac{\mu}{d(d\lambda+(d+1)\mu)}\int_{\O}\tr(\boldsymbol{\rho})\tr(\btau),\quad\forall\boldsymbol{\rho},\btau\in\mathbf{X}_0,$$
and
$b(\btau,\bv):=\displaystyle\int_{\O}\bv\cdot\btau$ for all $\btau\in\mathbf{X}_0$ and for all $\bv\in\L^2(\O)^d$, we have that problem \eqref{def:aux_mixed_problem} is well posed since $\widetilde{a}(\cdot,\cdot)$ is coercive in the kernel of $b(\cdot,\cdot)$ and the following inf-sup condition holds: there exists a constant $\beta>0$ such that 
\begin{equation*}
\displaystyle\sup_{\boldsymbol{0}\neq\btau\in\mathbf{X}_0}\frac{b(\btau,\bv)}{\|\btau\|_{\bdiv,\O}}\geq\beta\|\bv\|_{0,\O}\quad\forall\bv\in\L^2(\O)^d.
\end{equation*}
  We notice that \eqref{def:aux_mixed_problem} is the dual mixed formulation of the following problem:
 \begin{equation}\label{def:elast_system_rho_mixed_dual}
\left\{
\begin{array}{rcll}
\widetilde{\brho}^{\texttt{d}}+\displaystyle\frac{\mu}{d(d\lambda+(d+1)\mu)}\tr(\boldsymbol{\rho})\mathbf{I} & = &\nabla\widetilde{\bu}&  \text{ in } \quad \Omega, \\
-\bdiv\widetilde{\brho}& = & -\bdiv\brho& \text{ in } \quad \Omega, \\
\widetilde{\bu} & = & \mathbf{0} & \text{ on } \quad \Gamma.
\end{array}
\right.
\end{equation}
It is not difficult to check that the pair $(\widetilde{\brho},\widetilde{\bu})\in\boldsymbol{\mathcal{H}}(\bdiv,\O)\times \H_0^1(\O)^d$ is a solution of \eqref{def:elast_system_rho_mixed_dual} if and only if $(\widetilde{\brho},\widetilde{\bu})\in\mathbf{X}_0\times \L^2(\O)^d$ is solution of \eqref{def:aux_mixed_problem}.
 Now, according to  \cite[Section 2]{MR3962898}, we have that  for all $s\in (0,\widehat{s})$, where $\widehat{s}\in (0,1]$, the displacement $\widetilde{\bu}$ in \eqref{def:elast_system_rho_mixed_dual} is such that   $\widetilde{\bu}\in\H^{1+s}(\Omega)^d$. Also, there exists a constant $C>0$ such that
\begin{equation}
\label{eq:reg_s}
\|\widetilde{\bu}\|_{1+s,\O}\leq C\|\bdiv\brho\|_{0,\O},
\end{equation}
and as a consequence $\bP(\mathbf{X}_0)\subset\boldsymbol{\mathcal{H}}^s(\O)$. It is easy to check that $\bP$ is idempotent and consequently a projection. This is key, since allows to have the following decomposition $\mathbf{X}_0=\mathbf{K}\oplus\bP(\mathbf{X_0})$. Additionally, since $\bT_{\lambda}$ is selfadjoint, we conclude that $\bP(\mathbf{X}_0)$ is invariant for $\bT_{\lambda}$ (see \cite[Lemma 3.3]{MMR3}).}

\cblue{Let us denote by  $\boldsymbol{\mathcal{D}}(\O)$ the space of sufficiently smooth tensorial functions. If we set $\btau\in\boldsymbol{\mathcal{D}}(\O)\subset\mathbf{X}_0$ in \eqref{eq:source_cont}, we obtain
\begin{equation*}
\mu\varrho^{-1}\nabla(\bdiv\widehat{\brho})+\widehat{\brho}^{\texttt{d}}+\frac{\mu}{d(d\lambda+(d+1)\mu)}\tr(\widehat{\brho})\mathbf{I}=\boldsymbol{f}^{\texttt{d}}+\frac{\mu}{d(d\lambda+(d+1)\mu}\tr(\boldsymbol{f})\mathbf{I},
\end{equation*}  
which reveals that $\nabla(\bdiv\widehat{\brho})\in\L^2(\O)$ and hence $\bdiv\widehat{\brho}\in\H^1(\O)^d$. On the other hand, we observe that  $\bT_{\lambda}(\bP(\mathbf{X}_0))\subset\{\btau\in\boldsymbol{\mathcal{H}}^s(\O)\,:\,\bdiv\btau\in\H^1(\O)^d\}$ as a consequence of the inclusion $\bP(\mathbf{X}_0)\subset\boldsymbol{\mathcal{H}}^s(\O)$ and the fact that $\bP(\mathbf{X})$ is invariant for $\bT_{\lambda}$. Moreover, 
$\bT_{\lambda}|_{\bP(\mathbf{X}_0)}:\bP(\mathbf{X}_0)\rightarrow\bP(\mathbf{X}_0)$ is compact, due to the compact embedding $\{\widehat{\brho}\in\boldsymbol{\mathcal{H}}^s(\O)\,:\,\bdiv\widehat{\brho}\in\H^1(\O)^d\}\cap\mathbf{X}_0$ onto $\mathbf{X}_0$.  All these consequences are summarized in the following result.
\begin{lemma}
\label{lmm:invariancy}
Operator $\bT$ satisfies 
$
\bT_{\lambda}(\bP(\mathbf{X}_0))\subset\{\btau\in\boldsymbol{\mathcal{H}}^s(\O)\,:\,\bdiv\btau\in\H^1(\O)^d\}.
$
Moreover, there exists a constant $C>0$ such that for all $\boldsymbol{f}\in\bP(\mathbf{X}_0)$ if $\bT_{\lambda}\boldsymbol{f}:=\widehat{\brho}$, solution of \eqref{eq:source_cont}, there holds
\begin{equation*}
\|\widehat{\brho}\|_{s,\O}+\|\bdiv\widehat{\brho}\|_{1,\O}\leq C\|\boldsymbol{f}\|_{\bdiv,\O},
\end{equation*}
and hence, $\bT_{\lambda}|_{\bP(\mathbf{X}_0)}:\bP(\mathbf{X}_0)\rightarrow\bP(\mathbf{X}_0)$ is compact.
\end{lemma}}
\cblue{
 Following with the analysis, let us now define  the following space
 \begin{equation}
 \label{eq:spaceZ}
 \mathbf{Z}:=\{\btau\in\mathbf{X}_0\,\,:\,\,\btau^{\texttt{d}}=\boldsymbol{0}\,\,\text{in}\,\,\O\}.
 \end{equation}
 Then, if $\xi=0$, this implies that $(\kappa+1)^{-1}=0$. Then, from \eqref{eq:spectral_continuous_H0} we have 
 \begin{equation*}
 \int_{\O}\boldsymbol{\rho}^{\texttt{d}}:\btau^{\texttt{d}}+\frac{\mu}{d(d\lambda+(d+1)\mu)}\int_{\O}\tr(\boldsymbol{\rho})\tr(\btau)=0\quad\forall\btau\in\mathbf{X}_0,
 \end{equation*}
 implying that $\xi=0$ is an eigenvalue of $\bT_{\lambda}$ with associated eigenspace $\mathbf{Z}$. All the previous results  yield to  the following spectral characterization of $\bT$ which is obtained in the same manner as \cite[Theorem 3.5]{MMR3}.
\begin{theorem}
The spectrum of $\bT_{\lambda}$ decomposes as follows: $\sp(\bT_{\lambda})=\{0,1\}\cup\{\xi_k\}_{k\in\mathbb{N}}$, where the following properties hold:
\begin{itemize}
\item[(a)] $\xi=1$ is an infinite-multiplicity eigenvalue of $\bT_{\lambda}$ with associated eigenspace $\mathbf{K}$;
\item[(b)] $\xi=0$ is an eigenvalue of $\bT_{\lambda}$ with associated eigenspace $\mathbf{Z}$;
\item[(c)] $\{\xi_k\}_{k\in\mathbb{N}}\subset (0,1)$ is a sequence of nondefective finite-multiplicity eigenvalues of $\bT_{\lambda}$ that converge to 0 and the corresponding eigenspaces belong to $\bP(\mathbf{X}_0)$
\end{itemize}
\end{theorem}}
 
 \cblue{Now, the additional regularity for  the eigenfunctions is analogous to Lemma \ref{lmm:invariancy}.}
 \begin{lemma}
\label{lmm:add_eigen}
Let $\bu$ be an eigenfunction of $\bT_{\lambda}$ associated to an eigenvalue $\kappa$. Then, 
for all $r\in (0,\widehat{r})$, where $\widehat{r}\in (0,1]$, we have that $\bu\in\H^{1+r}(\Omega)^d$. Also, there exists a constant $C>0$ which in principle depends on $\lambda$, such that
\begin{equation*}
\|\bu\|_{1+r,\O}\leq \widehat{C}\|\bu\|_{0,\O}.
\end{equation*}
\end{lemma}

\begin{remark} \label{daniel0}
Observe that Lemma \ref{lmm:add_eigen}, in conjunction with the first equation of \eqref{def:elast_system_rho_1}, implies immediately that $\boldsymbol{\rho}\in  \boldsymbol{\mathcal{H}}^{r}(\O)$. On the other hand, for the divergence term, it is enough to consider the second equation in \eqref{def:elast_system_rho_1} to deduce that $\bdiv\boldsymbol{\rho}\in \H^{1+r}(\O)^{d}$.
\end{remark}

We mention that the dependency of the constants in the regularity exponents and boundedness on $\lambda$ is not completely evident, since in our numerical experiments (cf. Section \ref{sec:numerical-experiments}), even in the limit case ($\lambda=\infty$), our method obtains the expected convergence orders. This leads us to consider  the following assumption along our paper:
\begin{assumption}
Constants $\widehat{r}$ and $\widehat{C}$ in Lemma \ref{lmm:add_eigen} are independent of $\lambda$.
\end{assumption}

 \subsection{The limit problem}
Now, we must comment the following fact. Operator $\bT_{\lambda}$ has not a fixed spectrum, since   as the Lam\'e constant $\lambda$ changes, the spectrum is different. This is important to recall, since if $\lambda\rightarrow +\infty$ (i.e., the Poisson ratio approaches to $1/2$), standard numerical methods are locked. Our intention is to prove that the pseudostress formulation is capable to lead a numerical method that captures the spectrum on the perfectly incompressible case ($\lambda=+\infty$). This motivates, as in \cite{MR4570534,MR3962898}, to analyze the convergence of $\bT_{\lambda}$ to the limit case, now considering a pure pseudostress formulation. 

Let us consider the following eigenvalue problem: find $\kappa_{\infty}\in\mathbb{R}^+$ and $\boldsymbol{0}\neq\brho_{\infty}\in\mathbf{X}_0$ such that
\begin{equation}
\label{eq:eigen_limit}
a_{\infty}(\brho_{\infty},\btau)=(\kappa_{\infty}+1)b_{\infty}(\brho_{\infty},\btau)\quad\forall\btau\in\mathbf{X}_0,
\end{equation}
where bilinear forms  $a_{\infty}:\mathbf{X}_0\times\mathbf{X}_0\rightarrow\mathbb{R}$ and $b_{\infty}:\mathbf{X}_0\times\mathbf{X}_0\rightarrow\mathbb{R}$  are  defined, for any $\boldsymbol{\chi},\btau\in\mathbf{X}_0$,  by 
\begin{equation*}
a_{\infty}(\boldsymbol{\chi},\btau):=\mu\int_{\O}\varrho^{-1}\bdiv\boldsymbol{\chi}\cdot\bdiv\btau+\int_{\O}\boldsymbol{\chi}^{\texttt{d}}:\btau^{\texttt{d}},
\quad
\text{and}\quad 
b_{\infty}(\boldsymbol{\chi},\btau):=\int_{\O}\boldsymbol{\chi}^{\texttt{d}}:\btau^{\texttt{d}}.
\end{equation*}

Variational problem \eqref{eq:eigen_limit} is associated to the following PDE system
\begin{equation*}\label{def:elast_system_rho_pure_limit}
\left\{
\begin{array}{rcll}
\kappa_{\infty}\displaystyle\brho_{\infty}& = &-\mu\nabla(\varrho^{-1}\bdiv\brho_{\infty})&  \text{ in } \quad \Omega, \\
\varrho^{-1}\bdiv\brho_{\infty}  & = & \mathbf{0} & \text{ on } \quad \Gamma,
\end{array}
\right.
\end{equation*}
which is obtained from 
\begin{equation*}\label{def:elast_system_rho_1_limit}
\left\{
\begin{array}{rcll}
\brho_{\infty} & = &\mu\nabla\bu&  \text{ in } \quad \Omega, \\
\bdiv\brho_{\infty} & = & -\kappa_{\infty}\varrho\bu_{\infty} & \text{ in } \quad \Omega, \\
\div\bu_{\infty}&=&0& \text{ in } \quad \Omega,\\
\bu_{\infty} & = & \mathbf{0} & \text{ on } \quad \Gamma.
\end{array}
\right.
\end{equation*}
which corresponds to the perfectly incompressible linear elasticity equations. We remark that under this regime, the pseudotress and stress tensors  coincide.

Returning to the eigenvalue problem \eqref{eq:eigen_limit}, we introduce the linear solution operator $\bT_{\infty}$ defined by 
 \begin{equation*}
 \bT_{\infty}:\mathbf{X}_0\rightarrow \mathbf{X}_0,\qquad \boldsymbol{f}\mapsto \bT_{\infty}\boldsymbol{f}:=\widehat{\brho}_{\infty},
 \end{equation*}
 where $\widehat{\brho}_{\infty}\in\mathbf{X}_0$ is the solution of the following source problem: given $\boldsymbol{f}\in\mathbf{X}_0$, find $\widehat{\brho}_{\infty}\in\mathbf{X}_0$ such that 
 \begin{equation}
 \label{eq:source_cont_limit}
 a_{\infty}(\widehat{\brho}_{\infty},\btau)=b_{\infty}(\boldsymbol{f},\btau)\quad\forall\btau\in\mathbf{X}_0.
 \end{equation}
 
 Since $ a_{\infty}(\cdot,\cdot)$ is coercive in $\mathbf{X}_0$, from Lax-Milgram's lemma we have that \eqref{eq:source_cont_limit} has a unique solution and hence, $\bT_{\infty}$ is well defined. Similarly to $\bT_{\lambda}$,   $(\kappa_{\infty},\boldsymbol{\rho}_{\infty})\in \mathbb{R}^+\times\mathbf{X}_0$ solves problem \eqref{eq:eigen_limit} if and only if $(1/(1+\kappa_{\infty}),\boldsymbol{\rho}_{\infty})$ is an eigenpair of $\bT_{\infty}$ with a nonvanishing eigenvalue, i.e., $ \bT_{\infty}\boldsymbol{\rho}_{\infty}=\xi_{\infty}\boldsymbol{\rho}_{\infty}$ with $\xi_{\infty}:=\frac{1}{\kappa_{\infty}+1}\neq 0$  and $\boldsymbol{\rho}_{\infty}\neq\boldsymbol{0}$.

 On the other hand, the spectral characterization of $\bT_{\infty}$ follows the same arguments used for $\bT_{\lambda}$. For  simplicity we skip details and only present the result.
 \begin{lemma}
The spectrum of $\bT_{\infty}$ decomposes as follows: $\sp(\bT_{\infty})=\{0,1\}\cup\{\xi_{\infty,k}\}_{k\in\mathbb{N}}$, where the following properties hold:
\begin{itemize}
\item[(a)] $\xi_{\infty}=1$ is an infinite-multiplicity eigenvalue of $\bT_{\infty}$ with associated eigenspace $\mathbf{K}$ (cf. \eqref{eq:spaceK});
\item[(b)] $\xi_{\infty}=0$ is an eigenvalue of $\bT_{\infty}$ with associated eigenspace $\mathbf{Z}$ (cf. \eqref{eq:spaceZ});
\item[(c)] $\{\xi_{\infty,k}\}_{k\in\mathbb{N}}\subset (0,1)$ is a sequence of nondefective finite-multiplicity eigenvalues of $\bT_{\infty}$ that converge to 0.
\end{itemize}
\end{lemma}
 
 The following result states that $\bT_{\lambda}$ converges to $\bT_{\infty}$ as $\lambda\rightarrow+\infty$ (namely, $\nu\rightarrow 1/2$).
 \begin{lemma}
 \label{lmm:limit}
 Given $\boldsymbol{f}\in\mathbf{X}_0$, there exists a constant $C(\O,\mu,\varrho)>0$ depending on $\O$, $\mu$, and the density $\varrho$, such that
 \begin{equation*}
 \|(\bT_{\lambda}-\bT_{\infty})\boldsymbol{f}\|_{\bdiv,\O}\leq\frac{C(\O,\mu,\varrho)}{\lambda}\|\boldsymbol{f}\|_{0,\O}.
 \end{equation*}
 \end{lemma}
 \begin{proof}
 Let $\boldsymbol{f}\in\mathbf{X}_0$. Recalling the definition of operators $\bT$ and $\bT_{\infty}$, we need to estimate $\|\widehat{\brho}-\widehat{\brho}_{\infty}\|_{\bdiv,\O}$ in order to derive our result. Let us invoke the source problems \eqref{eq:source_cont} and \eqref{eq:source_cont_limit}. Subtracting \eqref{eq:source_cont_limit} to \eqref{eq:source_cont}  we obtain 
 \begin{multline}
 \label{eq:resta_1}
 \mu\int_{\O}\bdiv(\widehat{\brho}-\widehat{\brho}_{\infty})\cdot\bdiv\btau+\int_{\O}(\widehat{\brho}-\widehat{\brho}_{\infty})^{\texttt{d}}:\btau^{\texttt{d}}+\frac{\mu}{d(d\lambda+(d+1)\mu)}\int_{\O}\tr(\boldsymbol{\widehat{\rho}})\tr(\btau)\\
 =\frac{\mu}{d(d\lambda+(d+1)\mu)}\int_{\O}\tr(\boldsymbol{f})\tr(\btau)\quad\forall\btau\in\mathbf{X}_0.
 \end{multline}
 Set $\btau:=\widehat{\brho}-\widehat{\brho}_{\infty}$ in \eqref{eq:resta_1}. Hence
 \begin{multline}
 \label{eq:resta_2}
 \mu\varrho^{-1}\|\bdiv(\widehat{\brho}-\widehat{\brho}_{\infty})\|_{0,\O}^2+\|(\widehat{\brho}-\widehat{\brho}_{\infty})^{\texttt{d}}\|_{0,\O}^2\\
 =\frac{\mu}{d(d\lambda+(d+1)\mu)}\int_{\O}(\tr(\boldsymbol{f})-\tr(\widehat{\brho}))\tr(\widehat{\brho}-\widehat{\brho}_{\infty}).
 \end{multline}
 Now the task is to obtain a lower bound for the left-hand side of \eqref{eq:resta_2} and an upper bound or the right-hand side. For the left-hand side, invoking 
 Lemma \ref{lmm:cota} we obtain
 \begin{multline}
 \label{eq:LHS}
\mu\varrho^{-1}\|\bdiv(\widehat{\brho}-\widehat{\brho}_{\infty})\|_{0,\O}^2+\|(\widehat{\brho}-\widehat{\brho}_{\infty})^{\texttt{d}}\|_{0,\O}^2\\
\geq  \frac{1}{2}\min\{c_1,1\}\min\{\mu\varrho^{-1},1\}\|\widehat{\brho}-\widehat{\brho}_{\infty}\|_{\bdiv,\O}^2.
 \end{multline}
 Now, for the right hand side, we have
 \begin{multline}
 \label{eq:RHS}
 \frac{\mu}{d(d\lambda+(d+1)\mu)}\int_{\O}(\tr(\boldsymbol{f})-\tr(\widehat{\brho}))\tr(\widehat{\brho}-\widehat{\brho}_{\infty})\\\leq\frac{\mu}{d^2\lambda}\|\tr(\boldsymbol{f})-\tr(\widehat{\brho})\|_{0,\O}\|\tr(\widehat{\brho}-\widehat{\brho}_{\infty})\|_{0,\O}\\
 \leq\frac{\mu}{\lambda}\left(\|\boldsymbol{f}\|_{0,\O}+\|\widehat{\brho}\|_{0,\O}\right)\|\widehat{\brho}-\widehat{\brho}_{\infty}\|_{\bdiv,\O}\leq \frac{C}{\lambda}\|\boldsymbol{f}\|_{0,\O}\|\widehat{\brho}-\widehat{\brho}_{\infty}\|_{\bdiv,\O},
 \end{multline}
 where we have used the data dependency estimate \eqref{eq:dependency} and that $d^{-2}< 1$ for $d=2,3$. Finally, replacing \eqref{eq:LHS} and \eqref{eq:RHS} in \eqref{eq:resta_2} we obtain 
 \begin{equation*}
\|\widehat{\brho}-\widehat{\brho}_{\infty}\|_{\bdiv,\O}\leq\frac{C(\O,\mu,\varrho)}{\lambda}\|\boldsymbol{f}\|_{0,\O}, 
 \end{equation*}
 where $C(\O,\mu,\varrho):=\displaystyle\frac{2C}{\min\{c_1,1\}\min\{\mu\varrho^{-1},1\}}$. This concludes the proof.
 \end{proof}
As a consequence of the previous  lemma, we have the following result (see \cite{ MR1115235}, for instance).
\begin{theorem}[Separation of the spectrum]
Let $\xi_{\infty}>0$ be an eigenvalue of $\bT_{\infty}$ of multiplicity $m$. Let $D$ be any disc of the complex plane centered at  $\xi_{\infty}$ containing no other element of the spectrum of $\bT_{\infty}.$ Then, for $\lambda$ large enough, $D$ contains exactly $m$ eigenvalues of $\bT_{\lambda}$ (repeated according to their respective multiplicities). Consequently, each eigenvalue $\xi_{\infty}>0$ of $\bT_{\infty}$ is a limit of eigenvalues $\xi$ of $\bT_{\lambda}$, as $\lambda$ goes to infinity.
\end{theorem}

\section{Finite element discretization}
\label{sec:mixed_method}
The present section deals with the finite element approximation for the eigenvalue 
problem.  To do this task, we begin by introducing a regular family of triangulations of $\bar{\O}\subset\mathbb{R}^d$ denoted by $\{\CT_h\}_{h>0}$. Let $h_T$ be the diameter of a triangle/tetrahedron $T\subset\CT_{h}$ and let us define $h:=\max\{h_T\,:\, T\in \CT_h\}$.
\subsection{The finite element spaces}
Given an integer $\ell\geq 0$ and a subset $D$ of $\mathbb{R}^d$, we denote by $\mathbb{P}_\ell(D)$ the space of polynomials of degree at most $\ell$ defined in $D$. We mention that, for tensorial fields we will define $\mathbf{P}_\ell(D):=[\mathbb{P}_\ell(D)]^{d\times d}$ and for vector fields $P_\ell(D):=[\mathbb{P}_\ell(D)]^d$.
With these ingredients at hand, for $k\geq 0$ we define the local Raviart-Thomas space of order $k$
 as follows  (see \cite{MR3097958})
\begin{equation*}
 \mathbf{RT}_k(T):=P_k(T)\oplus \mathbb{P}_k(T)\boldsymbol{x},
 \end{equation*}
 where $\boldsymbol{x}\in\mathbb{R}^d$. With this local space, we define the global Raviart-Thomas (RT) space, which we denote by $\mathbb{RT}_k(\CT_h)$, as follows
 \begin{equation*}
 \boldsymbol{\mathcal{RT}}_k(\CT_h):=\{\btau\in\mathbf{X}\,:\,(\tau_{i1},\cdots,\tau_{in})^{\texttt{t}}\in\mathbf{RT}_k(T)\,\,\forall i\in\{1,\ldots,n\},\,\,\forall T\in\CT_h\}.
 \end{equation*}
 Hence, for $k\ge 0$, we define the discrete space
$$ \mathbf{X}_{h,0}:=\left\{ \btau_h\in \boldsymbol{\mathcal{RT}}_k(\CT_h)\,\,:\,\,\int_{\O}\tr(\btau_h)=0  \right\}\subset\mathbf{X}_0.
$$
 
From  \cite{MR2009375} we have the following well known approximation properties for the spaces defined above.  Let $\bPi_h^k:\boldsymbol{\mathcal{H}}^t (\O)\rightarrow \boldsymbol{\mathcal{RT}}_k(\CT_h)$ be the Raviart-Thomas interpolation operator. For $t\in (0,1]$ and $\btau\in\boldsymbol{\mathcal{H}}^t(\O)\cap\boldsymbol{\mathcal{H}}(\bdiv,\O)$ the following error estimate holds true
 \begin{equation*} \label{daniel1}
 \|\btau-\bPi_h^k\btau\|_{0,\O}\leq Ch^t \big(\|\btau\|_{t,\O}+\|\bdiv\btau\|_{0,\O}\big).
 \end{equation*}
 Also, for $\btau\in\mathbb{H}^t(\O)$ with $t>1$, there holds (see for instance \cite[Chapter 1]{MR2050138})
 \begin{equation*}\label{daniel2}
 \|\btau-\bPi_h^k\btau\|_{0,\O}\leq Ch^{\min\{t,k+1\}} |\btau|_{t,\O}.
 \end{equation*} 
 
 Let $\mathcal{P}_h^k:\mathbf{L}^2(\O)\rightarrow\mathbf{P}_k(\CT_h)$ be the $\L^2(\O)^d$-orthogonal projector. As a first property, we have the following commutative diagram $ \bdiv(\bPi_h^k\btau)=\mathcal{P}_h^k(\bdiv\btau)$.
 
 If $\bv\in\H^t (\O)^{d}$ with $t>0$, there holds
 \begin{equation*}\label{daniel3}
 \|\bv-\mathcal{P}_h^k\bv\|_{0,\O}\leq Ch^{\min\{t,k+1 \}} |\bv|_{t,\O}.
 \end{equation*}
 
 Finally, for each $\btau\in\boldsymbol{\mathcal{H}}^t(\O)$ such that $\bdiv\btau\in\H^t (\O)^{d}$, there holds
 \begin{equation*} \label{daniel4}
 \|\bdiv(\btau-\bPi_h^k\btau)\|_{0,\O}\leq Ch^{\min\{t,k+1\}} |\bdiv\btau|_{t,\O}.
 \end{equation*} 
 \subsection{The discrete the eigenvalue problem}
 With the finite element spaces previously defined, we introduce the discretization of \eqref{eq:spectral_continuous} as follows: find $\kappa_h\in\mathbb{R}^+$ and $\boldsymbol{0}\neq\brho_h\in\mathbf{X}_{h,0}$ such that
\begin{equation}
\label{eq:spectral_discrete}
a(\brho_h,\btau_h)=(\kappa_h+1) b(\brho_h,\btau_h)\quad\forall\btau_h\in\mathbf{X}_{h,0}.
\end{equation}
For the analysis of this discrete eigenvalue problem, we introduce the linear operator $\bT_h$, which is defined as follows
 \begin{equation*}
 \bT_{\lambda,h}:\mathbf{X}_0\rightarrow \mathbf{X}_{0},\qquad \boldsymbol{f}\mapsto \bT_{\lambda,h}\boldsymbol{f}:=\widehat{\brho}_h,
 \end{equation*}
 where $\widehat{\brho}\in\mathbf{X}_{h,0}$ is the solution of the following source problem: given $\boldsymbol{f}\in\mathbf{X}_0$, find $\widehat{\brho}_h\in\mathbf{X}_{h,0}$ such that 
 \begin{equation}
 \label{eq:source_disc}
 a(\widehat{\brho}_h,\btau_h)=b(\boldsymbol{f},\btau_h)\quad\forall\btau_h\in\mathbf{X}_{h,0}.
 \end{equation}
 
 Since $a(\cdot,\cdot)$ is $\mathbf{X}_{h,0}$-coercive, applying the Lax-Milgram's to \eqref{eq:source_disc} allows us to conclude the uniqueness of solution for this discrete problem and the existence of a constant $C>0$, independent of $h$ and $\lambda$, such that 
 \begin{equation*}
\|\bT_{\lambda,h}\boldsymbol{f}\|_{\bdiv,\O }= \|\widehat{\brho}_h\|_{\bdiv,\O}\leq C\|\boldsymbol{f}\|_{0,\O},
 \end{equation*}
implying that $\bT_{\lambda,h}$ is well defined. 

\cblue{As in the continuous case, it is easy to check that $(\kappa_h,\boldsymbol{\rho}_h)\in \mathbb{R}^+\times\mathbf{X}_{h,0}$ solves problem \eqref{eq:spectral_discrete} if and only if $(1/(1+\kappa_h),\boldsymbol{\rho}_h)$ is an eigenpair of $\bT_{\lambda,h}$ with a nonvanishing eigenvalue, i.e.,  $\bT_{\lambda,h}\boldsymbol{\rho}_h=\xi_h\boldsymbol{\rho}_h$, with $\xi_h:=\frac{1}{\kappa_h+1}\neq 0$ and $\boldsymbol{\rho}_h\neq\boldsymbol{0}$.}

\cblue{Let us define the space $\mathbf{K}_h:=\mathbf{K}\cap\mathbf{X}_{h,0}=\{\btau_h\in\mathbf{X}_{h,0}\,:\,\bdiv\btau_h=\boldsymbol{0}\,\,\text{in}\,\O\}$. We observe that on this space, $\bT_{\lambda,h}$ reduces to the identity and hence, $\xi_h=1$ is an eigenvalue of $\bT_{\lambda,h}$ with associated eigenspace $\mathbf{K}_h$. Let $\bP_h:\mathbf{X}_0\rightarrow\mathbf{X}_{h,0}$ be the discrete counterpart of $\bP$ that  satisfies $\bP_h\brho:=\widetilde{\brho}_h$ for $\brho\in\mathbf{X}_0$, where $\widetilde{\brho}_h$ satisfies the   discrete version of problem \eqref{def:aux_mixed_problem}
{ \footnotesize
\begin{align}
\displaystyle\int_{\O}\widetilde{\boldsymbol{\rho}}_h^{\texttt{d}}:\btau_h^{\texttt{d}}+\frac{\mu}{d(d\lambda+(d+1)\mu)}\int_{\O}\tr(\widetilde{\boldsymbol{\rho}}_h)\tr(\btau_h)+\int_{\O}\widetilde{\bu}_h\cdot\bdiv\btau_h& = 0\quad\forall\btau_h\in\mathbf{X}_{h,0}, \\
\displaystyle\int_{\O}\bv_h\cdot\bdiv\widetilde{\boldsymbol{\rho}}_h & =  \displaystyle\int_{\O}\bv_h\cdot\bdiv\boldsymbol{\rho}_h\quad\forall\bv_h\in\Q_h, 
\end{align}}
where $\Q_h:=\{\bv_h\in\L^2(\O)^d\,:\, \bv_h|_T\in P_k(T)^d\,\,\,\,\forall T\in\CT_h\}$. Thanks to the discrete version of the  Babu\v ska-Brezzi, this discrete problem is well posed, implying that $\bP_h$ is well defined. Hence, the following C\'ea estimate holds
\begin{equation}
\label{eq:cea1}
\|\widetilde{\brho}-\widetilde{\brho}_h\|_{\bdiv,\O}+\|\widetilde{\bu}-\widetilde{\bu}_h\|_{0,\O}\leq C\left\{\inf_{\btau_h\mathbf{X}_{h,0}}\|\widetilde{\brho}-\btau_h\|_{\bdiv,\O}+\inf_{\bv\in\Q_h}\|\widetilde{\bu}-\bv_h\|_{0,\O}\right\}
\end{equation}}

\cblue{Invoking \eqref{eq:cea1},\eqref{eq:reg_s},  approximation properties, and  following \cite[Lemma 4.4]{MMR3}, it is possible to prove that for $\brho_h\in\mathbf{X}_{h,0}$, the following estimate holds
\begin{equation}
\label{eq:eq:approx_PPh}
\|\bP\brho_h-\bP_h\brho_h\|_{\bdiv,\O}\leq Ch^s\|\bdiv\brho_h\|_{0,\O},
\end{equation}
where $C>0$ is independent of $h$ and $\lambda$ and $s$ is the regularity parameter involved in \eqref{eq:reg_s}.}

As we claimed before, operator $\bT_{\lambda}$ is non-compact, implying that the convergence and error estimates of the method must be analyzed with the theory of \cite{MR483400,MR483401}. This demands to consider discrete sources for $\bT_{\lambda}$ and $\bT_{\lambda,h}$. 
 
 We introduce the following norm (see \cite{MR483400})
\begin{equation}
\label{eq:norm_h}
\displaystyle \|\bT\|_h:=\sup_{\boldsymbol{0}\neq \boldsymbol{f}_h\in\mathbf{X}_{0,h}}\frac{\|\bT \boldsymbol{f}_h\|_{\bdiv,\O}}{\|\boldsymbol{f}_h\|_{\bdiv,\O}}.
\end{equation}

Also, we recall  the definitions of properties P1 and P2 of  \cite{MR483400}:
\begin{itemize}
\item P1: $\|\bT-\bT_h\|_h\rightarrow 0$ as $h\rightarrow 0$;
\item P2: $\forall\btau\in\mathbf{X}_0$, $\displaystyle \inf_{\btau_h\in\mathbf{X}_{0,h}}\|\btau-\btau_h\|_{\bdiv,\O}\rightarrow 0$ as $h\rightarrow 0$.
\end{itemize}

Our goal is to establish properties P1 and P2, in order to ensure the spectral convergence. We observe that P2 is an immediate, since the smooth functions are dense in 
$\mathbf{X}_0$. \cblue{On the other hand, P1 follows the same arguments of the proof for \cite[Lemma 5.1]{MMR3}, where \eqref{eq:eq:approx_PPh}  and Lemma \ref{lmm:invariancy} are necessary.
\begin{lemma}
\label{eq:P1}
There exists a constant $C>0$, independent of $h$ and $\lambda$ such that
$$
\displaystyle \|\bT_{\lambda}-\bT_{\lambda,h}\|_h\leq Ch^{s}.
$$
\end{lemma}}

\cblue{\subsection{An $\mathbf{L}^2$ error estimate for the pseudostress}
it is possible to obtain an improvement for the $\mathbf{L}^2$ error estimate for the pseudostress. This results is obtained by the classic  Aubin-Nitsche argument, where a duality argument is used. We begin by noticing that bilinear form $b(\cdot,\cdot)$ is precisely an inner product that we can consider on the $\mathbf{L}^2(\Omega)$ space. Indeed, if we denote by $b(\cdot,\cdot)=(\cdot,\cdot)_b$ this inner product, and for any $\btau\in\mathbf{L}^2(\O)$, we have that the norm induced by $(\cdot,\cdot)_b$ is defined as follows
\begin{equation*}
\|\btau\|_{b}^2:=\|\btau^{\texttt{d}}\|_{0,\O}^2+\|(d(d\lambda+(\mu+1)d))^{-1/2}\tr(\btau)\|_{0,\O}^2.
\end{equation*}
We notice that there exists a constant $C_{\textrm{II}}:=\sqrt{\left(1+\frac{\sqrt{d}}{d}\right)^2+\frac{d}{\mu(d+1)}}$, clearly independent of $\lambda$,  such that $\|\btau\|_{b}\leq C_{\textrm{II}}\|\btau\|_{0,\O}$ for all $\btau\in\mathbf{L}^2(\O)$. The task now is to prove the estimate on the other direction. If $\btau\in\mathbf{L}^2(\O)$ we have that $\btau=\btau^{\texttt{d}}+d^{-1}\tr(\btau)\mathbf{I}$. Since this decomposition is orthogonal with respect the $\mathbf{L}^2(\O)$ product, we have 
\begin{equation*}
\|\btau\|_{0,\O}^2=\|\btau^{\texttt{d}}\|_{0,\O}^2+\frac{1}{d}\|\tr(\btau)\|_{0,\O}^2.
\end{equation*}
On the other hand, from the definition of $\|\cdot\|_b$, we have
\begin{multline}
\label{eq:second_bound}
\|\btau\|_b^2=\|\btau^{\texttt{d}}\|_{0,\O}^2+(d(d\lambda+(\mu+1)d))^{-1}\|\tr(\btau)\|_{0,\O}^2\\
\geq \min\{1,(d\lambda+(\mu+1)d))^{-1}\}\left(\|\btau^{\texttt{d}}\|_{0,\O}^2+\frac{1}{d}\|\tr(\btau)\|_{0,\O}^2\right)=C_{\textrm{I}}\|\btau\|_{0,\O}^2,
\end{multline}
where $C_{\textrm{I}}:=\min\{1,(d\lambda+(\mu+1)d))^{-1}\}$. 
Let us prove the result of our interest.
\begin{lemma}
\label{lmm:L2error}
There exists a constant $C_{\star}>0$, independent of $h$ and $\lambda$, such that
\begin{equation*}
\|\widehat{\boldsymbol{\rho}}-\widehat{\boldsymbol{\rho}}_h\|_{0,\O}\leq C_{\star} h^{2s}\|\bF_h\|_{0,\O}.
\end{equation*}
\end{lemma}
\begin{proof}
We use a duality argument. Let us consider the following auxiliary problem: find $\boldsymbol{\varphi}\in\mathbf{X}_0$ such that
\begin{equation*}
a(\boldsymbol{\varphi},\btau)= b(\widehat{\boldsymbol{\rho}}-\widehat{\boldsymbol{\rho}}_h,\btau)\quad\forall\btau\in\mathbf{X}_0.
\end{equation*}
Due to Lax-Milgram's lemma, this problem is well posed and exists a unique solution  $\boldsymbol{\varphi}\in\mathbf{X}_0$ which satisfies 
\begin{equation}
\label{eq:stability_dual}
\|\boldsymbol{\varphi}\|_{s,\O}+\|\bdiv\boldsymbol{\varphi}\|_{1,\O}\leq C\|\widehat{\boldsymbol{\rho}}-\widehat{\boldsymbol{\rho}}_h\|_{0,\O}.
\end{equation}
Setting $\btau=\widehat{\boldsymbol{\rho}}-\widehat{\boldsymbol{\rho}}_h$ and using the Galerkin orthogonality property, in conjunction with \eqref{eq:second_bound} we have
\begin{align*}
C_{\textrm{I}}\|\widehat{\boldsymbol{\rho}}-\widehat{\boldsymbol{\rho}}_h\|_{0,\O}^2
&\leq a(\boldsymbol{\varphi}-\boldsymbol{\varphi}_h,\widehat{\boldsymbol{\rho}}-\widehat{\boldsymbol{\rho}}_h)\\
&\leq \mu\varrho^{-1}\|\bdiv(\boldsymbol{\varphi}-\boldsymbol{\varphi}_h)\|_{0,\O}
\|\bdiv(\widehat{\boldsymbol{\rho}}-\widehat{\boldsymbol{\rho}}_h)\|_{0,\O}\\
&\quad + \|(\boldsymbol{\varphi}-\boldsymbol{\varphi}_h)^{\texttt{d}}\|_{0,\O}
\|(\widehat{\boldsymbol{\rho}}-\widehat{\boldsymbol{\rho}}_h)^{\texttt{d}}\|_{0,\O}\\
&\quad + (d(d+1))^{-1}\|\tr(\boldsymbol{\varphi}-\boldsymbol{\varphi}_h)\|_{0,\O}
\|\tr(\widehat{\boldsymbol{\rho}}-\widehat{\boldsymbol{\rho}}_h)\|_{0,\O}\\
&\leq \mu\varrho^{-1}\|\bdiv(\boldsymbol{\varphi}-\boldsymbol{\varphi}_h)\|_{0,\O}
\|\bdiv(\widehat{\boldsymbol{\rho}}-\widehat{\boldsymbol{\rho}}_h)\|_{0,\O}\\
&\quad + \left(1+\frac{\sqrt{d}}{d}\right)^2
\|(\boldsymbol{\varphi}-\boldsymbol{\varphi}_h)\|_{0,\O}
\|(\widehat{\boldsymbol{\rho}}-\widehat{\boldsymbol{\rho}}_h)\|_{0,\O}\\
&\quad + (d+1)^{-1}\|\boldsymbol{\varphi}-\boldsymbol{\varphi}_h\|_{0,\O}
\|\widehat{\boldsymbol{\rho}}-\widehat{\boldsymbol{\rho}}_h\|_{0,\O}.
\end{align*}
Hence, using \eqref{eq:stability_dual},  Lemma \ref{eq:P1}, and the approximation properties of $\boldsymbol{\Pi}_h^k$, we have 
\begin{equation*}
\|\widehat{\boldsymbol{\rho}}-\widehat{\boldsymbol{\rho}}_h\|_{0,\O}\leq C_{\star}h^{2s}\|\bF_h\|_{0,\O}.
\end{equation*}
where  $ C_{\star}:=C_{\textrm{I}}^{-1}\left(\mu\varrho^{-1}+\left(1+\frac{\sqrt{d}}{d}\right)^2+(d+1)^{-1}\right)$. Since $\brho$ and $\brho_h$ are both trace null tensor, hence $C_{\textrm{I}}=1$, implying that $C_{\star}$ is independent of $\lambda$ and $h$. This  concludes the proof.
\end{proof}}

\section{Spectral convergence}
\label{sec:spec_convergence}
Now we analyze convergence of numerical method when is applied to the eigenvalue problem.   As we have claimed before, the arguments to perform this analysis 
follows from  \cite{MR483400}.

Let us begin by introducing some definitions. Given a linear bounded operator $\boldsymbol{S}:\mathbf{Y}\rightarrow\mathbf{Y}$, where $\mathbf{Y}$ is a Hilbert space, the spectrum of $\boldsymbol{S}$ is defined by 
\begin{equation*}
\sp(\boldsymbol{S}):=\{z\in\mathbb{C}\,:\, (z\boldsymbol{I}-\boldsymbol{S})\,\,\,\text{is not invertible}\}.
\end{equation*}
With this set at hand, we define the resolvent set of $\boldsymbol{S}$ by $\text{res}(\boldsymbol{S}):=\mathbb{C}\setminus\sp(\boldsymbol{S})$. Moreover, for any $z\in\text{res}(\boldsymbol{S})$, we define $\boldsymbol{R}_z(\boldsymbol{S}):=(z\boldsymbol{I}-\boldsymbol{S})^{-1}:\mathbf{Y}\rightarrow\mathbf{Y}$, which is the resolvent operator of $\boldsymbol{S}$ corresponding to $z$.

We begin by proving that the resolvent operator associated to $\bT_{\lambda}$ is bounded.
\begin{lemma}
\label{lmm:resolvent_bounded}
Let $F\subset\mathbb{C}$ be a closed subset such that $F\cap\sp(\bT_{\infty})=\emptyset$. Then, there exist constants $\lambda_0,C>0$ such that $\forall\lambda<\lambda_0$, there holds $F\cap\sp(\bT_{\lambda})=\emptyset$ and 
\begin{equation*}
\|\boldsymbol{R}_z(\bT_{\lambda})\|:=\sup_{\boldsymbol{0}\neq\btau\in\mathbf{X}_0}\frac{\|\boldsymbol{R}_z(\bT_{\lambda})\|_{\bdiv,\O}}{\|\btau\|_{\bdiv,\O}}\leq C\quad\forall z\in F.
\end{equation*}
\end{lemma}
\begin{proof}
Consider  for all $z\in\text{res}(\bT_{\infty})$ the continuous mapping $z\mapsto\|(z\boldsymbol{I}-\bT_{\infty})^{-1}\|$. This mapping goes to zero as $|z|\rightarrow\infty$. Hence, 
if $F\subset\text{res}(\bT_{\infty})$, the previous mapping attains a maximum that we define by $\widetilde{C}:=\max_{z\in F}\|(z\boldsymbol{I}-\bT_{\infty})^{-1}\|$.  This constant $\widetilde{C}>0$ leads to the following bound
\begin{equation*}
\|(z\boldsymbol{I}-\bT_{\infty})\btau\|_{\bdiv,\O}\geq \widetilde{C}^{-1}\|\btau\|_{\bdiv,\O}\quad\forall\btau\in\mathbf{X}_0,\,\,\,\forall z\in F.
\end{equation*}
From triangle inequality we have
\begin{equation*}
\|(z\boldsymbol{I}-\bT_{\infty})\btau\|_{\bdiv,\O}\leq \|(z\boldsymbol{I}-\bT_{\lambda})\btau\|_{\bdiv,\O}+\|(\bT_{\lambda}-\bT_{\infty})\btau\|_{\bdiv,\O}.
\end{equation*}
Invoking Lemma \ref{lmm:limit}, for the second term on the right-hand side of inequality above, we have the existence of $\lambda_0>0$ such that for all $\lambda<\lambda_0$, there holds
\begin{equation*}
\|(\bT_{\lambda}-\bT_{\infty})\btau\|_{\bdiv,\O}\leq \frac{\|\btau\|_{\bdiv,\O}}{2\widetilde{C}}	\quad\forall\btau\in\mathbf{X}_0.
\end{equation*}
With this bound at hand, if $\btau\in\mathbf{X}_0$, for all $\lambda<\lambda_0$ and  for all $z\in F$ we have
\begin{equation*}
\|(z\boldsymbol{I}-\bT_{\lambda})\btau\|_{\bdiv,\O}\geq \|(z\boldsymbol{I}-\bT_{\infty})\btau\|_{\bdiv,\O}-\|(\bT_{\lambda}-\bT_{\infty})\btau\|_{\bdiv,\O}\geq \frac{\|\btau\|_{\bdiv,\O}}{2\widetilde{C}},
\end{equation*}
implying that $z$ is not an eigenvalue of $\bT_{\lambda}$. Also, since $0\notin\text{res}(\bT_{\infty})$ we have that $z\neq 0$. Finally, $z\notin\sp(\bT_{\lambda})$ and hence $(z\boldsymbol{I}-\bT_{\lambda})^{-1}$ exists for all $\lambda<\lambda_0$ and for all $z\in F$ and the following estimate holds $\|\boldsymbol{R}_z(\bT_{\lambda})\|\leq 2\widetilde{C}$. This concludes the proof.
\end{proof}

For the  resolvent operator of the discrete version of $\bT_{\lambda}$ we are able to obtain an analogous result as the proved in Lemma \ref{lmm:resolvent_bounded}. In fact, to prove this result,  property P1 takes relevance.
\begin{lemma}
\label{lmm:discrete_resolvent_bounded}
Let $F\subset\mathbb{C}$ be a closed subset such that $F\cap\sp(\bT_{\infty})=\emptyset$. Then, there exist positive constants $h_0, \lambda_0$ and $C$, such that for all $\lambda<\lambda_0$ and for all $h<h_0$, there holds $F\cap\sp(\bT_{\lambda.h})=\emptyset$ and 
\begin{equation*}
\|\boldsymbol{R}_z(\bT_{\lambda,h})\|_h\leq C\quad\forall z\in F,
\end{equation*}
where the norm $\|\cdot\|_h$ is the one defined in \eqref{eq:norm_h}.
\end{lemma}
\begin{proof}
From Lemma \ref{lmm:resolvent_bounded}, if $F$ is a closed set that satisfies $F\cap\sp(\bT_{\infty})=\emptyset$, hence for all $\lambda<\lambda_0$, the following estimate holds
\begin{equation}
\label{eq:1}
\|\btau\|_{\bdiv,\O}\leq C\|(z\boldsymbol{I}-\bT_{\lambda})\btau\|_{\bdiv,\O}\quad\forall\btau\in\mathbf{X}_0.
\end{equation}
Invoking property P1 (cf. Lemma \ref{eq:P1}), there exists $h_0>0$ such that for all $h<h_0$ there holds
\begin{equation}
\label{eq:2}
\|(\bT_{\lambda}-\bT_{\lambda,h})\btau_h\|_{\bdiv,\O}\leq\frac{\|\btau_h\|_{\bdiv,\O}}{2C}\quad\forall\btau\in\mathbf{X}_{0,h}.
\end{equation}

Now, adding and subtracting $\bT_{\lambda}$, applying \eqref{eq:1} to $\btau_{h}\in\mathbf{X}_{0,h}\subset\mathbf{X}_0$, and using \eqref{eq:2}, for $z\in F$ we have 
\begin{equation*}
\|(z\boldsymbol{I}-\bT_{\lambda,h})\btau_h\|_{\bdiv,\O}\geq \|(z\boldsymbol{I}-\bT_{\lambda})\btau_h\|_{\bdiv,\O}-\|(\bT_{\lambda}-\bT_{\lambda,h})\btau_h\|_{\bdiv,\O}\geq \frac{\|\btau_h\|_{\bdiv,\O}}{2C}.
\end{equation*}

Finally, since $\dim(\mathbf{X}_{0,h})<+\infty$, it is immediate that $(z\boldsymbol{I}-\bT_{\lambda,h})^{-1}$ exists and therefore $z\notin\sp(\bT_{\lambda.h})$, implying that $\|\boldsymbol{R}_z(\bT_{\lambda,h})\|_h\leq C$ for $z\in F$. This completes the proof.
\end{proof}
\subsection{A priori error estimates}
We end  this section deriving error estimates for the eigenfunctions and eigenvalues. To do this task, we perform the analysis of adapting the theory of \cite{MR483401} to our case, similarly  as in \cite[Section 5]{LoMoRo_SIAM2010}. Let us begin by introducing the following projector $\boldsymbol{\Lambda}_h:\mathbf{X}_0\rightarrow\mathbf{X}_0$ with range $\mathbf{X}_{0,h}$. More precisely, if  $\btau\in\mathbf{X}_0$, then $(\boldsymbol{\Lambda}_0\btau)|_{\mathbf{X}_{0,h}}:=\bPi_h^k\btau$, where $\bPi_h^k$ is the Raviart-Thomas ($k\geq 0$) or BDM  ($k\geq 1$) interpolation operators. 

With this operator $\boldsymbol{\Lambda}_h$ at hand, we define $\boldsymbol{B}_{\lambda,h}:=\bT_{\lambda,h}\boldsymbol{\Lambda}_h:\mathbf{X}_0\rightarrow\mathbf{X}_0$ which have the same non-zero eigenvalues and corresponding eigenfunctions of $\bT_{\lambda,h}$. It is possible to prove using the same steps and arguments of \cite[Lemma 1]{LoMoRo_SIAM2010}, the existence of constants $h_0,\lambda_0$ and $C$, all strictly positive, such that 
\begin{equation}
\label{eq:bound_Bh}
\|\boldsymbol{R}_z(\boldsymbol{B}_{\lambda,h})\|\leq C\quad\forall h<h_0,\,\,\,\forall \lambda<\lambda_0,\,\,\,\forall z\in\gamma,
\end{equation}
where $\gamma$ is the boundary of a closed disk $D$ centered at a nonvanishing eigenvalue $\xi_{\infty}$ of $\bT_{\infty}$. This disk and the eigenvalue $\xi_{\infty}$ are such that $D\cap\sp(\bT_{\infty})=\{\xi_{\infty}\}$.

Let $\boldsymbol{\mathcal{E}}$ be the eigenspace of $\bT_{\lambda}$ corresponding to $\xi$. Observe that  if $\xi \in (0,1)$ is an isolated eigenvalue of $\bT_{\lambda}$ with multiplicity $m$, then there exist $m$
eigenvalues $\xi_{h}^{(1)},...,\xi_{h}^{(m)}$ of $\bT_{\lambda,h}$, repeated according to their respective multiplicities, which converge to $\xi$.
Also, let $\boldsymbol{\mathcal{E}}_{h}$ be the direct sum of their corresponding associated eigenspaces (see \cite{MR0203473}).

We introduce the spectral projector of $\bT_{\lambda}$, defined by $\boldsymbol{E}_{\lambda}:\mathbf{X}_0\rightarrow\mathbf{X}_0$. This projector is associated to an isolated eigenvalue $\xi$ and is such that
$$\boldsymbol{E}_\lambda:=\frac{1}{2\pi i}\int_{\gamma}\boldsymbol{R}_z(\bT_{\lambda})\,dz,$$
whereas $\boldsymbol{F}_{\lambda,h}:\mathbf{X}_0\rightarrow\mathbf{X}_0$ represents the spectral projector of $\boldsymbol{B}_{\lambda,h}$ corresponding to the eigenvalue $\xi_h$ and is defined by 
$$\boldsymbol{F}_{\lambda,h}:=\frac{1}{2\pi i}\int_{\gamma}\boldsymbol{R}_z(\boldsymbol{B}_{\lambda,h})\,dz.$$

It is worth to mention that $\boldsymbol{F}_{\lambda,h}$ is uniformly bounded in $h$ and $\lambda$ when these quantities are small enough (cf. \eqref{eq:bound_Bh}) and $\boldsymbol{F}_{\lambda,h}(\mathbf{X}_0)$ is the eigenspace of $\boldsymbol{B}_{\lambda,h}$ (and hence of $\bT_{\lambda,h}$), whereas $\boldsymbol{E}_{\lambda}(\mathbf{X}_0)$ is the eigenspace of $\bT_{\lambda}$.

Now we recall the definition of the \textit{gap} $\hdel$ between two closed
subspaces \cblue{$\boldsymbol{\CM}$ and $\boldsymbol{\CN}$ of $\boldsymbol{\L}^2(\O)$:
$$
\hdel(\boldsymbol{\CM},\boldsymbol{\CN})
:=\max\big\{\delta(\boldsymbol{\CM},\boldsymbol{\CN}),\delta(\boldsymbol{\CN},\boldsymbol{\CM})\big\}, \text{ where } \delta(\boldsymbol{\CM},\boldsymbol{\CN})
:=\sup_{\underset{\left\|\boldsymbol{x}\right\|_{0,\O}=1}{\boldsymbol{x}\in\boldsymbol{\CM}}}
\left(\inf_{\boldsymbol{y}\in\boldsymbol{\CN}}\left\|\boldsymbol{x}-\boldsymbol{y}\right\|_{0,\O}\right).
$$}

From now and on, let su assume that $\boldsymbol{E}_{\lambda}(\mathbf{X}_0)\subset \cblue{\boldsymbol{\mathcal{H}}^{r}(\O)}$, where $r$ is as in Lemma \ref{lmm:add_eigen}.  We establish with this \cblue{regularity} assumption, the following result that follows the same arguments of the proof of \cite[Lemma 5.3]{LoMoRo_SIAM2010}.
\begin{lemma}
There exist positive constants $h_0,\lambda_0$ and $C$ such that for all $h<h_0$ and for all $\lambda<\lambda_0$, the following estimates hold
$$
\|(\boldsymbol{E}_{\lambda}-\boldsymbol{F}_{\lambda,h})|_{\boldsymbol{E}_{\lambda}(\mathbf{X}_0)}\|\leq\|(\bT_{\lambda}-\boldsymbol{B}_{\lambda,h})|_{\boldsymbol{E}_{\lambda}(\mathbf{X}_0)}\|\leq Ch^r.
$$
\end{lemma}

Now we are in position to establish our first result of error estimates for the eigenfunctions and eigenvalues.
\begin{theorem}
\label{thm:error1}
There exist positive constants $h_0,\lambda_0$ and $C$ such that for all $h<h_0$ and for all $\lambda<\lambda_0$,  such that 
\begin{equation}
\label{eq:gap_error}
\hdel(\boldsymbol{F}_{\lambda,h}(\mathbf{X}_0),\boldsymbol{E}_{\lambda}(\mathbf{X}_0))\leq Ch^{\min\{r,k\}},
\end{equation}
and 
\begin{equation}
\label{eq:error_xi}
|\xi-\xi_h^{(i)}|\leq Ch^{\min\{r,k\}}\quad i=1,2,\ldots, m,
\end{equation}
where $k\geq 0$ is the polynomial degree of approximation.
\end{theorem}
\begin{proof}
The proofs of these estimates are adapted in \cite{LoMoRo_SIAM2010} according to \cite{MR483401}. More precisely, the proof for \eqref{eq:gap_error} follows from \cite[Theorem]{LoMoRo_SIAM2010} whereas the proof of \eqref{eq:error_xi} is identical to the proof of \cite[Lemma 5.6]{LoMoRo_SIAM2010}. 
\end{proof}

Let us recall that the eigenvalue $\xi$ of $\bT_{\lambda}$ and the eigenvalue  $\kappa$ that solves problem \eqref{eq:spectral_continuous_H0} are such that $\xi:=(1+\kappa)^{-1}$ and the same for their discrete counterparts. This relation, together with \eqref{eq:error_xi} implies that for $\kappa$ and $\kappa_h^{(i)}$ the order of convergence is also $\mathcal{O}(h^{\min\{r,k\}})$. However, the order of convergence for $|\kappa-\kappa_h^{(i)}|$ can be improved as is stated in the following result.
\begin{theorem}
\label{thm:double_order}
There exists a constant $C>0$ independent of $h$ and $\lambda$ such that for $h<h_0$, the following estimate holds
$$
|\kappa-\kappa_h^{(i)}|\leq Ch^{2\min\{r,k\}},\quad i=1,\ldots,m,
$$
where $k\geq 0$ is the polynomial degree of approximation and $r$ is the regularity parameter of the eigenspace.
\end{theorem}
\begin{proof}
Let $(\kappa,\boldsymbol{\rho})\in\mathbb{R}^+\times\mathbf{X}_0$ and $(\kappa_h^{(i)},\boldsymbol{\rho}_h)\in\mathbb{R}^+\times\mathbf{X}_{h,0}$ with $i\in\{1,\ldots, m\}$ be the eigensolutions of problems  
\eqref{eq:spectral_continuous_H0} and \eqref{eq:spectral_discrete}, respectively. Let us suppose that $\|\boldsymbol{\rho}_h\|_{\bdiv,\O}=1$. Now consider the following well known 
algebraic identity
\begin{equation}
\label{eq:padra}
(\kappa_h^{(i)}-\kappa)b(\boldsymbol{\rho}_h,\boldsymbol{\rho}_h)=a(\boldsymbol{\rho}-\boldsymbol{\rho}_h,\boldsymbol{\rho}-\boldsymbol{\rho}_h)-\kappa b(\boldsymbol{\rho}-\boldsymbol{\rho}_h,\boldsymbol{\rho}-\boldsymbol{\rho}_h).
\end{equation}
 Applying modulus on both sides of \eqref{eq:padra} we have, in one hand
 \begin{multline}
 \label{eq:padra1}
 |a(\boldsymbol{\rho}-\boldsymbol{\rho}_h,\boldsymbol{\rho}-\boldsymbol{\rho}_h)|\leq \mu\varrho^{-1}\|\bdiv(\boldsymbol{\rho}-\boldsymbol{\rho}_h)\|_{0,\O}^2 \hspace{-.1cm}+
\left(\left(1+\frac{\sqrt{d}}{d}\right)^2+\frac{\sqrt{d}}{(d+1)\mu}\right)\|\boldsymbol{\rho}-\boldsymbol{\rho}_h\|_{0,\O}^2\\
\leq C\max\left\{\mu\varrho^{-1},\left(1+\frac{\sqrt{d}}{d}\right)^2+\frac{\sqrt{d}}{(d+1)\mu}\right\} h^{\min\{r,k\}},
 \end{multline}
 where we have used the approximation of the eigenfunctions. Now, for the remaining term we have and using similar arguments as for the previous estimate, we have 
 \begin{multline}
 \label{eq:padra2}
 |b(\boldsymbol{\rho}-\boldsymbol{\rho}_h,\boldsymbol{\rho}-\boldsymbol{\rho}_h)|\leq \|(\boldsymbol{\rho}-\boldsymbol{\rho}_h)^{\texttt{d}}\|_{0,\O}^2+\frac{\|\tr(\boldsymbol{\rho}-\boldsymbol{\rho}_h)\|_{0,\O}^2}{(d+1)\mu}\\
 \leq \left\{\left( 1+\frac{\sqrt{d}}{d}\right)^2 +\frac{d}{(d+1)\mu}\right\}\|\boldsymbol{\rho}-\boldsymbol{\rho}_h\|_{0,\O}^2\\
 \leq C\left\{\left( 1+\frac{\sqrt{d}}{d}\right)^2 +\frac{d}{(d+1)\mu}\right\}h^{2\min\{r,k\}}.
 \end{multline}
 Finally, invoking the discrete eigenvalue problem $a(\boldsymbol{\rho}_h,\boldsymbol{\rho}_h)=(\kappa_h^{(i)}+1)b(\boldsymbol{\rho}_h,\boldsymbol{\rho}_h)$, the coercivity of $a(\cdot,\cdot)$ given in \eqref{eq:coercive_a} we obtain
 \begin{equation*}
 b(\boldsymbol{\rho}_h,\boldsymbol{\rho}_h)=\frac{a(\boldsymbol{\rho}_h,\boldsymbol{\rho}_h)}{\kappa_h^{(i)}+1}
 \geq\frac{2^{-1}\min\{\mu\varrho^{-1},1\}\min\{c_1,1\}\|\boldsymbol{\rho}_h\|_{\bdiv,\O}}{\kappa_h^{(i)}+1},
 \end{equation*}
 and hence, invoking Theorem \ref{thm:error1} we know that  $\kappa_h^{(i)}\rightarrow \kappa$ as $h\rightarrow 0$ and from the assumption $\|\boldsymbol{\rho}_h\|_{\bdiv,\O}=1$, we have
  \begin{equation}
  \label{eq:padra3}
 b(\boldsymbol{\rho}_h,\boldsymbol{\rho}_h)\geq\frac{2^{-1}\min\{\mu\varrho^{-1},1\}\min\{c_1,1\}}{\kappa+1}>0,
 \end{equation}
 where the constant involved is independent of $h$. Hence, gathering \eqref{eq:padra1}, \eqref{eq:padra2}, and \eqref{eq:padra3} and replacing them in \eqref{eq:padra} we conclude the proof.
\end{proof}
%
%
\section{A posteriori error analysis}
\label{sec:apost}
\cblue{This section is devoted to the development of novel residual-based robust a posteriori error estimators for the eigenvalue problems in its nearly incompressible regime   and the perfectly  incompressible case.  The main aim is to prove that the proposed estimator  is equivalent  with the error. Let us mention that on the forthcoming analysis we will focus only on eigenvalues with multiplicity
1. To perform the analysis of the error indicators, we begin with some definitions and elements of the mesh that are necessary. For any element  $T\in \CT_h$, we denote by $\CE_{T}$ the set of facets of $T$
and 
$$\CE_h:=\bigcup_{T\in\CT_h}\CE_{T}.$$
We decompose $\CE_h=\CE_{\O}\cup\CE_{\partial\O}$,
where  $\CE_{\partial\O}:=\{\ell\in \CE_h:\ell\subset \partial\O\}$
and $\CE_{\O}:=\CE\backslash\CE_{\partial \O}$. 
On the other hand, for each face/edge $e\in\mathcal{E}_h$ we fix a unit normal vector $\bn_e$ to $e$. Moreover, given $\btau\in\boldsymbol{\mathcal{H}}(\curl, \Omega)$ and $e\in\mathcal{E}_h$, we let $\jumpp{\btau\times\bn_e}$ be the corresponding jump of the tangential traces across $e$, that is
\begin{equation*}
\jumpp{\btau\times\bn_e}:=(\btau|_T-\btau|_{T'})\big|_e\times\bn_e,
\end{equation*} 
where $T$ and $T'$ are two elements of the triangulation with common edge/face $e$.}


%
\subsection{Residual-based a posteriori error estimator}
\cred{In this section we derive a new residual-based a posteriori error estimator for the proposed mixed formulation. The construction relies on the strong form of the discrete problem together with suitable element and edge residuals associated with the pseudostress approximation.}
\subsubsection{Definition of error indicator for nearly incompressible elasticity}
We begin by introducing the local element residual on each element, denoted by $\eta_{K}$, together with the measure of the mismatch between solutions on adjacent elements across their interfaces, denoted by $\eta_{J}$, as follows:
\begin{equation*}
	\eta_{K}^2:= h_K^2\|\rho_R\textbf{R}_{1,T}\|_{0,T}^2+h_K^2\|\rho_R R_{2,T}\|_{0,T}^2,\quad \eta_{J}^2:=h_E\|\rho_EJ_{1,\ell} \|_{0,\ell}^2+h_E\|\rho_EJ_{2,\ell} \|_{0,\ell}^2,
\end{equation*}
where 
$$
	\textbf{R}_{1,T} :=\mu\nabla(\varrho^{-1}\bdiv\brho_h)+\kappa_h 
	\brho^d_h, 
	\quad R_{2,T} := \rot\left(\kappa_h
	\brho^d_h
	\right). 
$$
with $\brho^d_h= \brho_h-\frac{\lambda+\mu}{d\lambda+(d+1)\mu}\tr(\brho_h)\mathbf{I}$.
We define the jump contributions by
{\small{\begin{align*}
	J_{1,\ell} &:= \begin{cases}
		\displaystyle \frac{1}{2}[\![\mu(\varrho^{-1}\bdiv\brho_h)\bn]\!], & E\in \Omega\cap\mathcal{E}_h\\
		\mu(\varrho^{-1}\bdiv\brho_h) \bn, & E\in \Gamma\cap\mathcal{E}_h\\
	\end{cases};\;
	J_{2,\ell} := \begin{cases}
		\displaystyle \frac{1}{2}\left[\!\left[\left(\kappa_h\brho^d_h
		\right)\times\bn\right]\!\right], & E\in \Omega\cap\mathcal{E}_h\\
		\left(\kappa_h\brho^d_h
		\right)\times \bn, & E\in \Gamma\cap\mathcal{E}_h.
	\end{cases}
\end{align*}}}
In this context, the parameters $\rho_R$ and $\rho_E$ are specified by
$$
\rho_R:=  (\mu \varrho^{-1})^{-\frac 12}, \ \    	\hspace{-0.1cm}\rho_E := \left(\sqrt{2}\right)^{-1}\min\left\{(\mu \varrho^{-1})^{-1/2}, \left(\kappa_h\left(\frac{1}{d}-\frac{\lambda+\mu}{d\lambda+(d+1)\mu}\right)\right)^{-\frac 12}\right\}.  
$$
\subsubsection{Definition of error indicator for incompressible elasticity}
In the incompressible case,  the local element residual on each element, denoted by $\eta_{K,\infty}$, together with the measure of the mismatch between solutions on adjacent elements across their interfaces, denoted by $\eta_{J,\infty}$ is  as follows:
\begin{align*}
	\eta_{K,\infty}^2&:= h_K^2\|\tilde{\rho}_{R}\textbf{R}_{1,T}\|_{0,T}^2+h_K^2\|\tilde{\rho}_{R} R_{2,T}\|_{0,T}^2,\;\;
	 \eta_{J,\infty}^2:=h_E\|\tilde{\rho}_{E}J_{1,\ell} \|_{0,\ell}^2+h_E\|\tilde{\rho}_{E}J_{2,\ell} \|_{0,\ell}^2,
\end{align*}
where 
$$
	\textbf{R}_{1,T}:=\mu\nabla(\varrho^{-1}\bdiv\brho_h)+\kappa_h\brho^d_{h,\infty},\quad
	R_{2,T} := \rot\left(\kappa_h\brho^d_{h,\infty}\right). 
$$
with $\brho^d_{h,\infty}=\brho_h-\frac{1}{d}\tr(\brho_h)\mathbf{I}$.
We define the jump contributions by
{\small \begin{align*}
	J_{1,\ell} &:= \begin{cases}
		\displaystyle \frac{1}{2}[\![\mu(\varrho^{-1}\bdiv\brho_h)\bn]\!], & E\in \Omega\cap\mathcal{E}_h\\
		\mu(\varrho^{-1}\bdiv\brho_h) \bn, & E\in \Gamma\cap\mathcal{E}_h\\
	\end{cases};\;
	J_{2,\ell} := \begin{cases}
		\displaystyle \frac{1}{2}\left[\!\left[\left(\kappa_h\brho^d_{h,\infty}\right)\times\bn\right]\!\right], & E\in \Omega\cap\mathcal{E}_h\\
		\left(\kappa_h \brho^d_{h,\infty}\right)\times \bn, & E\in \Gamma\cap\mathcal{E}_h.
	\end{cases}
\end{align*}}
In this context, the parameters $\rho_{R,\infty}$ and $\rho_{E,\infty}$  are specified by
$$
\rho_{R,\infty}:=  (\mu \varrho^{-1})^{-1/2}, \ \    	\rho_{E,\infty} := \left(\sqrt{2}\right)^{-1}\min\left\{(\mu \varrho^{-1})^{-1/2}, \left(\kappa_h/d\right)^{-1/2}\right\}.  
$$

\subsection{Reliability}
{We are now in a position to establish the reliability of the residual-based error estimator introduced above.}
\begin{theorem}
Let $(\kappa,\brho)\in \mathbb{R}^+\times\mathbf{X}_0$ be the solution of \eqref{eq:spectral_continuous}, and let $(\kappa_h,\brho_h)\in \mathbb{R}^+\times\mathbf{X}_{0,h}$ denote the corresponding discrete solution of \eqref{eq:spectral_discrete}. Then there exist positive constants $C$ and $h_0$, independent of $h$ and $\kappa$, such that, for all $h<h_0$, the following estimates hold:
\begin{align*}
||\brho-\brho_h||_{\div,\Omega}&\le C (\eta + \|(\kappa\brho-\kappa_h\brho_h)\|_{0,\Omega}+\|(\brho-\brho_h)\|_{0,\Omega}),\\
|\kappa-\kappa_h|& \le C (\eta^2 + \|(\kappa\brho-\kappa_h\brho_h)\|_{0,\Omega}^2+\|(\brho-\brho_h)\|_{0,\Omega}^2).
\end{align*} 

\end{theorem}
\begin{proof}
The coercivity estimate \eqref{eq:coercive_a} implies that 
\begin{align*}
	\frac{1}{2}\min\{\mu\varrho^{-1}, c_1,1\}||\brho-\brho_h||_{\div,\Omega}^2&\le {a}(\brho-\brho_h,\brho-\brho_h),\\
	&=(\kappa+1) b(\brho,\brho-\brho_h)-(\kappa_h+1) b(\brho_h,\brho-\brho_h)\\
	&\quad+(\kappa_h+1) b(\brho_h,\brho-\brho_h)-{a}(\brho_h,\brho-\brho_h).
\end{align*}
From the definition of the bilinear form $b(\cdot,\cdot)$ and and by applying the Cauchy--Schwarz's inequality, we \cblue{obtain}
\begin{align*}
	|(\kappa+1) b(\brho,\brho-\brho_h)-&(\kappa_h+1) b(\brho_h,\brho-\brho_h)|\\
	&\le C(\|(\kappa\brho-\kappa_h\brho_h)\|_{0,\Omega}+\|(\brho-\brho_h)\|_{0,\Omega})\|\btau\|_{0,\Omega}.
\end{align*}
To derive an estimate for the term 
$(\kappa_h+1) b(\brho_h,\brho-\brho_h)-{a}(\brho_h,\brho-\brho_h)$,
we first augment it with
\begin{align*}
	{a}(\brho_h,I_h(\brho-\brho_h))- (\kappa_h+1) b(\brho_h,I_h(\brho-\brho_h))=0,
\end{align*}
and subsequently employ element-wise integration by parts on $\brho-\brho_h- I_h(\brho-\brho_h)$ in conjunction with the Helmholtz decomposition. Consequently, we obtain 
\begin{multline*}
	(\kappa_h+1) b(\brho_h,\brho-\brho_h-I_h(\brho-\brho_h))-{a}(\brho_h,\brho-\brho_h-I_h(\brho-\brho_h))\\
	=\int_{\O}(\textbf{R}_{1,T} )\cdot (\nabla\bz-I_h(\nabla \bz))
	+\int_{\O} \left(\kappa_h\displaystyle \left\{\brho_h-\frac{\lambda+\mu}{d\lambda+(d+1)\mu}\tr(\brho_h)\mathbf{I} \right\}\right). (\curl (\chi-\chi_h))\\  
	\quad -\sum_{E\in\mathcal{E}(\mathcal{T}_h)\cap\O}\int_E  (\mu(\varrho^{-1}\bdiv\brho_h)\bn)\cdot (\nabla\bz-I_h(\nabla \bz))\\
	\quad-\sum_{E\in\mathcal{E}(\mathcal{T}_h)\cap\Gamma}\int_E(\mu(\varrho^{-1}\bdiv\brho_h)\bn)((\nabla\bz-I_h(\nabla \bz)), 
\end{multline*}
where $\brho-\brho_h= \nabla \bz + \curl \chi$ and $I_h(\brho-\brho_h)= I_h(\nabla \bz)+ \curl \chi_h -d_hI$, for a more detailed discussion see also \cite[Lemma 4.6, Lemma 4.7]{lepe2022posteriori} and \cite[Lemma 4.3]{MR3453481}.
Applying integration by parts together with the Cauchy–Schwarz inequality yields that 
$$
	(\kappa_h+1) b(\brho_h,\brho-\brho_h-I_h(,\brho-\brho_h))-{a}(\brho_h,\brho-\brho_h-I_h(,\brho-\brho_h)) \le C \eta {\|\brho-\brho_h\|_{  \H(\vdiv,\O_f)}}.
$$
The second estimate follows immediately from \eqref{eq:padra} and \eqref{eq:padra3}, in conjunction with the continuity of \(a(\cdot,\cdot)\) and \(b(\cdot,\cdot)\), and the preceding reliability estimate.
\end{proof}
 \subsection{Efficiency}
\cblue{Let us begin by recalling some technical results related to bubble functions. These results are available in, for instance,  \cite{MR1885308,MR3059294}.
\begin{lemma}[Interior bubble functions]
\label{burbujainterior}
For any $T\in \CT_h$, let $\psi_{T}$ be the corresponding interior bubble function.
Then, there holds
\begin{align*}
\|q\|_{0,T}^2&\lesssim \int_{T}\psi_{T} q^2\leq \|q\|_{0,T}^2\qquad \forall q\in \mathbb{P}_k(T),\\
\| q\|_{0,T}&\lesssim \|\psi_{T} q\|_{0,T}+h_K\|\nabla(\psi_{T} q)\|_{0,T}\lesssim \|q\|_{0,T}\qquad \forall q\in \mathbb{P}_k(T),
\end{align*}
where the hidden constants are 
independent of  $h_K$.
\end{lemma}
\begin{lemma}[Facet bubble functions]
\label{burbuja}
For any $T\in \CT_h$ and $\ell\in\CE_{T}$, let $\psi_{\ell}$
be the corresponding facet bubble function. Then, there holds
 \begin{equation*}
\|q\|_{0,\ell}^2\lesssim \int_{\ell}\psi_{\ell} q^2 \leq \|q\|_{0,\ell}^2\qquad
\forall q\in \mathbb{P}_k(\ell).
\end{equation*}
Moreover, for all $q\in\mathbb{P}_k(\ell)$, there exists an extension of  $q\in\mathbb{P}_k(T)$ (again denoted by $q$) such that
 \begin{align*}
h_K^{-1/2}\|\psi_{\ell} q\|_{0,T}+h_K^{1/2}\|\nabla(\psi_{\ell} q)\|_{0,T}&\lesssim \|q\|_{0,\ell},
\end{align*}
where the hidden constants are independent of  $h_K$.
\end{lemma}}

 Firstly, we define $\bv_T= \chi_Th_K^2\rho_R^2\textbf{R}_{1,T}$, where  $\chi_T$ is the element bubble function defined on $T$. Then
\begin{align*}
	h_K^2\|\rho_R\textbf{R}_{1,T}\|_{0,T}^2&\le (\textbf{R}_{1,T} , \bv_T)
	= \underbrace{\int_T  \mu\nabla(\varrho^{-1}\bdiv(\brho_h-\brho))\cblue{:} \bv_T}_{T_1}\\
	&\quad \underbrace{-\int_T  \left(\displaystyle \left\{\kappa_h\brho_h-\kappa \brho-\frac{\lambda+\mu}{d\lambda+(d+1)\mu}(\kappa_h\tr(\brho_h)-\kappa \tr(\brho))\mathbf{I} \right\}\right)\cblue{:} \bv_T}_{T_2}
\end{align*} 
Applying integration by parts to  $T_1$ yields that
\begin{align*}
	T_1&=\int_T  \mu\nabla(\varrho^{-1}\bdiv(\brho_h-\brho))\cblue{:} \bv_T= \int_T   \mu(\varrho^{-1}\bdiv(\brho_h-\brho))\cblue{\cdot}\bdiv\bv_T \\
	&\le \| \mu \varrho^{-1}\bdiv(\brho_h-\brho))\|_{0,T}h_T\|\rho_{R}\textbf{R}_{1,T}\|_{0,T}.
\end{align*}
Finally, to estimate $T_2$, we apply the Cauchy–Schwarz inequality, which yields
\begin{align*}
	T_2&= -\int_T  \left(\displaystyle \left\{\kappa_h\brho_h-\kappa \brho-\frac{\lambda+\mu}{d\lambda+(d+1)\mu}(\kappa_h\tr(\brho_h)-\kappa \tr(\brho))\mathbf{I} \right\}\right)\cblue{:} \bv_T\\
	&\le \|\rho_R(\kappa_h\brho_h-\kappa \brho)\|_{0,T}h_T^2\|\rho_R\textbf{R}_{1,T}\|_{0,T}.
\end{align*}
 Next, we define $$\bv_\ell:=\psi_\ell h_E(\rho_E)^2[\![\mu(\varrho^{-1}\bdiv\brho_h)\bn]\!],$$ where $\psi_\ell$ is the bubble function that satisfies the properties outlined in Lemma \ref{burbuja}. We then proceed to estimate the term $h_E\|\rho_EJ_{\ell}\|_{0,\ell}^2$, leading to  
\begin{align}\label{burbuja11}
	h_E\|\rho_EJ_{1,\ell}\|_{0,\ell}^2&{\le} C([\![\mu(\varrho^{-1}\bdiv\brho_h)\bn]\!], \bv_\ell)_\ell 
\end{align}
Applying integration by parts on $\omega_T$ yields
\begin{align*}
	([\![\mu(\varrho^{-1}\bdiv\brho_h)\bn]\!], \bv_\ell)_\ell =\sum_{T\in\omega_T} &\Big((\mu(\varrho^{-1}\bdiv\brho_h)-\mu(\varrho^{-1}\bdiv\brho)\, \vdiv(\bv_\ell))_T\\
	&+(\textbf{R}_{1,T}, \bv_\ell)_T+\kappa b(\brho,\bv_\ell)-\kappa_h b(\brho_h,\bv_\ell)|\Big).
\end{align*}
Using the Cauchy--Schwarz inequality along with Lemma \ref{burbuja} and combining it with (\ref{burbuja11}), we obtain that
\begin{multline*}
	h_E\|\rho_EJ_{1,\ell}\|_{0,\ell}^2{\le} C ( \|\mu^{1/2}\varrho^{-1/2}\vdiv(\brho-\brho_h)\|_{0,\omega_T}\\
	+h_E\|(\rho_E)(\kappa_h\brho_h -\kappa\brho)\|_{0,\omega_T}) h_E^{1/2}\|\rho_EJ_{1,\ell}\|_{0,\ell}.
\end{multline*}
 The following results follow by an argument analogous to that {of \cite[Lemma 4.11]{MR3453481}}
 \begin{equation*}
 h_K^2\|\rho_R R_{2,T}\|_{0,T}^2\le C \|\brho-\brho_h\|_{0,T}^2
\quad\text{and}\quad 
 h_E\|\rho_EJ_{2,\ell} \|_{0,\ell}^2\le C \|\brho-\brho_h\|_{0,\omega_\ell}^2
 \end{equation*}
 for all $\ell\in \mathcal{E}_h(\Omega)$.
 \begin{theorem}\label{teo:efficiency}
Let $(\kappa,\brho)\in \mathbb{R}^+\times\mathbf{X}_0$ be the solution of \eqref{eq:spectral_continuous}, and let $(\kappa_h,\brho_h)\in \mathbb{R}^+ \times \mathbf{X}_{0,h}$ denote the corresponding discrete solution of \eqref{eq:spectral_discrete}. Then there exist positive constants $C$ and $h_0$, independent of $h$ and $\kappa$, such that, for all $h<h_0$, the following estimate holds
$$
\eta \le C (||\brho-\brho_h||_{\cblue{\bdiv},\Omega}+ \Theta).
$$

\end{theorem}

\section{Numerical results}
\label{sec:numerical-experiments}
  In this section, we present a collection of numerical experiments aimed at assessing the performance of the proposed formulation for the elasticity eigenvalue problem. The discrete schemes were implemented in the \texttt{FEniCSx} environment \cite{barrata2023dolfinx,scroggs2022basix}, and the resulting algebraic eigenvalue problems were solved by means of standard routines from \texttt{SLEPc} \cite{hernandez2005slepc}.
  
  The discrete eigenvalue problem is assembled according to the mixed pseudostress-based formulation introduced in the previous sections. The reported convergence rates for the eigenvalues are obtained by applying a least-squares fit to the corresponding numerical errors over a sequence of successively refined meshes. In order to examine the behavior of the method under different regularity regimes, we consider both two- and three-dimensional test configurations. For several representative examples, we report the smallest computed eigenvalues together with the corresponding experimental orders of convergence.
  
  \cblue{For conciseness of presentation, the numerical tests reported below are carried out with the lowest-order Raviart--Thomas approximation $\boldsymbol{\mathcal{RT}}_0$. This choice keeps the number of numerical tables and figures limited while retaining the main features to be assessed. Higher-order variants are covered by the theoretical analysis and follow the same discrete framework.}

  Throughout this section, \(N\) denotes the characteristic mesh resolution, with \(h \simeq N^{-1}\), whereas \(\mathrm{dof}\) stands for the total number of degrees of freedom associated with the corresponding finite element discretization.
  
  For each eigenvalue \(\kappa_i\), we define the approximation error by
  $
  \err(\kappa_i):=\big|\kappa_{h,i}-\kappa_i\big|,
  $
  where \(\kappa_i\) denotes the corresponding extrapolated value. We recall that $\sqrt{\kappa_i}$ (resp. $\sqrt{\kappa_{h,i}}$) corresponds to the i-th eigenfrequency (resp. discrete eigenfrequency). In turn, the effectivity index associated with the estimator \(\eta\) and the discrete eigenvalue \(\kappa_{h,i}\) is defined as
  \[
  \eff(\kappa_i):=\frac{\err(\kappa_i)}{\eta^2}.
  \]
  
  To perform the adaptive finite element experiments, we generate a sequence of nested conforming triangulations according to the standard loop
  \begin{center}
  	\emph{solve} \(\rightarrow\) \emph{estimate} \(\rightarrow\) \emph{mark} \(\rightarrow\) \emph{refine},
  \end{center}
  following \cite{verfuhrt1996}. 
%
%
%
%
  We remark that to compute the velocity fields $\bu_{h,i}$ or the stress tensors $\bsig_{h,i}$ we resort to \eqref{def:elast_system} and \cite[Section 2]{MR3860570} in order to have the postprocessing formulas
  $$
  	\bu_{h,i}=-\frac{\bdiv \brho_{h,i}}{\kappa_{h,i}}, \quad \bsig_{h,i}:=\brho_{h,i} + \brho_{h,i}^\texttt{t} -\left(\frac{\lambda+2\mu}{d\lambda+(d+1)\mu}\right)\tr(\brho_{h,i})\mathbf{I}.
  $$ 
  In the limit case, we have $\ds \bsig_{h,i}:=\brho_{h,i} +\brho_{h,i}^\texttt{t} -\frac{1}{d}\tr(\brho_{h,i})\mathbf{I}$.
  
 \subsection{Convergence on a square domain} \label{sec:2D-square}

In this test we study the convergence of the proposed scheme on a two-dimensional convex geometry. We consider the square domain $\Omega:=(0,1)^2$. The first four computed eigenvalues are studied for different values of the Poisson ratio $\nu$.

The results reported in Table~\ref{table-square2D-RT0} show that the computed eigenvalues converge with rates close to $\mathcal{O}(h^2)$ for all tested values of $\nu$. These rates are consistent with the estimate $\mathcal{O}(h^{2\min\{r,k\}})$ derived in Theorem~\ref{thm:double_order}. The behavior is maintained as $\nu \to 0.5$, and no degradation in the convergence order is observed. The extrapolated values are in agreement with the reference values, which confirms the accuracy of the method in the nearly incompressible regime.


\cblue{The same convergence behavior is observed in three dimensions. Uniform refinement experiments on the unit cube with the lowest-order Raviart--Thomas discretization recover the expected second-order convergence for the eigenvalues for all tested values of the Poisson ratio, with no deterioration as $\nu\to0.5$. These results are omitted for brevity.}

\begin{table}[hbt!]
 	\centering 
 	{\footnotesize
 		\begin{center}
 			\caption{Test \ref{sec:2D-square}. Convergence behavior of the first four computed eigenvalues when using different values of $\nu$.}
 			\begin{tabular}{|c|c c c c |c| c|c|}
 				\toprule
 				\hline
 				$\nu$&$N=40$           &  $N=50$         &  $N=60$         & $N=70$ & Order & $\sqrt{\widehat{\kappa}_{\text{extr}}}$ & Ref. \cite{MR3962898}  \\ 
 				\hline
 				\hline
 				\multirow{5}{0.05\linewidth}{0.35}
 				&   4.190383  &   4.191342  &   4.191870  &   4.192192  & 1.96 &   4.193108 & 4.19311  \\
 				&   4.190383  &   4.191342  &   4.191870  &   4.192192  & 1.96 &   4.193108 & 4.19311  \\
 				&   4.371890  &   4.371991  &   4.372046  &   4.372079  & 1.97 &   4.372174 &4.37217 \\
 				&   5.928247  &   5.929946  &   5.930891  &   5.931470  & 1.91 &   5.933173 &5.93318  \\
 				\hline
 				\multirow{5}{0.05\linewidth}{0.49}
 				&   4.188000  &   4.188205  &   4.188318  &   4.188386  & 1.97 &   4.188580 &    4.18858  \\
 				&   5.513792  &   5.515156  &   5.515897  &   5.516344  & 2.02 &   5.517574 &    5.51758  \\
 				&   5.513792  &   5.515156  &   5.515897  &   5.516344  & 2.02 &   5.517574 &    5.51758  \\
 				&   6.539852  &   6.541112  &   6.541798  &   6.542212  & 2.01 &   6.543359 &    6.54337  \\
 				\hline
 				\multirow{5}{0.05\linewidth}{0.50}
 				&   4.176499  &   4.176715  &   4.176834  &   4.176906  & 1.97 &   4.177111 &    4.17711  \\
 				&   5.538282  &   5.539438  &   5.540065  &   5.540444  & 2.02 &   5.541486 &    5.54149  \\
 				&   5.538282  &   5.539438  &   5.540065  &   5.540444  & 2.02 &   5.541486 &    5.54149  \\
 				&   6.533942  &   6.535153  &   6.535813  &   6.536211  & 2.01 &   6.537313 &    6.53732  \\
 				\hline
 			\end{tabular}
 	\end{center}}
 	\smallskip
 	\label{table-square2D-RT0}
 \end{table}

\subsection{A posteriori results on a 2D non-convex geometry}
\label{subsec:2D-circle}

In this test we assess the adaptive strategy on a two-dimensional non-convex geometry. Let
\[
\Omega_R:=\left\{(x,y)\in\mathbb{R}^2:\ x^2+y^2<1\right\}
\setminus
\left\{(x,y)\in\mathbb{R}^2:\ x>0,\ y<0\right\}.
\]
Thus, $\Omega_R$ is a circular sector of aperture $3\pi/2$ with a re-entrant corner at the origin. The adaptive tests are performed for the first eigenpair and for different values of the Poisson ratio $\nu$. The purpose of this experiment is to verify whether the residual estimator detects the singular region and whether the adaptive method recovers the expected rate in terms of degrees of freedom. For this type of geometry, there is no closed-form eigenvalues. Hence, the errors are computed with respect to extrapolated reference values: $\kappa_{1}=6.79322 (\nu=0.35)$,  $\kappa_{1}=12.58428 (\nu=0.49)$,  $\kappa_{1}=12.65791 ( \nu=0.50)$.

The error curves presented in Figure~\ref{fig:adapt-error-circle} shows that err$(\kappa_1) \simeq C\,\mathrm{dof}^{-1}$, which corresponds to the optimal rate for adaptive methods applied to eigenvalue problems with singular eigenfunctions. This rate is observed for all values of $\nu$. The slope of the curves matches the reference line $\mathcal{O}(\mathrm{dof}^{-1})$ over several refinement levels, indicating that the adaptive procedure compensates for the reduced regularity induced by the re-entrant corner. No reduction in the convergence rate is detected as $\nu \to 0.5$.

The efficiency indices shown in Figure~\ref{fig:adapt-error-circle} remain uniformly bounded with respect to the number of degrees of freedom. The behavior is consistent for all values of $\nu$, and no significant degradation of the estimator is detected in the nearly incompressible regime. The efficiency curves remain almost constant, indicating a stable estimator. This behavior is in agreement with the reliability result established in Theorem~\ref{teo:efficiency}, which ensures that the estimator provides an upper bound for the eigenvalue error up to a multiplicative constant.

Figure~\ref{fig:adapt-mesh-circle-nu} shows that the adaptive meshes depend on the value of the Poisson ratio. While for $\nu=0.35$ the refinement is mainly concentrated near the re-entrant corner, for $\nu=0.50$ a stronger and more localized refinement is observed in its vicinity. This behavior reflects the change in the local error distribution in the nearly incompressible regime. Consequently, the error estimator drives a more concentrated mesh around the critical region. In both cases, the refinement identifies the singular zone, which confirms the robustness of the adaptive strategy with respect to variations in $\nu$.

We end the experiment by presenting the studied lowest mode. The eigenfunctions displayed in Figure~\ref{fig:circle-eigenfunctions-rho1} show a localization of the gradients near the re-entrant corner, which is consistent with the singular behavior of elasticity solutions in non-smooth domains. The pseudostress magnitude exhibits a peak at the corner, while the displacement field presents smoother but still concentrated variations in the same region. As $\nu$ increases, the amplitude and distribution of the displacement field change, but the singularity-driven pattern remains governed by the geometry.

\begin{figure}[!hbt]\centering
	\begin{minipage}{0.4\linewidth}\centering
		\includegraphics[width=\linewidth]{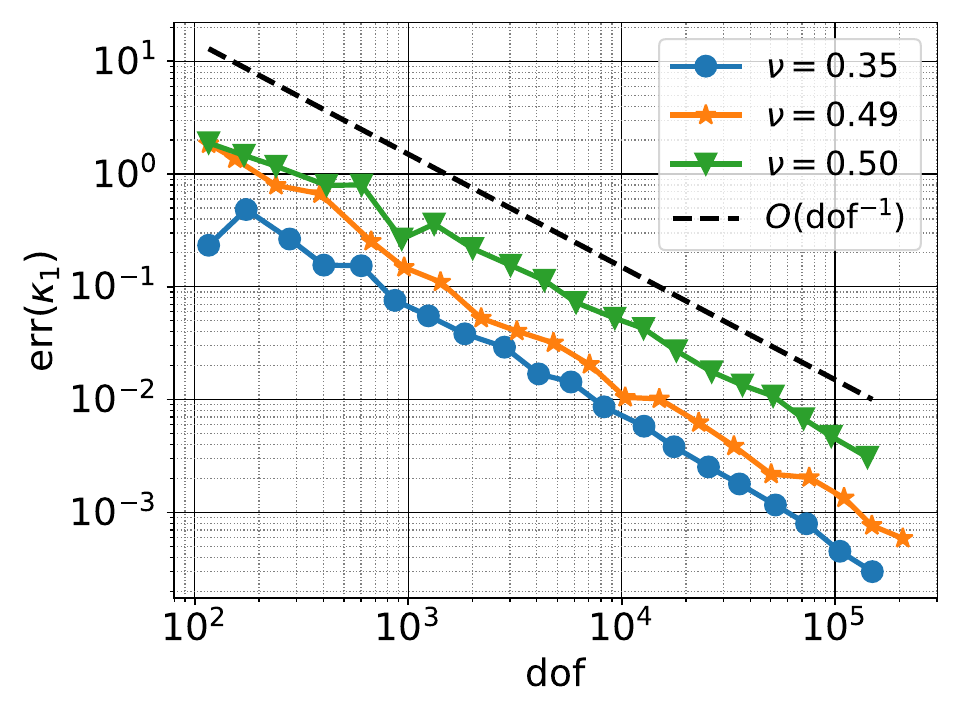}
	\end{minipage}
	\begin{minipage}{0.4\linewidth}\centering
		\includegraphics[width=\linewidth]{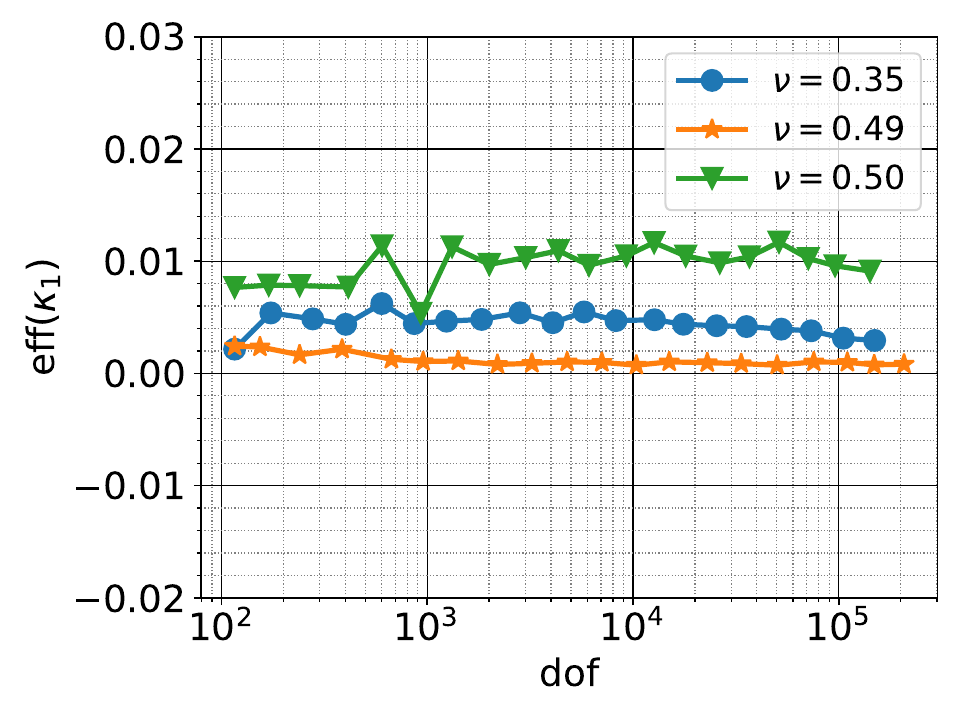}
	\end{minipage}
	\caption{Test~\ref{subsec:2D-circle}. Error and efficiency curves for the first computed eigenvalue under adaptive refinements.}
	\label{fig:adapt-error-circle}
\end{figure}

{\small\begin{figure}[!hbt]\centering
\begin{minipage}{0.32\linewidth}\centering
		{Initial mesh}\\
		\includegraphics[scale=0.07,trim=20cm 0cm 20cm 0cm, clip]{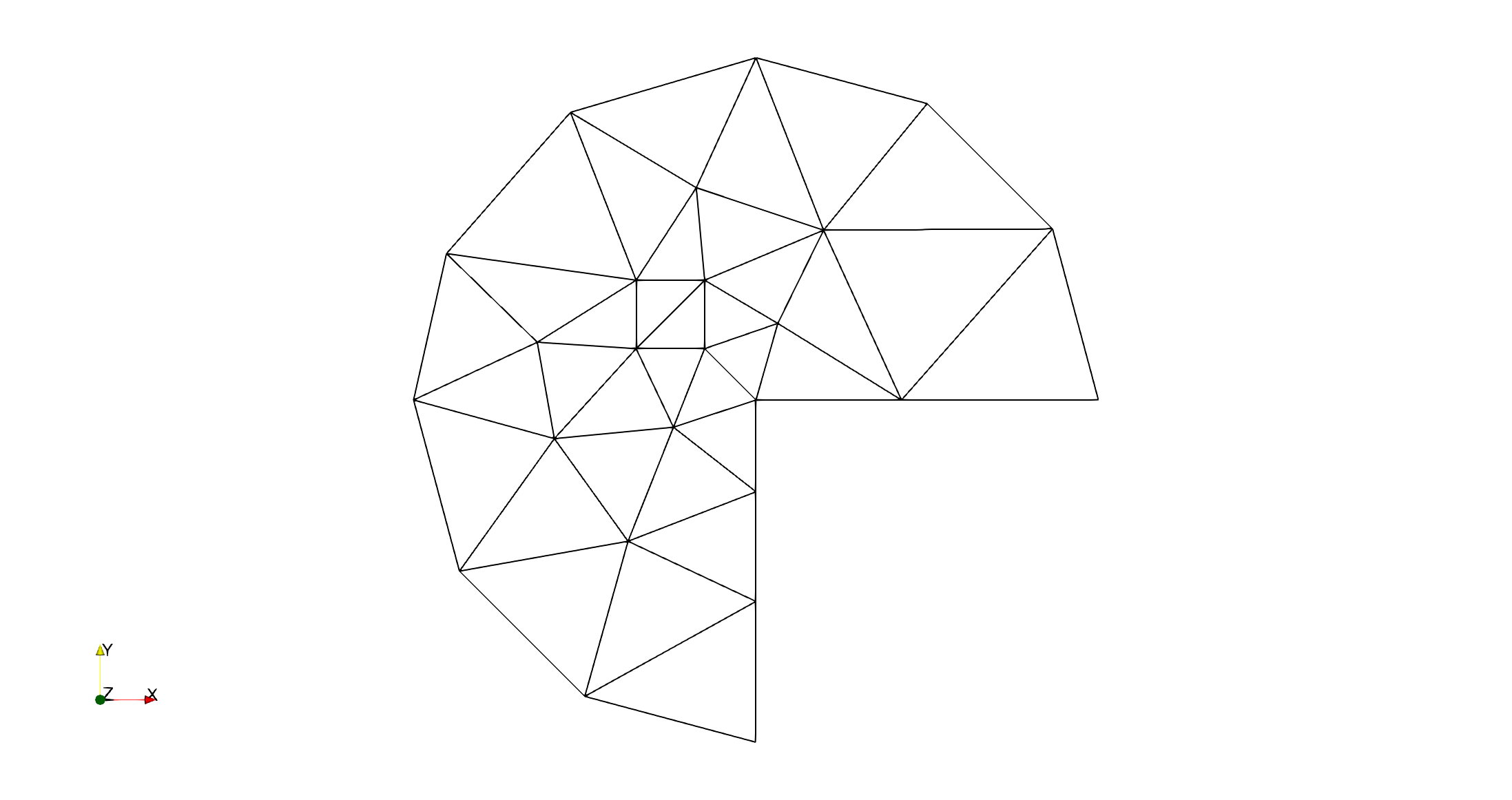}\\
		{$\mathrm{dof}$=116}
	\end{minipage}
	\begin{minipage}{0.32\linewidth}\centering
		{$\nu=0.35$}\\
		\includegraphics[scale=0.07,trim=20cm 0cm 20cm 0cm, clip]{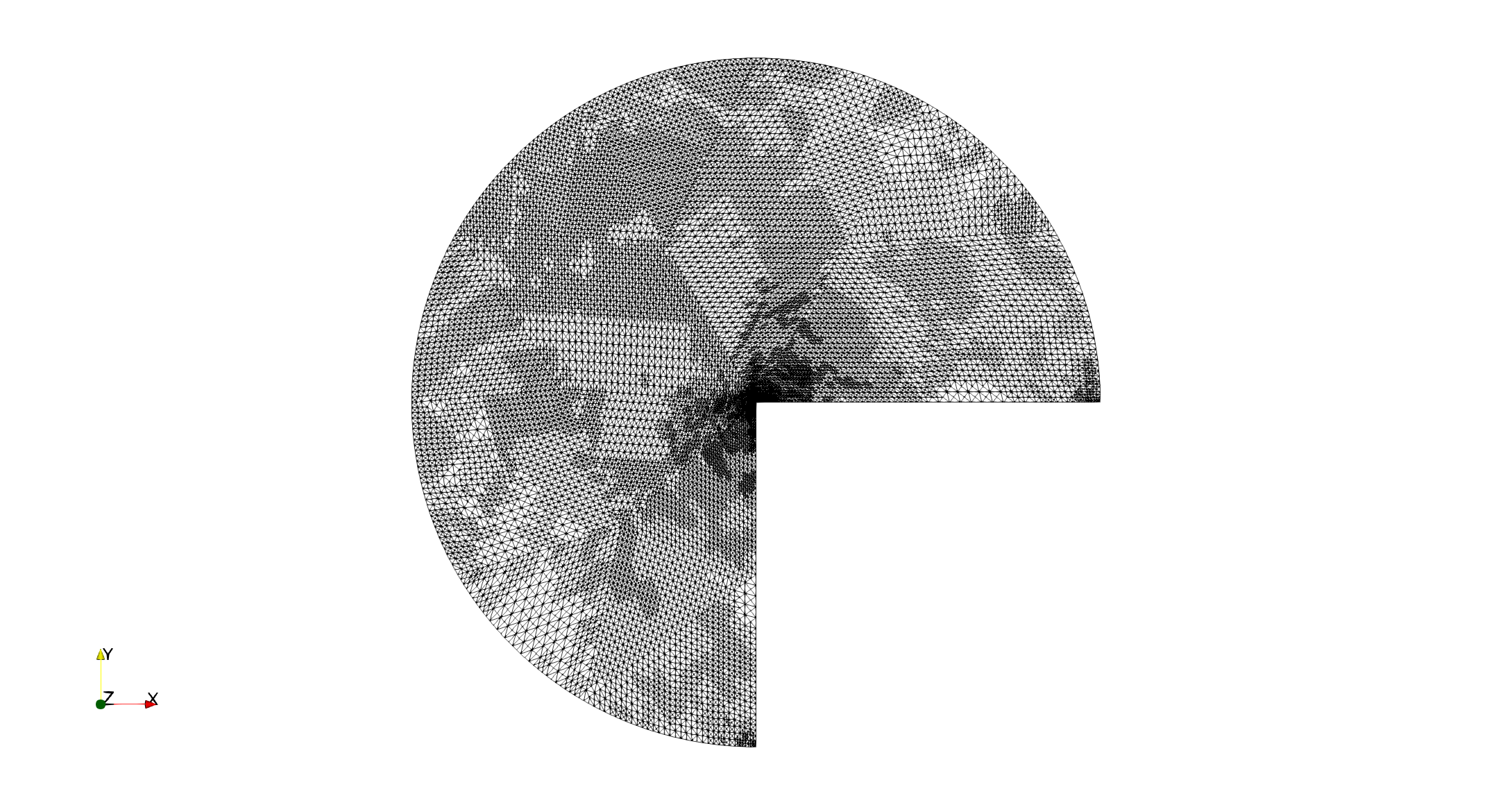}\\
		{$\mathrm{dof}$=149472}
	\end{minipage}
	\begin{minipage}{0.32\linewidth}\centering
		{$\nu=0.50$}\\
		\includegraphics[scale=0.07,trim=20cm 0cm 20cm 0cm, clip]{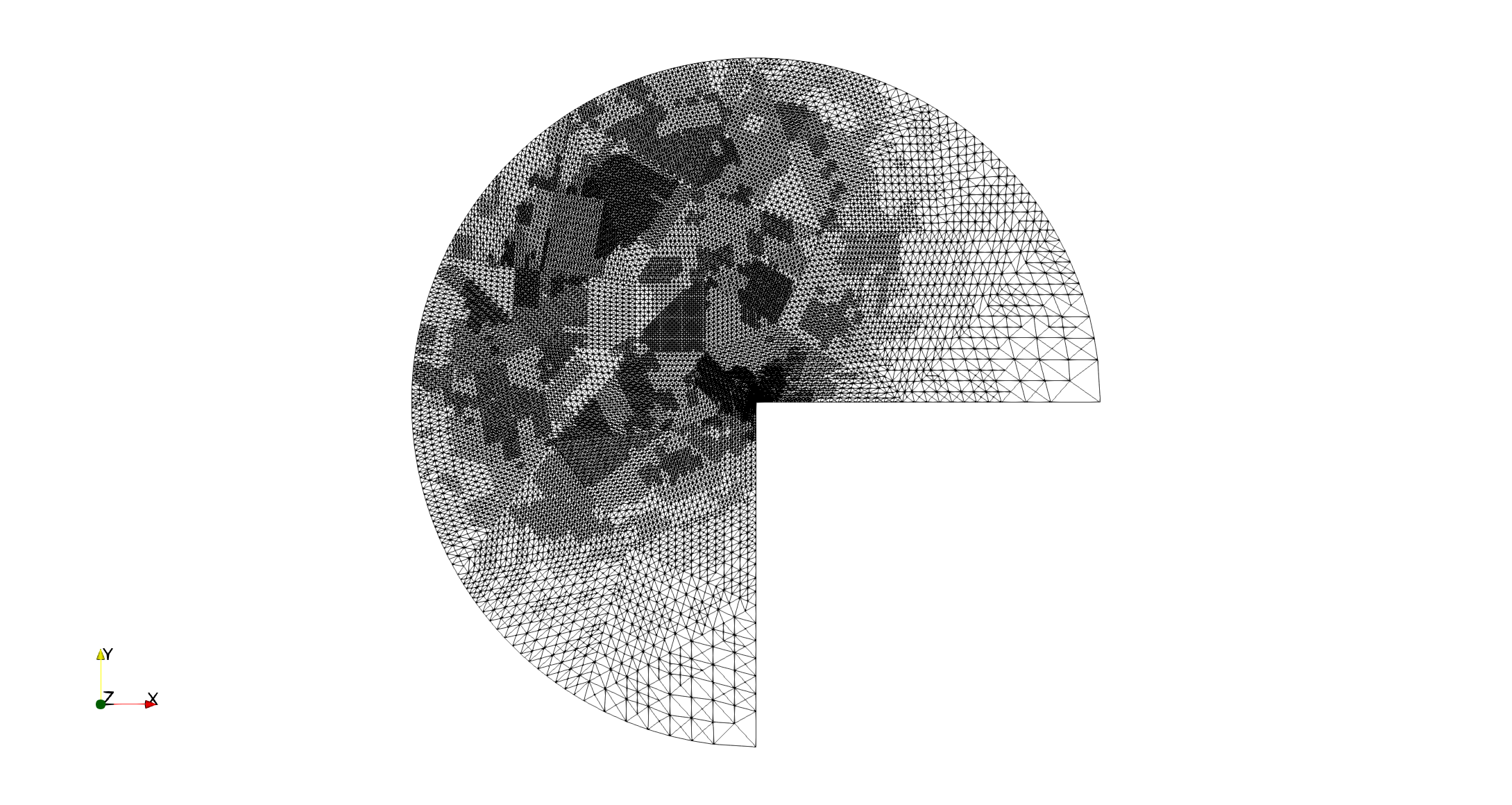}\\
		{$\mathrm{dof}$=141898}
	\end{minipage}
	\caption{Test~\ref{subsec:2D-circle}. Initial and final meshes for the first eigenpair at the last adaptive iteration.}
	\label{fig:adapt-mesh-circle-nu}
\end{figure}}

{\begin{figure}[!hbt]\centering
\begin{minipage}{0.32\linewidth}\centering
		{$|\brho_{h,1}|, \nu=0.35$}\\
		\includegraphics[scale=0.07,trim=20cm 0cm 20cm 0cm, clip]{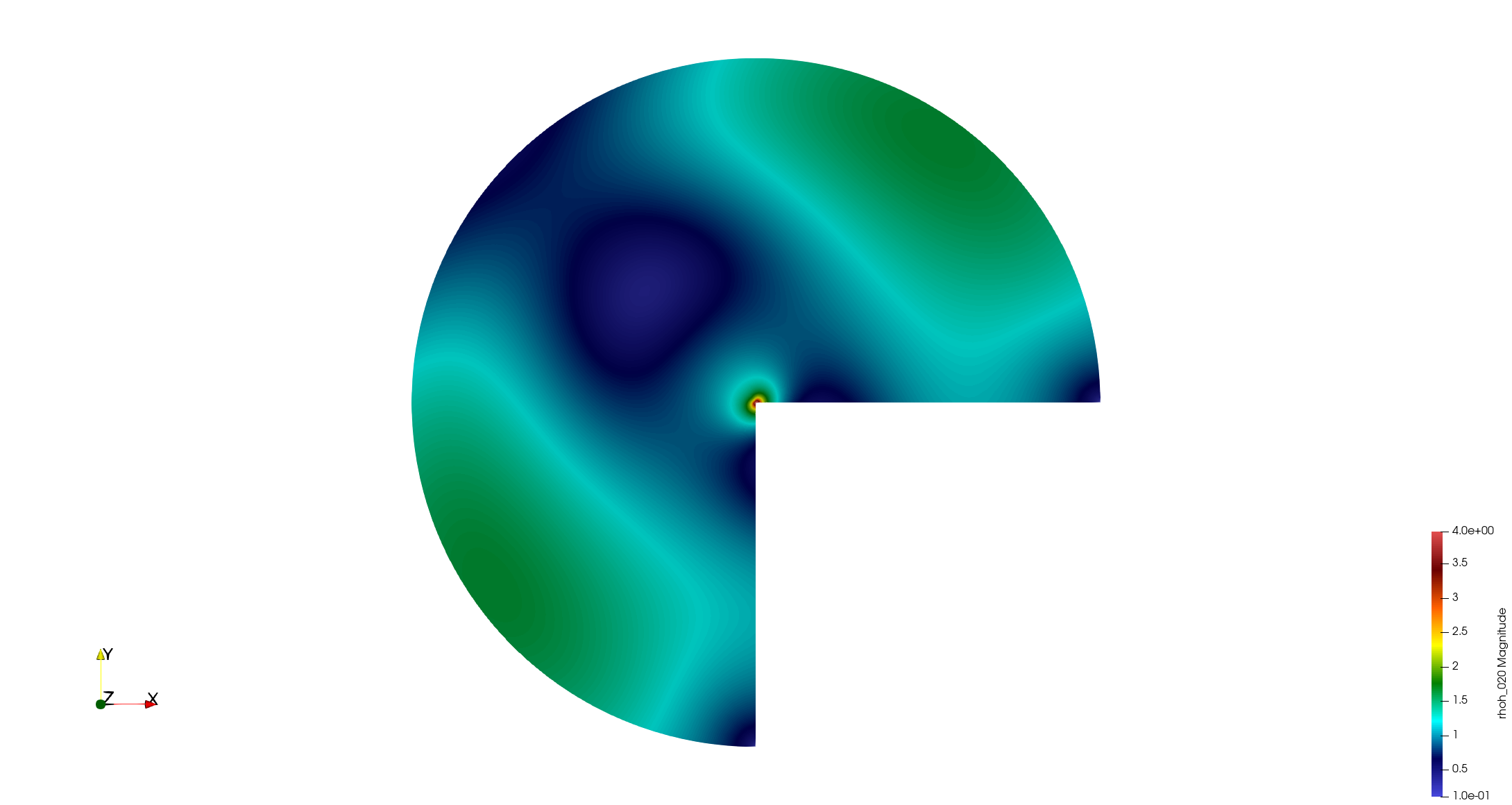}\\
	\end{minipage}
	\begin{minipage}{0.32\linewidth}\centering
		{$|\bsig_{h,1}|, \nu=0.35$}\\
		\includegraphics[scale=0.07,trim=20cm 0cm 20cm 0cm, clip]{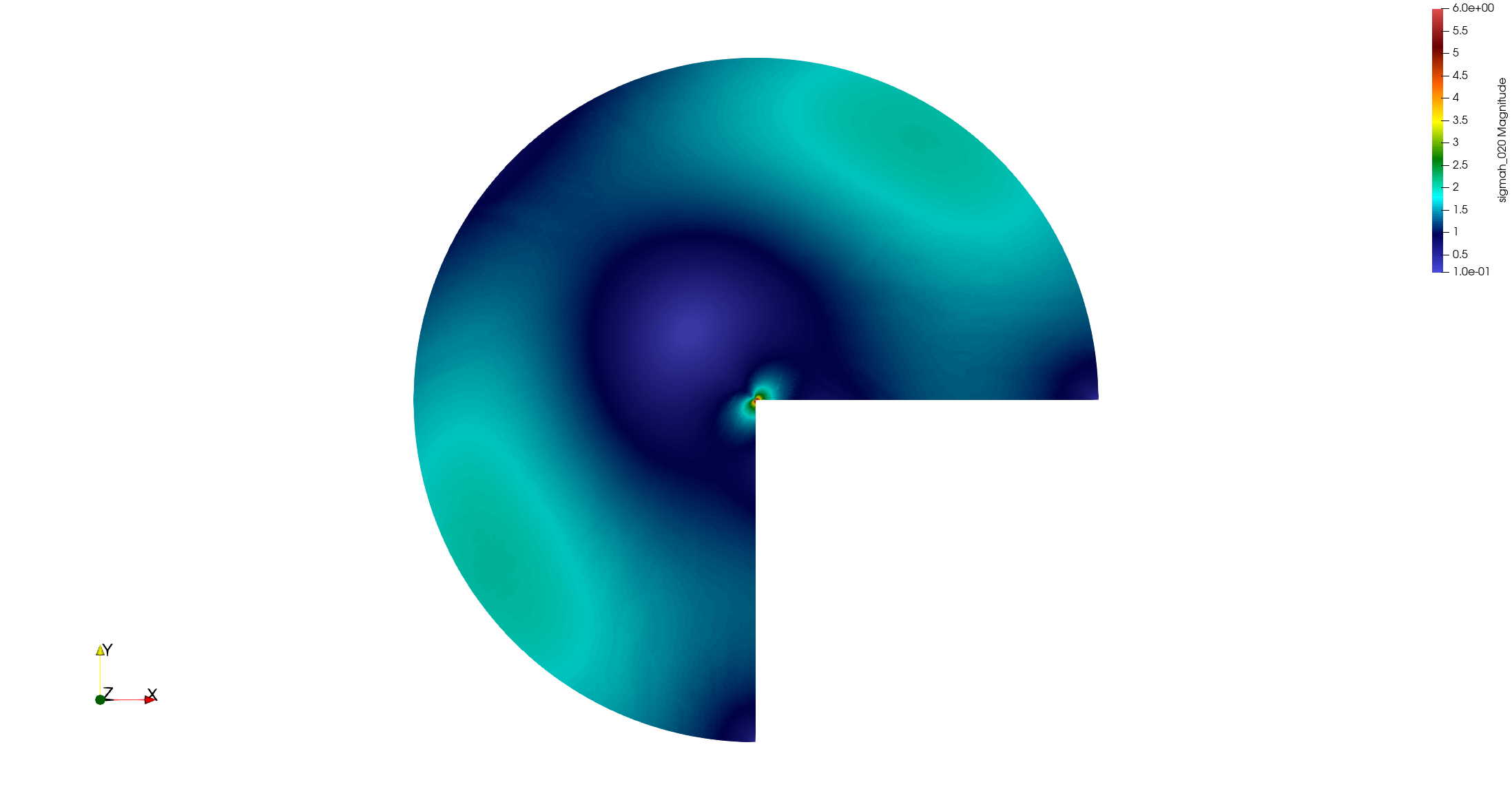}\\
	\end{minipage}
	\begin{minipage}{0.32\linewidth}\centering
		{$|\bu_{h,1}|, \nu=0.35$}\\
		\includegraphics[scale=0.07,trim=20cm 0cm 20cm 0cm, clip]{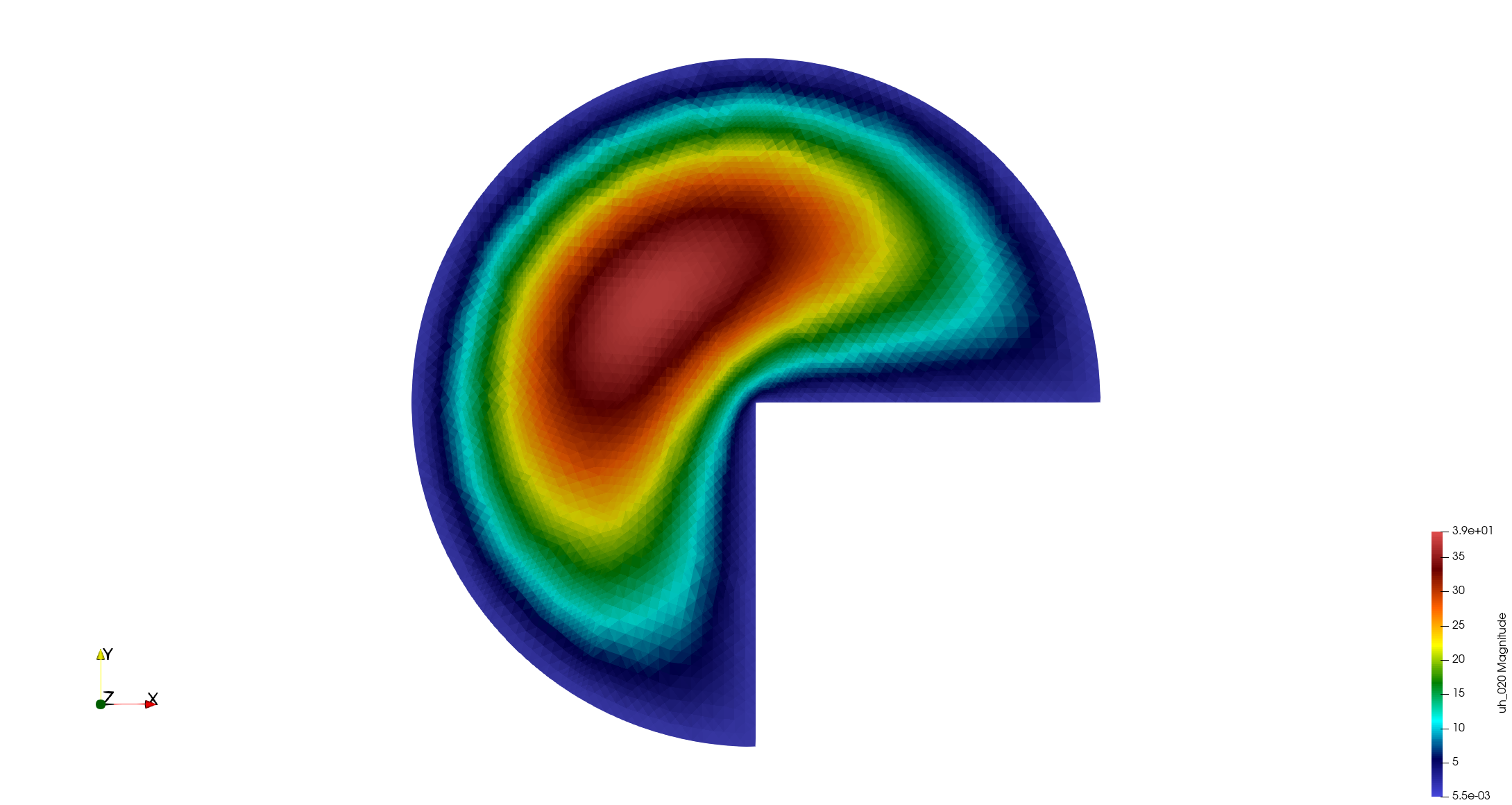}\\
	\end{minipage}\\
	\begin{minipage}{0.32\linewidth}\centering
		{$|\brho_{h,1}|, \nu=0.50$}\\
		\includegraphics[scale=0.07,trim=20cm 0cm 20cm 0cm, clip]{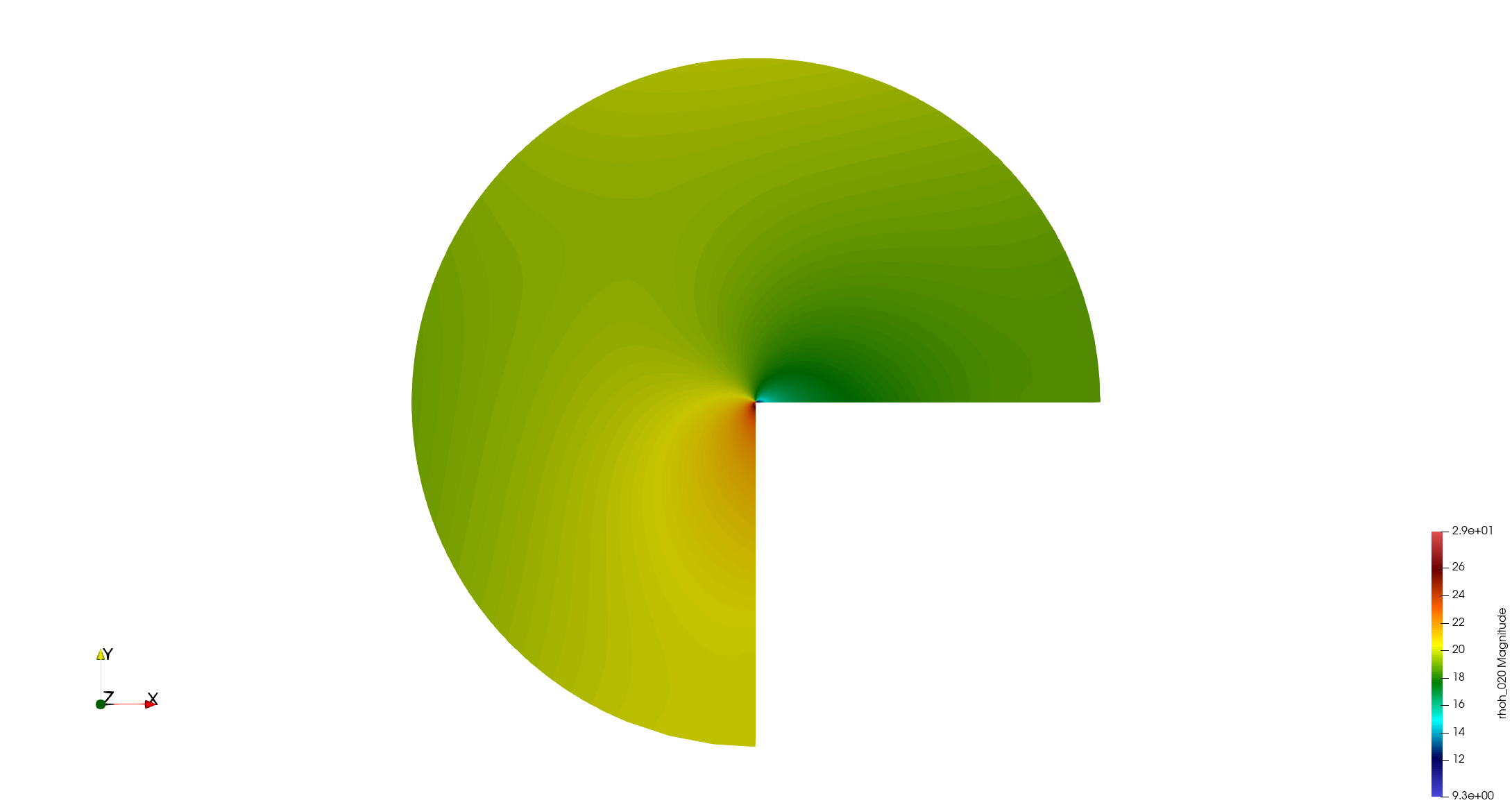}\\
	\end{minipage}
	\begin{minipage}{0.32\linewidth}\centering
		{$|\bsig_{h,1}|, \nu=0.50$}\\
		\includegraphics[scale=0.07,trim=20cm 0cm 20cm 0cm, clip]{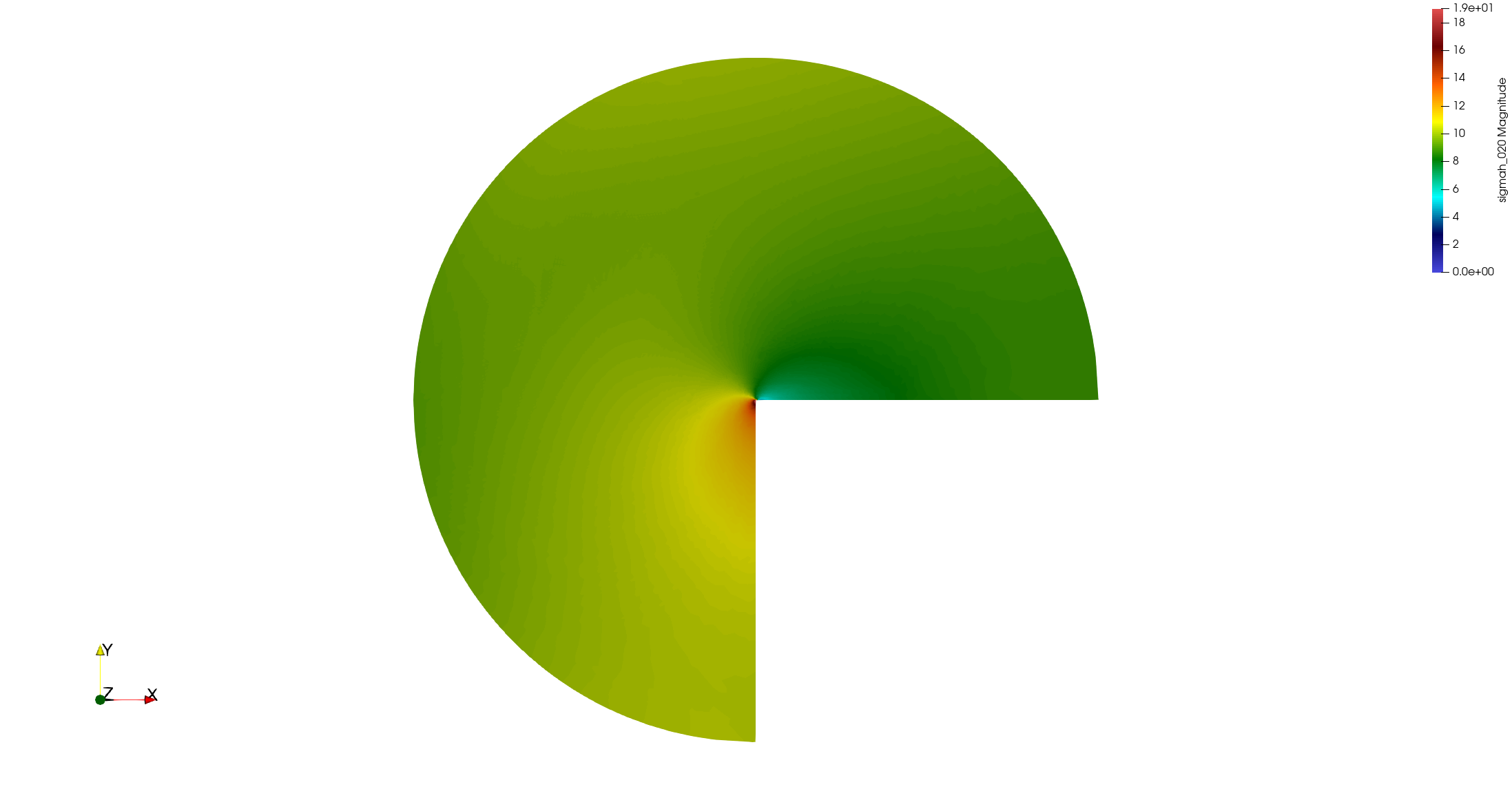}\\
	\end{minipage}
	\begin{minipage}{0.32\linewidth}\centering
		{$|\bu_{h,1}|, \nu=0.50$}\\
		\includegraphics[scale=0.07,trim=20cm 0cm 20cm 0cm, clip]{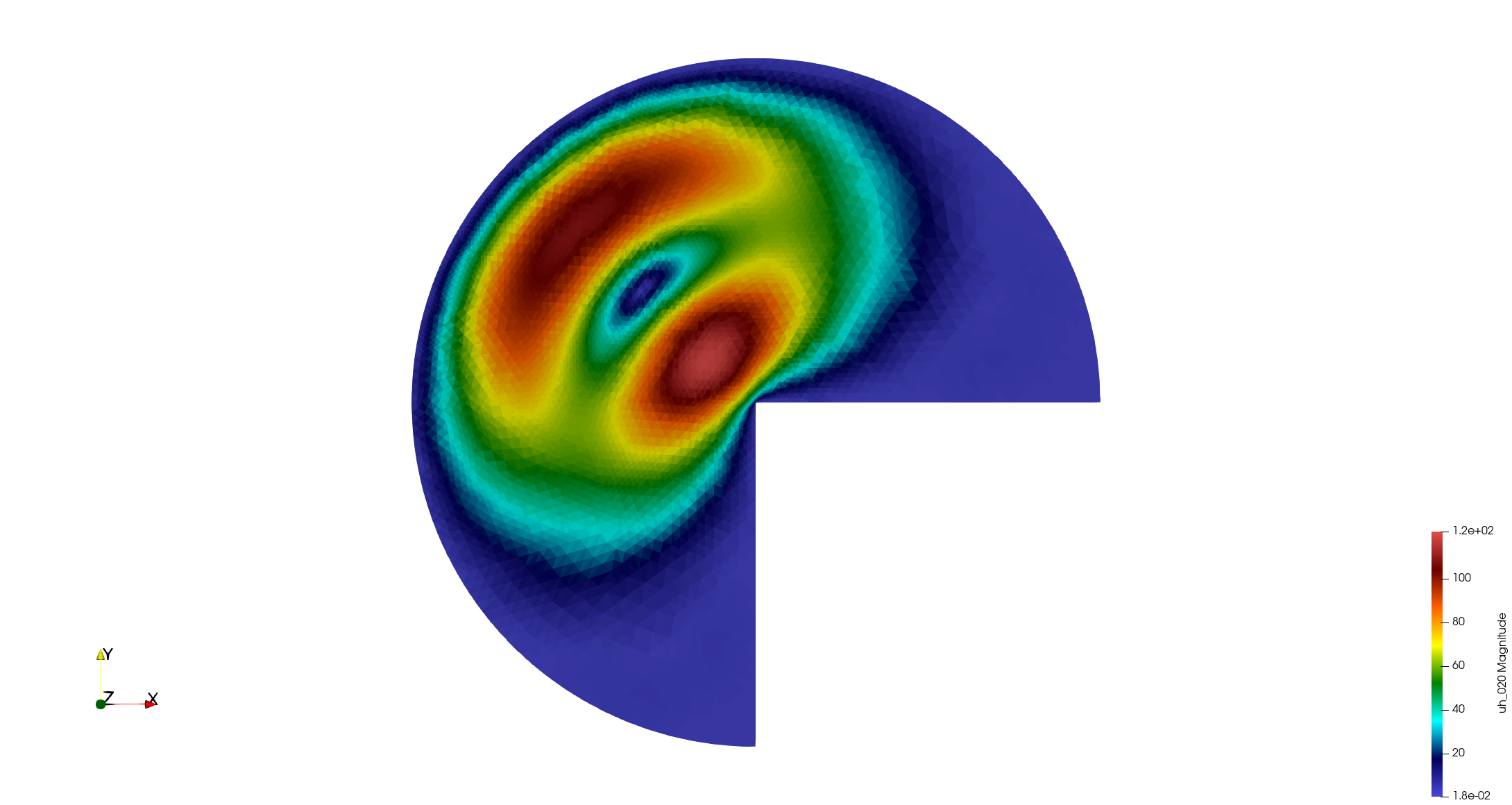}\\
	\end{minipage}
	\caption{Test~\ref{subsec:2D-circle}. Magnitudes of the pseudostress and the corresponding post-processed stress and velocity fields for the fourth eigenfrequency $\sqrt{\kappa_{h,1}}$ with different values of $\nu$.}
	\label{fig:circle-eigenfunctions-rho1}
\end{figure}}

\subsection{Adaptive refinement on a three-dimensional reentrant cube}
\label{sec:adaptive-cube-notch}

As a final experiment, we consider the adaptive approximation of the elasticity eigenvalue problem on the three-dimensional non-convex domain given by
$
\Omega:=(0,1)\times(0,1)\times(0,0.77)\setminus \overline{\mathcal{N}},
$
where the notch $\mathcal{N}$ is the triangular prism
$$
\mathcal{N}:=
\left\{(x,y,z)\in\mathbb{R}^3:\ (x,y)\in T,\ 0.10<z<0.65\right\},
$$
and $T$ is the triangular region with vertices $(1,0.30)$,$(1,0.75)$, and $(0.45,0.48)$. Thus, the notch enters from the face $x=1$ and generates reentrant edges inside the cube (see Figure \ref{fig:cube-notch-adapt-mesh}). The singular geometry produced by the internal notch gives rise to reduced regularity of the eigenfunctions, making this benchmark particularly suitable for assessing the robustness of the proposed adaptive strategy. The computations are performed using the lowest-order Raviart--Thomas discretization ($k=0$), and the adaptive algorithm is driven by the residual estimator introduced in Section~\ref{sec:apost}. The first eigenpair is approximated for $\nu=0.35$, $\nu=0.49$, and $\nu=0.50$. The experiment is performed with the lowest-order Raviart--Thomas discretization. Since no closed-form eigenvalues are available for this geometry, the errors are computed with respect to extrapolated reference values: $\kappa_{1}=28.48844 (\nu=0.35)$,  $\kappa_{1}=36.02493 (\nu=0.49)$,  $\kappa_{1}=35.95801 ( \nu=0.50)$.

\begin{figure}[!hbt]\centering
	\begin{minipage}{0.40\linewidth}\centering
		\includegraphics[width=\linewidth]{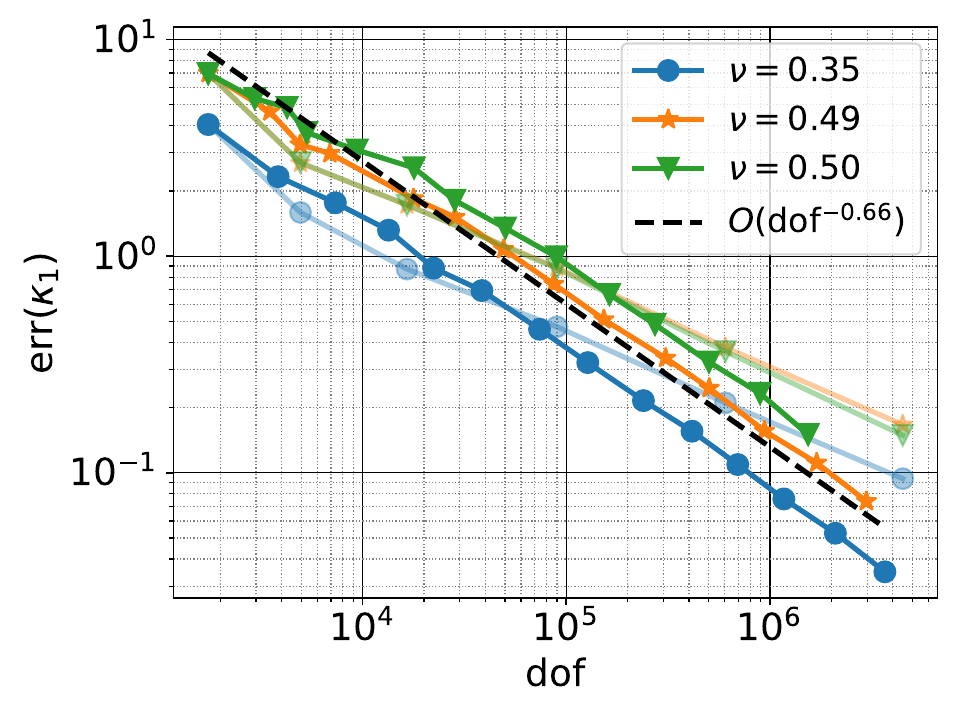}
	\end{minipage}
	\begin{minipage}{0.40\linewidth}\centering
		\includegraphics[width=\linewidth]{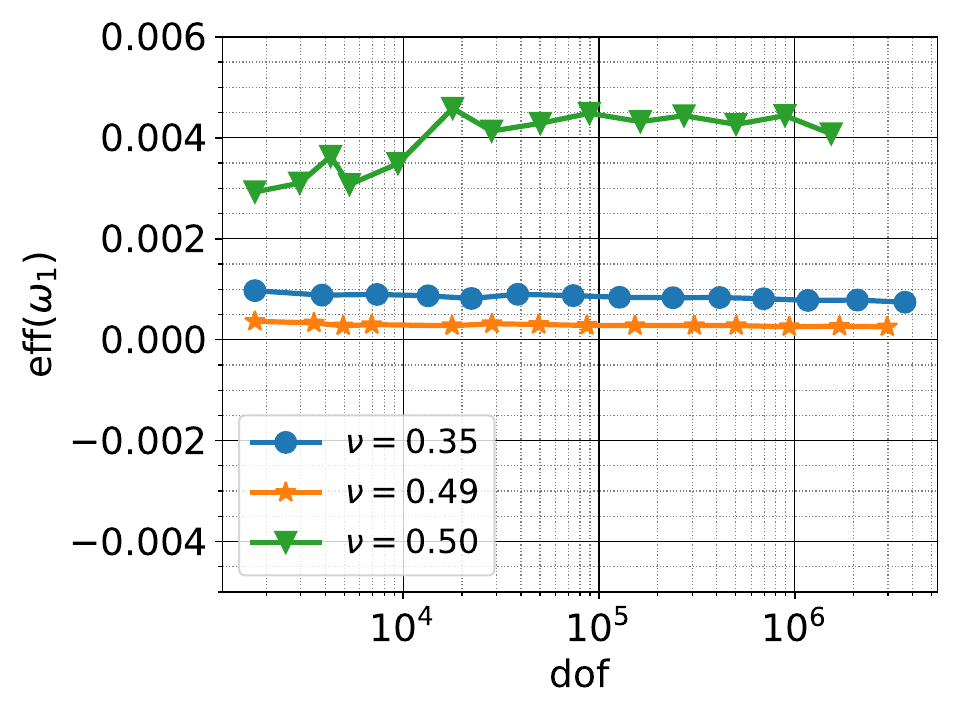}
	\end{minipage}
	\caption{Test~\ref{sec:adaptive-cube-notch}. Error curves and efficiency indices for the first computed eigenfrequency using lowest-order $\boldsymbol{\mathcal{RT}}$ elements.}
	\label{fig:adapt-error-eff-cube-notch}
\end{figure}

The error curves in Figure~\ref{fig:adapt-error-eff-cube-notch} show an asymptotic behavior of approximately $\mathcal{O}(\mathrm{dof}^{-0.67})$, $\mathcal{O}(\mathrm{dof}^{-0.66})$, and $\mathcal{O}(\mathrm{dof}^{-0.66})$ for $\nu=0.35$, $\nu=0.49$, and $\nu=0.50$, respectively. Hence, the adaptive algorithm recovers the expected rate $\mathcal{O}(\mathrm{dof}^{-2/3})$. The experimental rate is preserved in the nearly incompressible regime, which indicates that the adaptive method remains locking-free. The efficiency indices remain bounded throughout refinement, giving $\texttt{eff}(\kappa_1)\in[7.4180\cdot10^{-4},9.7062\cdot10^{-4}]$ for $\nu=0.35$, $\texttt{eff}(\kappa_1)\in[2.5373\cdot10^{-4},3.6350\cdot10^{-4}]$ for $\nu=0.49$, and $\texttt{eff}(\kappa_1)\in[2.9256\cdot10^{-3},4.5813\cdot10^{-3}]$ for $\nu=0.50$. The larger values for $\nu=0.50$ indicate a change in the estimator scaling in the incompressible limit, but the curves remain bounded and do not deteriorate with the number of degrees of freedom.

{\small\begin{figure}[!hbt]\centering
\begin{minipage}{0.32\linewidth}\centering
		{Initial mesh}\\
		\includegraphics[scale=0.08,trim=16cm 0cm 19cm 5cm,clip]{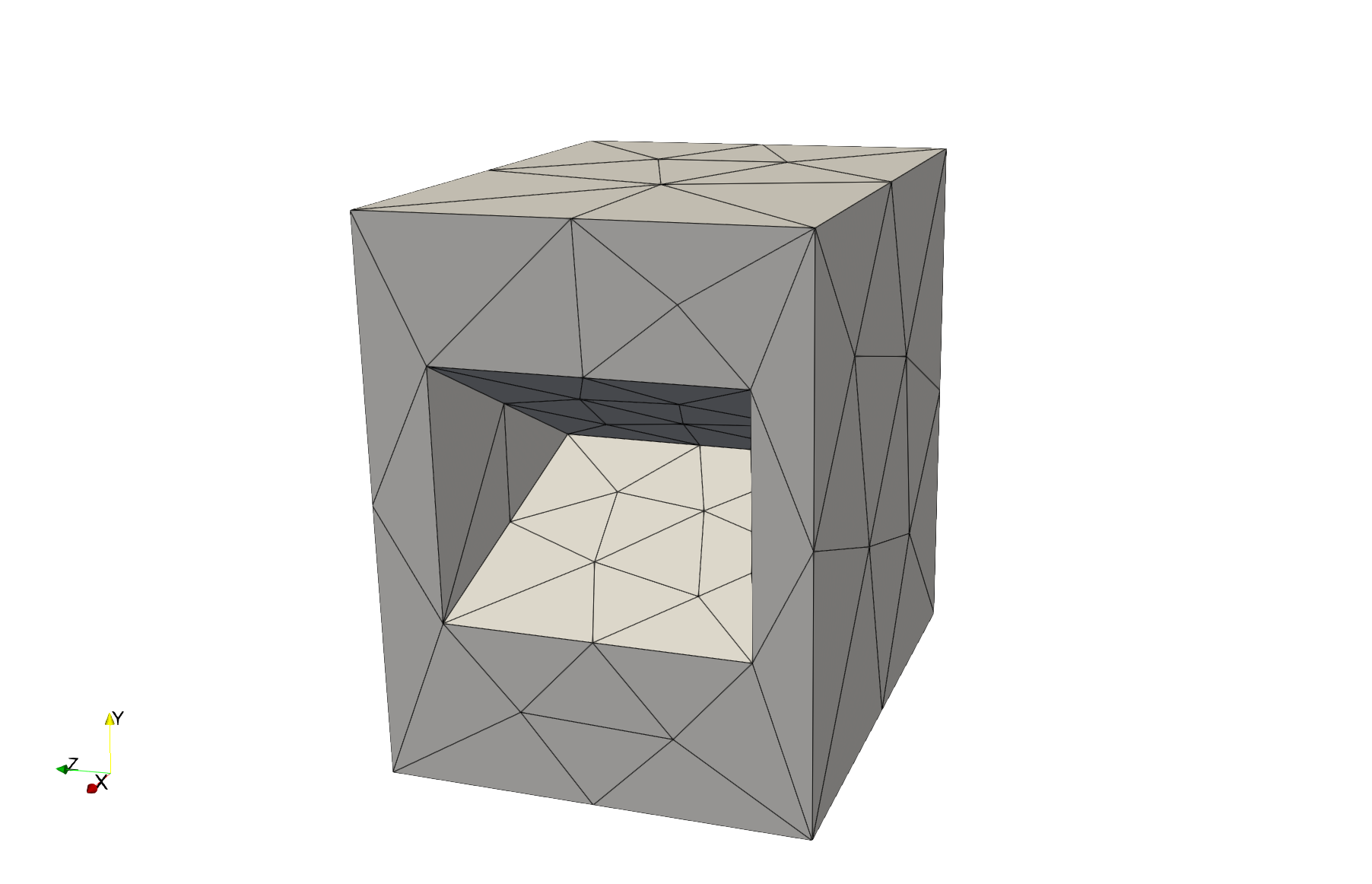}\\
		{$\mathrm{dof}$=1749}
	\end{minipage}
	\begin{minipage}{0.32\linewidth}\centering
		{$\nu=0.35$}\\
		\includegraphics[scale=0.08,trim=16cm 0cm 19cm 5cm,clip]{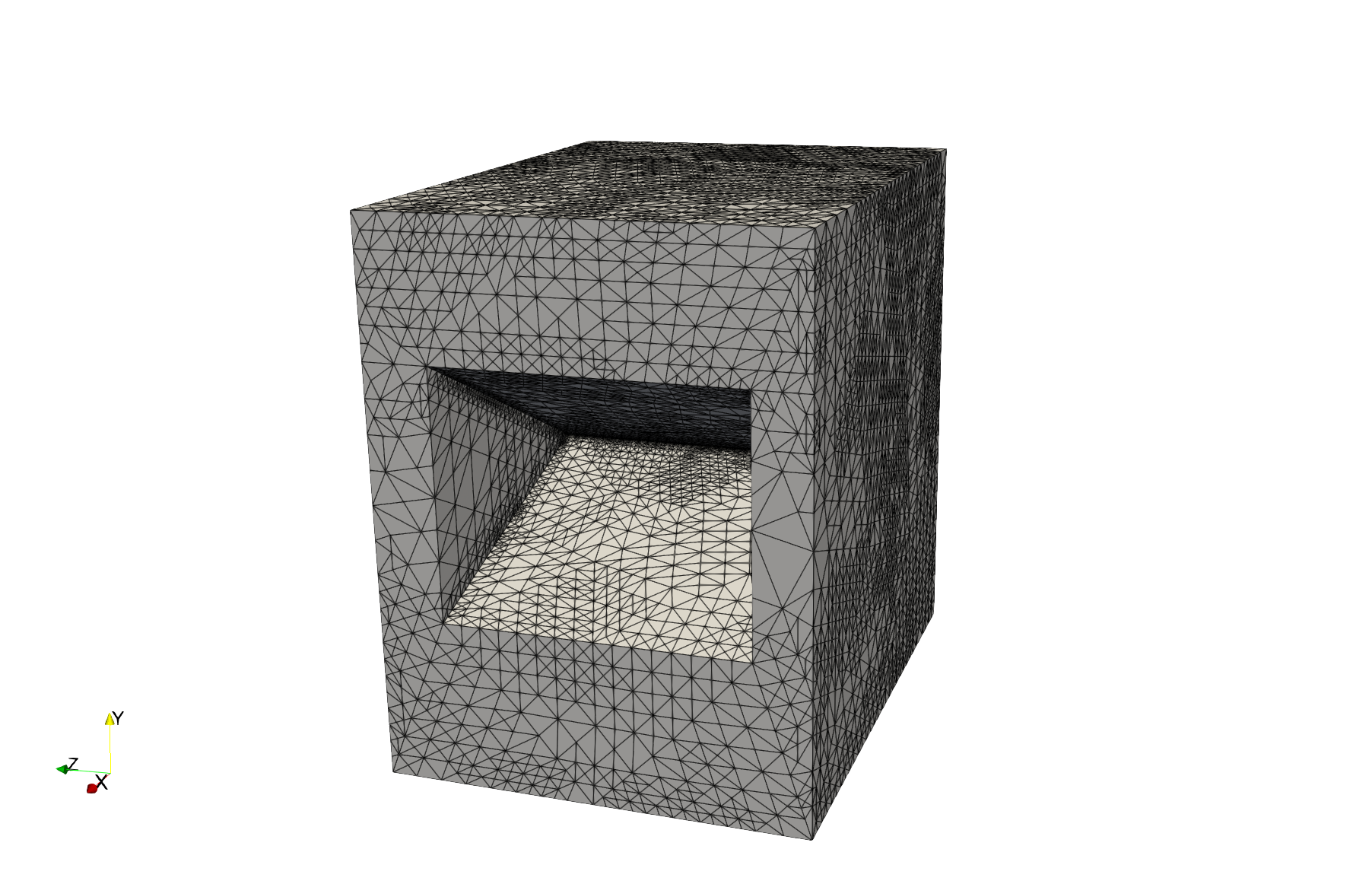}\\
		{$\mathrm{dof}$=3661050}
	\end{minipage}
	\begin{minipage}{0.32\linewidth}\centering
		{$\nu=0.50$}\\
		\includegraphics[scale=0.08,trim=16cm 0cm 19cm 5cm,clip]{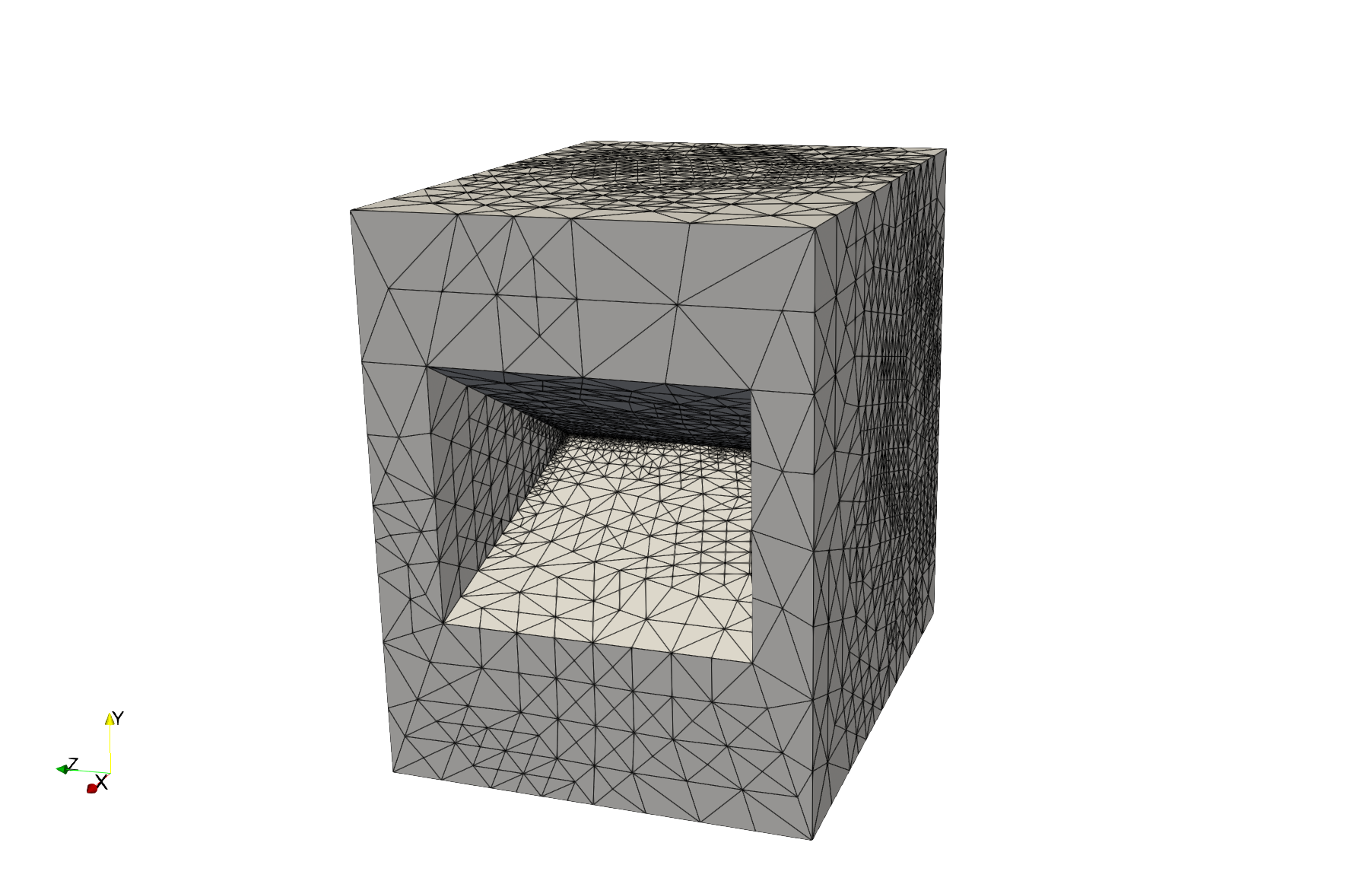}\\
		{$\mathrm{dof}$=1537422}
	\end{minipage}
	\caption{Test~\ref{sec:adaptive-cube-notch}. Initial and final adaptive meshes for the notched cube in different compressibility scenarios.}
	\label{fig:cube-notch-adapt-mesh}
\end{figure}}
The final adaptive meshes in Figure~\ref{fig:cube-notch-adapt-mesh} show a concentration of polyhedras near the reentrant edge generated by the notch. Away from this region, the mesh remains coarser, showing that the estimator identifies the local contributions to the error.

Finally, in Figure~\ref{fig:cube-notch-eigenfunctions}, we present the plots of the first eigenfunctions corresponding to $\nu=0.35$ and $\nu=0.50$. Note that the pseudostress magnitude $|\boldsymbol{\rho}_h|$ develops its largest values along the reentrant edge of the notch, where the geometric singularity induces strong stress concentrations. Similarly, we observe that the post-processed stress $|\boldsymbol{\sigma}_h|$ exhibits the same qualitative behavior. This singular mechanical behavior was properly detected by the proposed estimator. On the other hand, the displacement field $|\mathbf{u}_h|$ remains smooth over most of the domain while presenting localized high gradients near the singular edge. 

{\small\begin{figure}[!hbt]\centering
	\begin{minipage}{0.32\linewidth}\centering
		{ $|\brho_{h,1}|$, $\nu=0.35$}\\
		\includegraphics[scale=0.07,trim=16cm 0cm 19cm 5cm,clip]{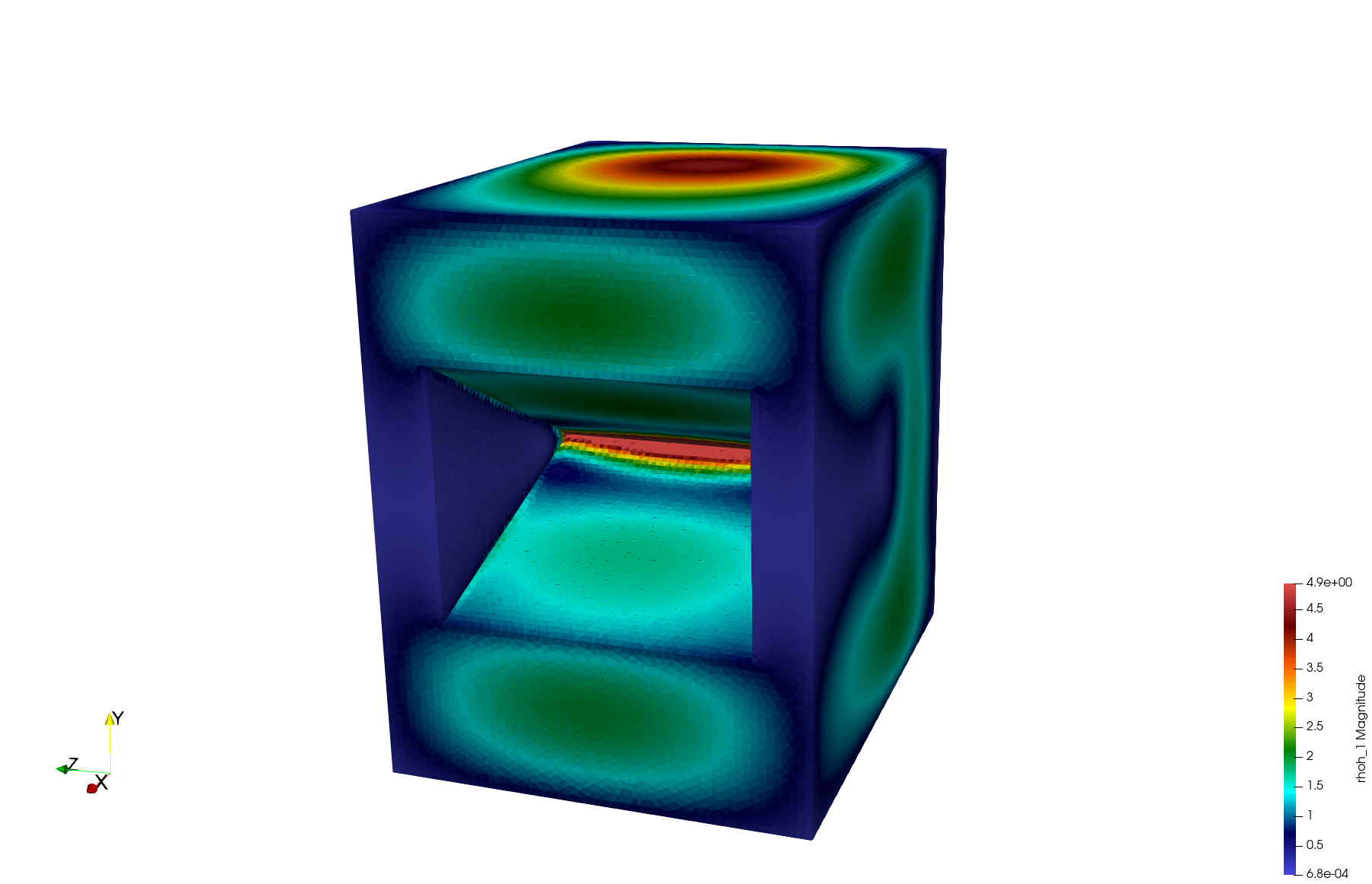}
	\end{minipage}
	\begin{minipage}{0.32\linewidth}\centering
		{ $|\bsig_{h,1}|$, $\nu=0.35$}\\
		\includegraphics[scale=0.07,trim=16cm 0cm 19cm 5cm,clip]{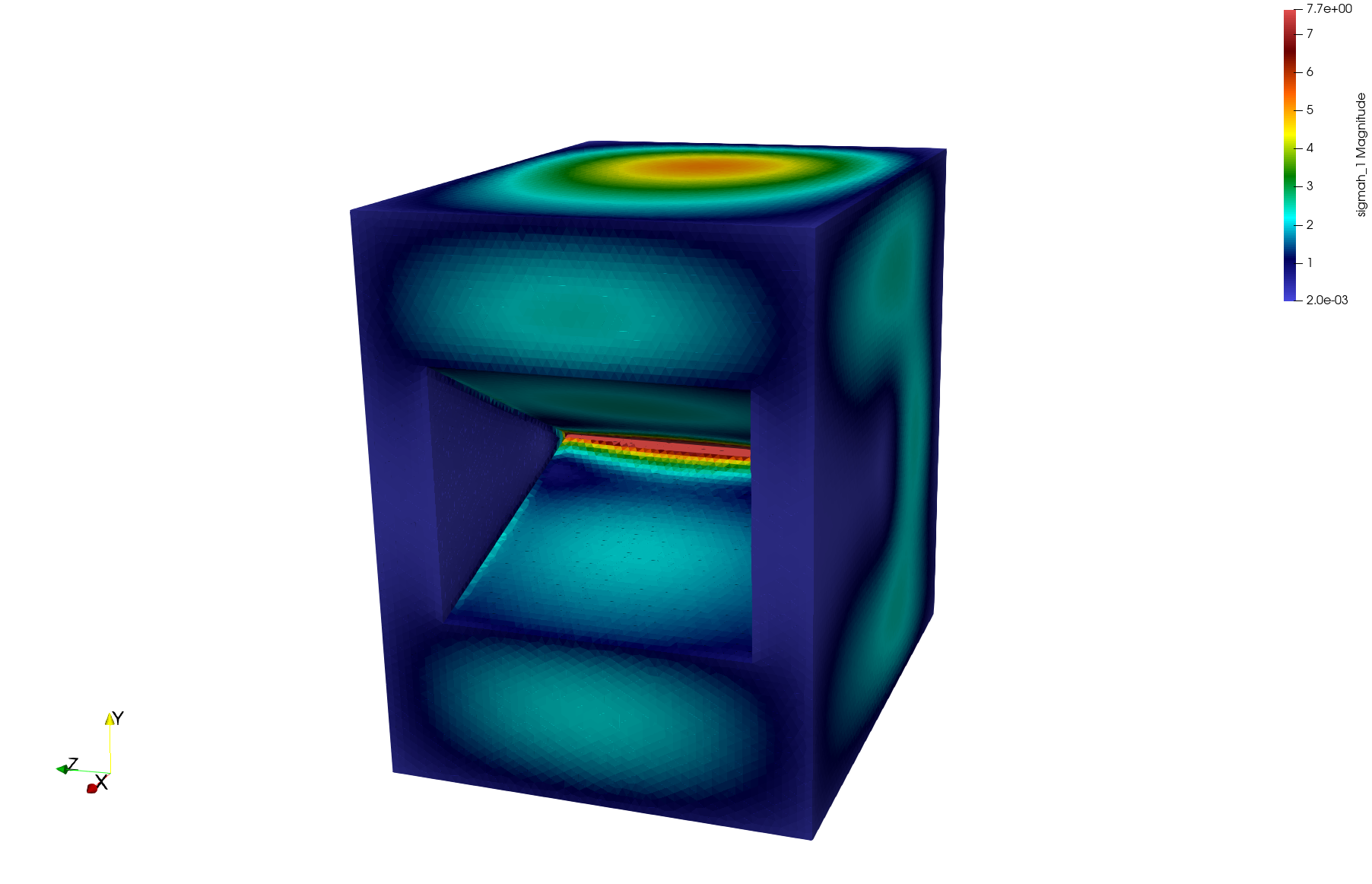}
	\end{minipage}
	\begin{minipage}{0.32\linewidth}\centering
		{ $|\bu_{h,1}|$, $\nu=0.35$}\\
		\includegraphics[scale=0.07,trim=16cm 0cm 19cm 5cm,clip]{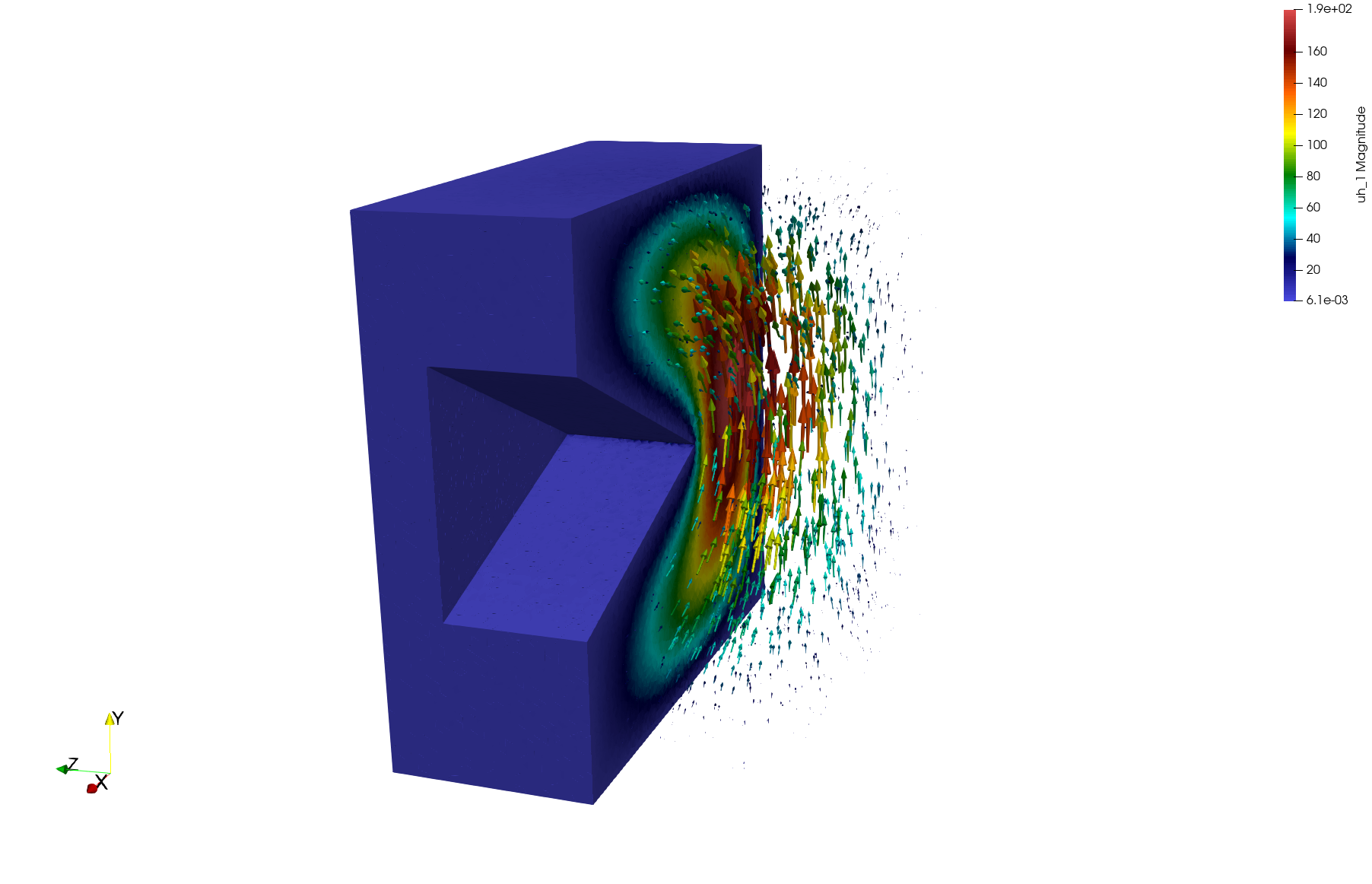}
	\end{minipage}\\
	\begin{minipage}{0.32\linewidth}\centering
		{ $|\brho_{h,1}|$, $\nu=0.50$}\\
		\includegraphics[scale=0.07,trim=16cm 0cm 19cm 5cm,clip]{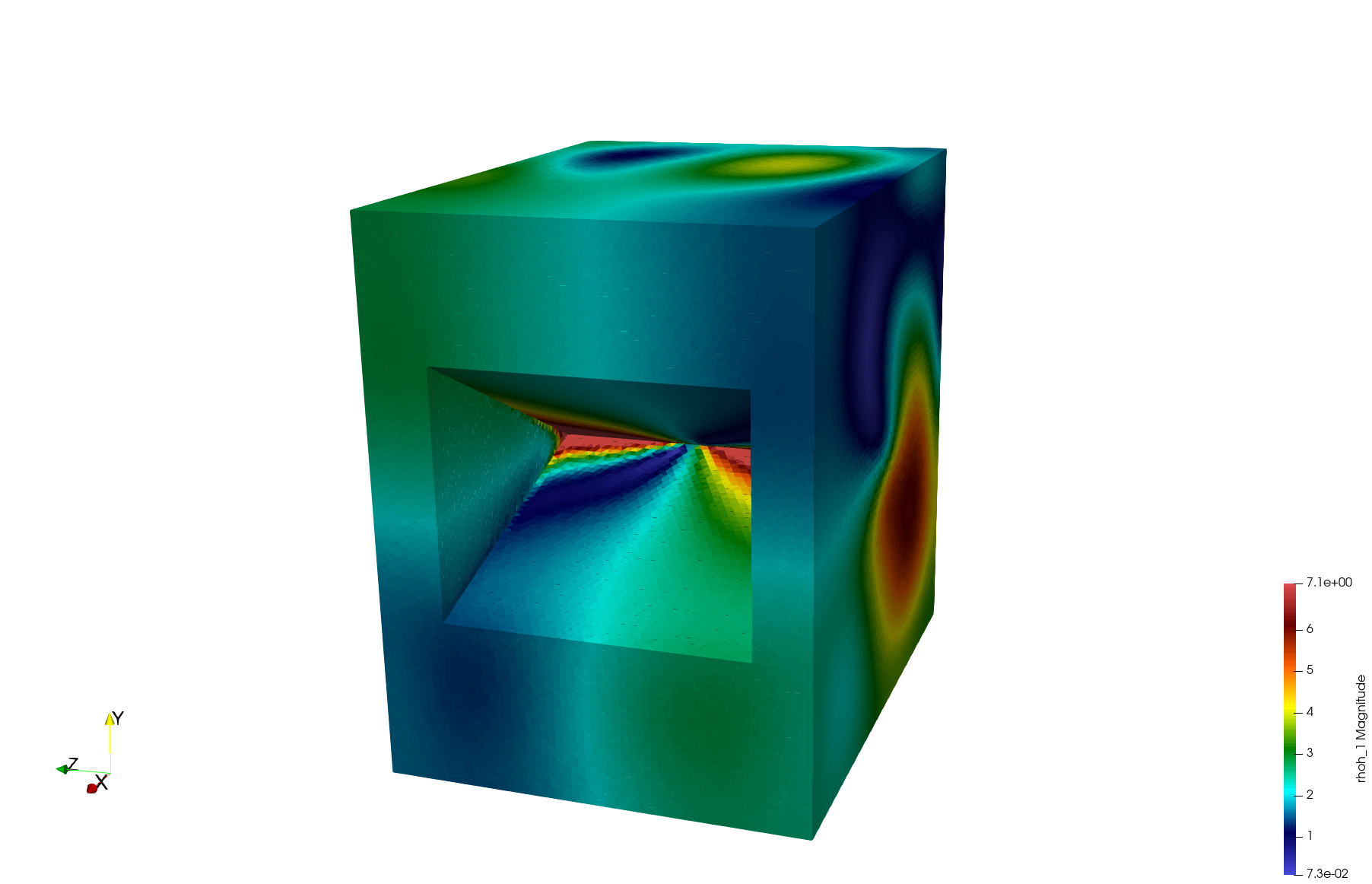}
	\end{minipage}
	\begin{minipage}{0.32\linewidth}\centering
		{ $|\bsig_{h,1}|$, $\nu=0.50$}\\
		\includegraphics[scale=0.07,trim=16cm 0cm 19cm 5cm,clip]{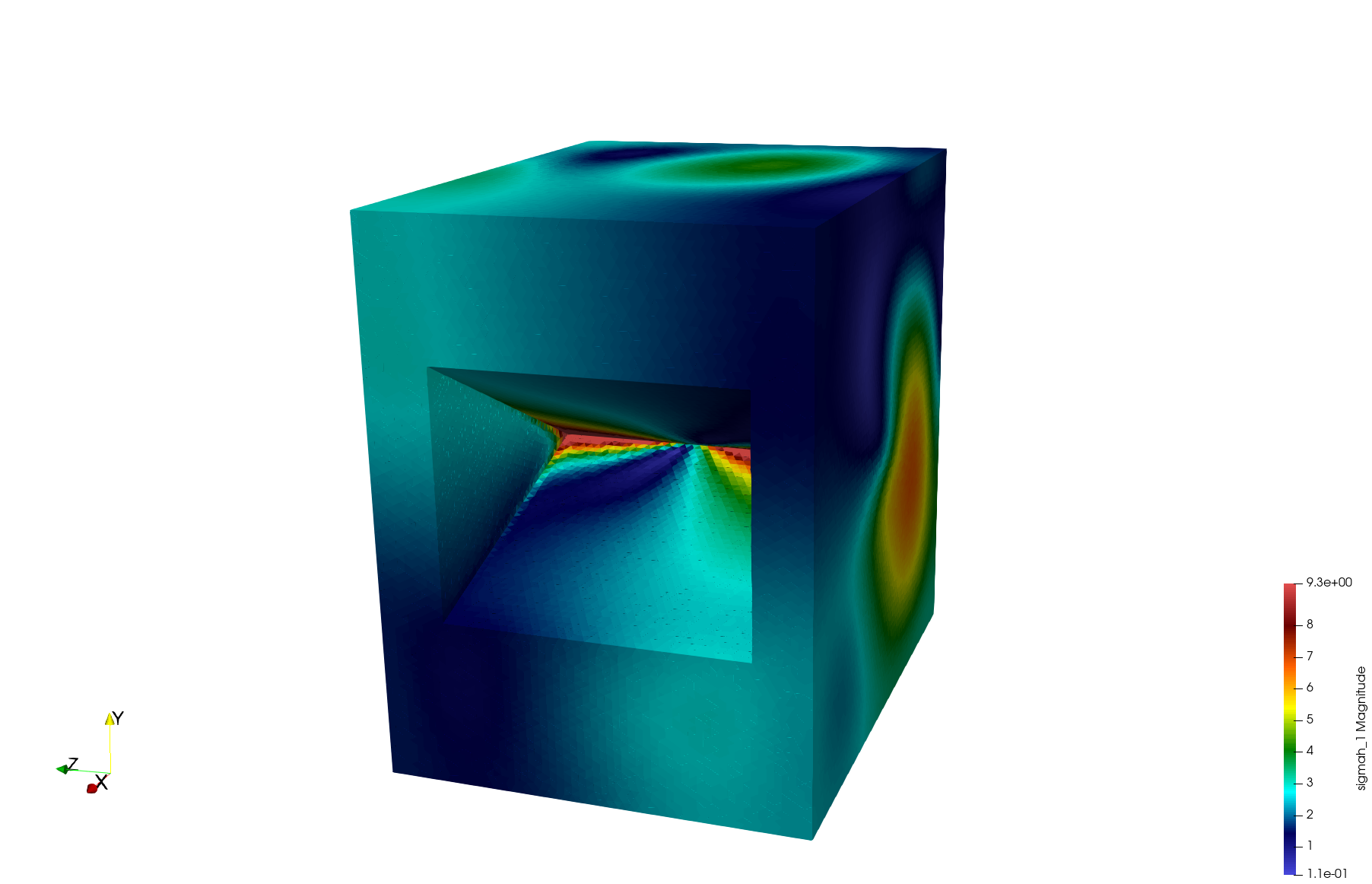}
	\end{minipage}
	\begin{minipage}{0.32\linewidth}\centering
		{ $|\bu_{h,1}|$, $\nu=0.50$}\\
		\includegraphics[scale=0.07,trim=16cm 0cm 19cm 5cm,clip]{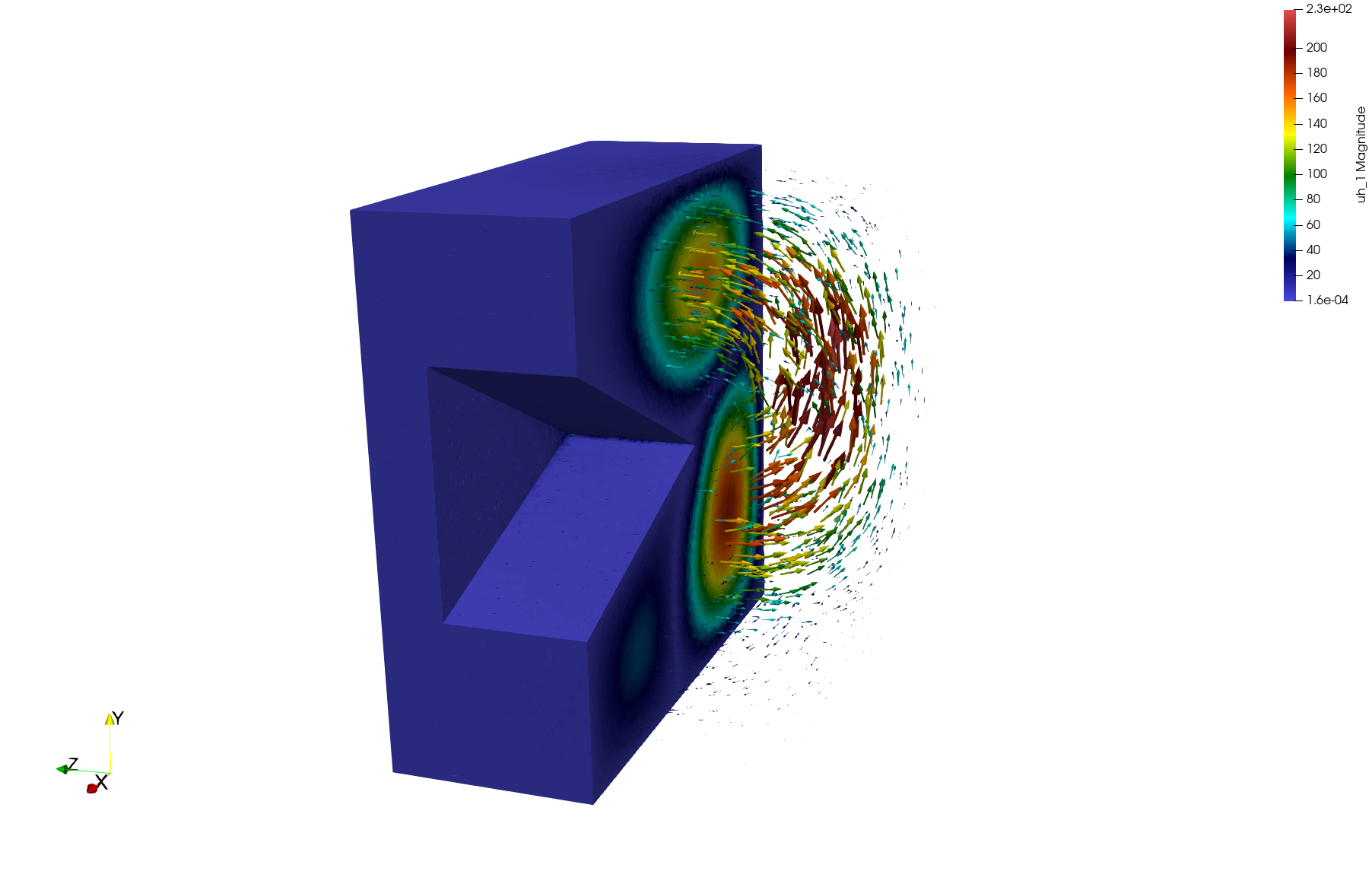}
	\end{minipage}
	\caption{Test~\ref{sec:adaptive-cube-notch}. Magnitudes of the pseudostress and the corresponding post-processed stress and displacement fields for the first eigenfrequency with different values of $\nu$.}
	\label{fig:cube-notch-eigenfunctions}
\end{figure}}

\bibliographystyle{siamplain}
\bibliography{oseen-eigenvalue_up}

@article {MR4279087,
    AUTHOR = {Bertrand, Fleurianne and Boffi, Daniele and Ma, Rui},
     TITLE = {An adaptive finite element scheme for the
              {H}ellinger-{R}eissner elasticity mixed eigenvalue problem},
   JOURNAL = {Comput. Methods Appl. Math.},
  FJOURNAL = {Computational Methods in Applied Mathematics},
    VOLUME = {21},
      YEAR = {2021},
    NUMBER = {3},
     PAGES = {501--512},
      ISSN = {1609-4840},
   MRCLASS = {65N25 (65N30 65N50 74B05)},
  MRNUMBER = {4279087},
       DOI = {10.1515/cmam-2020-0034},
}

@article{lepe2022posteriori,
  title={A posteriori analysis for a mixed FEM discretization of the linear elasticity spectral problem},
  author={Lepe, Felipe and Rivera, Gonzalo and Vellojin, Jesus},
  journal={Journal of Scientific Computing},
  volume={93},
  number={1},
  pages={10},
  year={2022},
  publisher={Springer},
  DOI = {https://doi.org/10.1007/s10915-022-01972-y},
}

@article {MR4396855,
    AUTHOR = {Meddahi, Salim},
     TITLE = {Variational eigenvalue approximation of non-coercive operators
              with application to mixed formulations in elasticity},
   JOURNAL = {SeMA J.},
  FJOURNAL = {SeMA Journal. Boletin de la Sociedad Espa\~{n}ola de Matem\'{a}tica
              Aplicada},
    VOLUME = {79},
      YEAR = {2022},
    NUMBER = {1},
     PAGES = {139--164},
      ISSN = {2254-3902},
   MRCLASS = {65N25 (65M15 65N12 65N30)},
  MRNUMBER = {4396855},
       DOI = {10.1007/s40324-021-00279-6},
}

@article {MR4570534,
    AUTHOR = {Inzunza, Daniel and Lepe, Felipe and Rivera, Gonzalo},
     TITLE = {Displacement-pseudostress formulation for the linear
              elasticity spectral problem},
   JOURNAL = {Numer. Methods Partial Differential Equations},
  FJOURNAL = {Numerical Methods for Partial Differential Equations. An
              International Journal},
    VOLUME = {39},
      YEAR = {2023},
    NUMBER = {3},
     PAGES = {1996--2017},
      ISSN = {0749-159X,1098-2426},
   MRCLASS = {65M60 (74B05)},
  MRNUMBER = {4570534},
       DOI = {10.1002/num.22955},

}

@article{LoMoRo_SIAM2010,
author = {Lovadina, Carlo and Mora, David and Rodr\'{\i}guez, Rodolfo},
title = {Approximation of the Buckling Problem for Reissner–Mindlin Plates},
journal = {SIAM Journal on Numerical Analysis},
volume = {48},
number = {2},
pages = {603-632},
year = {2010},
doi = {10.1137/090747336},
}

@article {MMR3,
    AUTHOR = {Meddahi, Salim and Mora, David and Rodr\'{\i}guez, Rodolfo},
     TITLE = {A finite element analysis of a pseudostress formulation for
              the {S}tokes eigenvalue problem},
   JOURNAL = {IMA J. Numer. Anal.},
  FJOURNAL = {IMA Journal of Numerical Analysis},
    VOLUME = {35},
      YEAR = {2015},
    NUMBER = {2},
     PAGES = {749--766},
      ISSN = {0272-4979},
   MRCLASS = {65N25 (65N15 65N30 76D07)},
  MRNUMBER = {3335223},
MRREVIEWER = {Jiguang Sun},
       DOI = {10.1093/imanum/dru006},

}

@book {MR0203473,
    AUTHOR = {Kato, Tosio},
     TITLE = {Perturbation theory for linear operators},
    SERIES = {Die Grundlehren der mathematischen Wissenschaften, Band 132},
 PUBLISHER = {Springer-Verlag New York, Inc., New York},
      YEAR = {1966},
     PAGES = {xix+592},
   MRCLASS = {47.00 (47.48)},
  MRNUMBER = {0203473},
MRREVIEWER = {L. de Branges},
}

@article {MR483401,
    AUTHOR = {Descloux, Jean and Nassif, Nabil and Rappaz, Jacques},
     TITLE = {On spectral approximation. {II}. {E}rror estimates for the
              {G}alerkin method},
   JOURNAL = {RAIRO Anal. Num\'{e}r.},
  FJOURNAL = {RAIRO Analyse Num\'{e}rique},
    VOLUME = {12},
      YEAR = {1978},
    NUMBER = {2},
     PAGES = {113--119, iii},
      ISSN = {0399-0516},
   MRCLASS = {65J05},
  MRNUMBER = {483401},
MRREVIEWER = {Juhani Pitk\"{a}ranta},
       DOI = {10.1051/m2an/1978120201131},
}

@article {MR483400,
    AUTHOR = {Descloux, Jean and Nassif, Nabil and Rappaz, Jacques},
     TITLE = {On spectral approximation. {I}. {T}he problem of convergence},
   JOURNAL = {RAIRO Anal. Num\'{e}r.},
  FJOURNAL = {RAIRO Analyse Num\'{e}rique},
    VOLUME = {12},
      YEAR = {1978},
    NUMBER = {2},
     PAGES = {97--112, iii},
      ISSN = {0399-0516},
   MRCLASS = {65J05},
  MRNUMBER = {483400},
MRREVIEWER = {Juhani Pitk\"{a}ranta},
       DOI = {10.1051/m2an/1978120200971},
}

@article {MR4542511,
    AUTHOR = {Meddahi, Salim},
     TITLE = {A {DG} method for a stress formulation of the elasticity
              eigenproblem with strongly imposed symmetry},
   JOURNAL = {Comput. Math. Appl.},
  FJOURNAL = {Computers \& Mathematics with Applications. An International
              Journal},
    VOLUME = {135},
      YEAR = {2023},
     PAGES = {19--30},
      ISSN = {0898-1221,1873-7668},
   MRCLASS = {65N25 (65N30 74Bxx)},
  MRNUMBER = {4542511},
MRREVIEWER = {Yu\ Zhang},
       DOI = {10.1016/j.camwa.2023.01.022},
}

@article {MR3860570,
    AUTHOR = {C\'aceres, Ernesto and Gatica, Gabriel N. and Sequeira,
              Fil\'ander A.},
     TITLE = {A mixed virtual element method for a pseudostress-based
              formulation of linear elasticity},
   JOURNAL = {Appl. Numer. Math.},
  FJOURNAL = {Applied Numerical Mathematics. An IMACS Journal},
    VOLUME = {135},
      YEAR = {2019},
     PAGES = {423--442},
      ISSN = {0168-9274,1873-5460},
   MRCLASS = {65N30 (65N15 74B05)},
  MRNUMBER = {3860570},
       DOI = {10.1016/j.apnum.2018.09.003},
}

@book {MR2050138,
    AUTHOR = {Ern, Alexandre and Guermond, Jean-Luc},
     TITLE = {Theory and practice of finite elements},
    SERIES = {Applied Mathematical Sciences},
    VOLUME = {159},
 PUBLISHER = {Springer-Verlag, New York},
      YEAR = {2004},
     PAGES = {xiv+524},
      ISBN = {0-387-20574-8},
   MRCLASS = {65-02 (65M60 65N30 74S05 76M10 78M10)},
  MRNUMBER = {2050138},
MRREVIEWER = {R.\ S.\ Anderssen},
       DOI = {10.1007/978-1-4757-4355-5},

}

@article{barrata2023dolfinx,
	title={{DOLFINx}: The next generation FEniCS problem solving environment},
	author={Barrata, Igor A and Dean, Joseph P and Dokken, J{\o}rgen S and HABERA, Michal and HALE, Jack and Richardson, Chris and Rognes, Marie E and Scroggs, Matthew W and Sime, Nathan and Wells, Garth N},
	year={2023},
	doi = {10.5281/zenodo.10447666},
}

@article{scroggs2022basix,
	title={Basix: a runtime finite element basis evaluation library},
	author={Scroggs, Matthew W and Baratta, Igor A and Richardson, Chris N and Wells, Garth N},
	journal={Journal of Open Source Software},
	volume={7},
	number={73},
	pages={3982},
	year={2022}
}

@book {MR3097958,
    AUTHOR = {Boffi, Daniele and Brezzi, Franco and Fortin, Michel},
     TITLE = {Mixed finite element methods and applications},
    SERIES = {Springer Series in Computational Mathematics},
    VOLUME = {44},
 PUBLISHER = {Springer, Heidelberg},
      YEAR = {2013},
     PAGES = {xiv+685},
      ISBN = {978-3-642-36518-8; 978-3-642-36519-5},
   MRCLASS = {65-02 (65M60 65N30)},
  MRNUMBER = {3097958},
MRREVIEWER = {Beny Neta},
       DOI = {10.1007/978-3-642-36519-5},
       OPTURL = {https://doi.org/10.1007/978-3-642-36519-5},
}

@article{hernandez2005slepc,
	title={{SLEPc}: A scalable and flexible toolkit for the solution of eigenvalue problems},
	author={Hernandez, Vicente and Roman, Jose E and Vidal, Vicente},
	journal={ACM Transactions on Mathematical Software (TOMS)},
	volume={31},
	number={3},
	pages={351--362},
	year={2005},
	publisher={ACM New York, NY, USA}
}

@book{verfuhrt1996,
	author = {Verfürth, Rüdiger},
	isbn = {978-0-471-96795-8},
	pages = {I-VI, 1-127},
	publisher = {Wiley},
	series = {Advances in numerical mathematics},
	title = {A review of a posteriori error estimation and adaptive mesh-refinement techniques},
	year = 1996
}

@book {MR1885308,
    AUTHOR = {Ainsworth, Mark and Oden, J. Tinsley},
     TITLE = {A posteriori error estimation in finite element analysis},
    SERIES = {Pure and Applied Mathematics (New York)},
 PUBLISHER = {Wiley-Interscience [John Wiley \& Sons], New York},
      YEAR = {2000},
     PAGES = {xx+240},
      ISBN = {0-471-29411-X},
   MRCLASS = {65-02 (65N15)},
  MRNUMBER = {1885308},
MRREVIEWER = {Ricardo G. Dur\'{a}n},
       DOI = {10.1002/9781118032824},

}

@book {MR3059294,
    AUTHOR = {Verf\"{u}rth, R\"{u}diger},
     TITLE = {A posteriori error estimation techniques for finite element
              methods},
    SERIES = {Numerical Mathematics and Scientific Computation},
 PUBLISHER = {Oxford University Press, Oxford},
      YEAR = {2013},
     PAGES = {xx+393},
      ISBN = {978-0-19-967942-3},
   MRCLASS = {65N30 (35J25 35K20 35Q30 35Q74 65N15)},
  MRNUMBER = {3059294},
MRREVIEWER = {Manfred Dobrowolski},
       DOI = {10.1093/acprof:oso/9780199679423.001.0001},
}

@article {MR3453481,
    AUTHOR = {Gatica, Gabriel N. and Gatica, Luis F. and Sequeira, Fil\'{a}nder
              A.},
     TITLE = {A priori and a posteriori error analyses of a
              pseudostress-based mixed formulation for linear elasticity},
   JOURNAL = {Comput. Math. Appl.},
  FJOURNAL = {Computers \& Mathematics with Applications. An International
              Journal},
    VOLUME = {71},
      YEAR = {2016},
    NUMBER = {2},
     PAGES = {585--614},
      ISSN = {0898-1221},
   MRCLASS = {65N30 (65N15 74B05)},
  MRNUMBER = {3453481},
MRREVIEWER = {Igor Bock},
       DOI = {10.1016/j.camwa.2015.12.009},
}

@article {MR3962898,
    AUTHOR = {Lepe, Felipe and Meddahi, Salim and Mora, David and Rodr\'{\i}guez,
              Rodolfo},
     TITLE = {Mixed discontinuous {G}alerkin approximation of the elasticity
              eigenproblem},
   JOURNAL = {Numer. Math.},
  FJOURNAL = {Numerische Mathematik},
    VOLUME = {142},
      YEAR = {2019},
    NUMBER = {3},
     PAGES = {749--786},
      ISSN = {0029-599X},
   MRCLASS = {65N30 (65N12 65N15 65N25 74B10)},
  MRNUMBER = {3962898},
MRREVIEWER = {Ioannis Doltsinis},
       DOI = {10.1007/s00211-019-01035-9},
       OPTURL = {https://doi.org/10.1007/s00211-019-01035-9},
}

@article{MR1115235,
    AUTHOR = {Babu\v{s}ka, I. and Osborn, J.},
     TITLE = {Handbook of numerical analysis. {V}ol. {II}},
    EDITOR = {Ciarlet, P. G. and Lions, J.-L.},
      NOTE = {Finite element methods. Part 1},
 PUBLISHER = {North-Holland, Amsterdam},
      YEAR = {1991},
     PAGES = {x+928},
      ISBN = {0-444-70365-9},
   MRCLASS = {65-00 (65-02 65M60 65N30)},
  MRNUMBER = {1115235},
MRREVIEWER = {Chris Arney},
}

@article {MR2009375,
    AUTHOR = {Hiptmair, R.},
     TITLE = {Finite elements in computational electromagnetism},
   JOURNAL = {Acta Numer.},
  FJOURNAL = {Acta Numerica},
    VOLUME = {11},
      YEAR = {2002},
     PAGES = {237--339},
      ISSN = {0962-4929},
   MRCLASS = {78M10 (65N30)},
  MRNUMBER = {2009375},
MRREVIEWER = {JiChun Li},
       DOI = {10.1017/S0962492902000041},
       OPTURL = {https://doi.org/10.1017/S0962492902000041},
}

@article {MR3036997,
    AUTHOR = {Meddahi, Salim and Mora, David and Rodr\'{\i}guez, Rodolfo},
     TITLE = {Finite element spectral analysis for the mixed formulation of
              the elasticity equations},
   JOURNAL = {SIAM J. Numer. Anal.},
  FJOURNAL = {SIAM Journal on Numerical Analysis},
    VOLUME = {51},
      YEAR = {2013},
    NUMBER = {2},
     PAGES = {1041--1063},
      ISSN = {0036-1429},
   MRCLASS = {65N30 (65N25 74B05 74S05)},
  MRNUMBER = {3036997},
MRREVIEWER = {Ignacio Romero},
       DOI = {10.1137/120863010},
}

@article {MR3376135,
    AUTHOR = {M\'{a}rquez, Antonio and Meddahi, Salim and Tran, Thanh},
     TITLE = {Analyses of mixed continuous and discontinuous {G}alerkin
              methods for the time harmonic elasticity problem with reduced
              symmetry},
   JOURNAL = {SIAM J. Sci. Comput.},
  FJOURNAL = {SIAM Journal on Scientific Computing},
    VOLUME = {37},
      YEAR = {2015},
    NUMBER = {4},
     PAGES = {A1909--A1933},
      ISSN = {1064-8275},
   MRCLASS = {65N30 (65N12 65N15 74B10)},
  MRNUMBER = {3376135},
MRREVIEWER = {Igor Bock},
       DOI = {10.1137/14099022X},
}
\end{document}